\documentclass[11pt]{article}

\usepackage[a4paper,
    left=3.2cm,
    right=3.4cm,
    top=3cm,
    bottom=3cm]{geometry}

\usepackage{amsmath, amssymb, amsthm}
\usepackage{xcolor}
\usepackage{xspace}
\usepackage{hyperref}
\usepackage{comment}
\usepackage[numbers,sort&compress]{natbib}

\bibpunct[, ]{[}{]}{,}{n}{,}{,}

\theoremstyle{plain}
\newtheorem{theorem}{Theorem}[section]
\newtheorem{lemma}[theorem]{Lemma}
\newtheorem{corollary}[theorem]{Corollary}

\theoremstyle{definition}
\newtheorem{definition}[theorem]{Definition}

\theoremstyle{remark}
\newtheorem{remark}[theorem]{Remark}

\newcommand{\dx}{\,\mathrm{d}x}
\newcommand{\dt}{\,\mathrm{d}t}
\newcommand{\ds}{\,\mathrm{d}s}
\newcommand{\Pb}{\mbox{\rm (P)}\xspace}
\newcommand{\PbT}{\mbox{\rm (P$_T$)}\xspace}

\newcommand{\Pbr}{\mbox{\rm (P$_\rho$)}\xspace}
\newcommand{\PbTr}{\mbox{\rm (P$_{T,\rho}$)}\xspace}

\newcommand{\Uad}{\mathcal{U}_{ad}}
\newcommand{\UTad}{U_{T,ad}}
\newcommand{\K}{{K_{ad}}}
\newcommand{\proj}{\operatorname{Proj}}
\newcommand{\e}{{\rm e}}

\begin{document}
\date{}

\title{Infinite Horizon Optimal Control Problems with Discount Factors}

\author{Eduardo Casas\thanks{Departamento de Matem\'{a}tica Aplicada y Ciencias de la Computaci\'{o}n, E.T.S.I. Industriales y de Telecomunicaci\'on, Universidad de Cantabria, 39005 Santander, Spain (eduardo.casas@unican.es).}
    \and Karl Kunisch\thanks{Institute for Mathematics and Scientific Computing, University of Graz, Heinrichstrasse 36, A-8010 Graz, Austria,  and Radon Institute, Austrian Academy of Sciences, Linz, Austria (karl.kunisch@uni-graz.at).}}

%

\maketitle

\begin{abstract}
This paper is dedicated to the analysis of infinite horizon optimal control problems subject to semilinear parabolic equations with constraints on the controls and discounted cost functionals. The discount factors on the cost and the state components are allowed to differ from each other. First-order as well as second-order optimality conditions are derived and the importance of allowing different discount factors for the second-order analysis for the class of nonlinearities under consideration is demonstrated. Finally convergence and rate of convergence for the approximation of the infinite horizon problem by a family of finite horizon problems is proven.
\end{abstract}



\pagestyle{myheadings} \thispagestyle{plain} \markboth{E.~CASAS AND K.~KUNISCH}{Infinite Horizon Optimal Control Problems}


\section{Introduction}
\label{S1}

In this paper we analyze the following optimal control problem
\[
   \Pb \  \min_{u \in \Uad} J(u) = \frac{1}{2}\int_0^\infty[\e^{-{\sigma_s}t}\|y_u(t) - y_d(t)\|^2_{L^2(\Omega)} + \nu\e^{-{\sigma_c}t}\|u(t)\|^2_{L^2(\omega)}]\dt,
\]
where $y_u$ denotes the solution of the following parabolic equation:
\begin{equation}
\left\{
\begin{array}{l}
\displaystyle
\frac{\partial y}{\partial t} + Ay + f(y) = g + u\chi_\omega\ \mbox{ in } Q = \Omega \times (0,\infty),\\[0.5ex] \partial_{n_A}y = 0 \ \mbox{ on } \Sigma = \Gamma \times (0,\infty),\   y(0) = y_0 \ \mbox{ in } \Omega.
\end{array}
\right.
\label{E1.1}
\end{equation}
$\Omega$ is a bounded domain in $\mathbb{R}^n$, $1 \le n \le 3$, with a Lipschitz boundary $\Gamma$, $\omega$ is a subdomain of $\Omega$, $\chi_\omega$ denotes the characteristic function of $\omega$, and $y_0 \in L^\infty(\Omega)$. The regularity of the functions $g$ and $y_d$ will be specify in sections \ref{S3} and \ref{S4}, respectively. We also assume that $\nu > 0$ and $\Uad$ is defined by
\[
\Uad = \{u \in L^\infty(0,\infty;L^2(\omega)) : u(t) \in \K \ \text{for a.a. } t \in (0,\infty)\}.
\]
Above $\K$ denotes a closed, convex, and bounded subset of $L^2(\omega)$. Possible choices for $\K$ include
\begin{align}
&\K = B_\gamma = \{v \in L^2(\omega) : \|v\|_{L^2(\omega)} \le \gamma\}, \ 0 < \gamma < \infty,\label{E1.2}\\
&\K = \{v \in L^2(\omega) : \alpha \le v(x) \le \beta \text{ for a.a. } x \in \omega\},\ -\infty < \alpha < \beta < \infty.\label{E1.3}
\end{align}
The exponents $\sigma_s > 0$ and $\sigma_c > 0$ in the cost functional are the  discount factors. In this work,  at  first,  they are chosen independent from each other. Introducing discount factors relaxes the behavior of the tracking error and the control costs as $t\to \infty$, and this effect is increased with the magnitudes of $\sigma_s$ and $\sigma_c$.  Discount factors not only change the nature of the effect of the cost functional on the optimal solutions to $\Pb$, but they also increase the class of nonlinearities $f$ which render $\Pb$ to be well-posed. An alternative to enlarge the class of nonlinearities $f$ is provided by making stabilizability assumptions, as was discussed in \cite{Casas-Kunisch2026}.  We refer to \cite{Rodrigues2020}, for instance, for a general theory on the stabilizability of  semilinear parabolic equations. We recall that there is a distinction between global and local results with respect to the choice of initial conditions. Here we consider  the global case, since $y_0$ will not assumed to lie in an apriori fixed bounded set. We will argue the existence of a solution to $\Pb$ as soon as  $\sigma_s$ is larger than some positive threshold.  For the first- and second-order necessary optimality conditions $\sigma_s$ will need to be chosen in relation to the nonlinear function $f$. The second-order sufficient condition also requires $\sigma_c$ to be strictly smaller than $\sigma_s$; see Theorem \ref{T5.2} for the precise condition.
This condition is necessary to establish the rate of convergence of the family of finite horizon problems posed on the interval $(0,T)$ to the infinite horizon problem $\Pb$ as $T\to \infty$.

We now comment on the importance of infinite horizon problems. First, it is a natural setting for stabilization and asymptotic tracking problems. Further, even if the infinite horizon is not in the focus of formulating an optimal control problem, the specific choice of a finite horizon problem can be accompanied by ambiguity. In this case the formulation as infinite horizon problem together with a convergence analysis of finite horizon problems to the infinite horizon problem can be a viable option. From the practical point of view, the numerical realization of infinite horizon problem presents difficulties. Here, the receding horizon control approach can be useful.

Let us briefly point to selected publications dealing with infinite horizon optimal control problems.
Possibly the first work is that of  Halkin, see \cite{halkin1974}.
In the monograph \cite{CHL1991} the importance of the infinite time horizon for problems in mathematical biology and in economics is stressed. More recent contributions, involving discounted formulations with $\sigma_s=\sigma_c$ or abstract time dependent running costs,  are given in e.g.
\cite{AK2004, AV2012, Basco2024, BCF2018}, where ordinary differential equations are analyzed. Partial differential equations have received much less attention. One chapter is reserved for them both in \cite[Chapter 9]{CHL1991} and \cite[Chapter III.6]{JL1971}. In \cite{BKP2019}, bilinear optimal control problems are investigated.
In these works, no discount factors are involved. The case of partial differential equations with discount factor was treated in \cite{Casas-Kunisch2023B} and the current paper is an extension of that work.

The current paper is a continuation of our efforts in the investigation of infinite horizon optimal control problems. We now comment on this earlier work, first pointing to papers without the use of discounted cost. In \cite{Casas-Kunisch2017} and \cite{Casas-Kunisch2022A} we focused on the long term behavior of the sparsifying  effect of an $L^1$-in time term on the control in the cost. While in \cite{Casas-Kunisch2017} the nonlinearities  are of polynomial type with odd degree, a more general class of nonlinearities is considered in \cite{Casas-Kunisch2022A}.  These papers do not include a second order analysis of the optimal control problems. In \cite{Casas-Kunisch2023C} we take advantage of  $L^\infty$ estimates on the state variables  which allow a successful second-order analysis. In these three papers, the assumption that $f'(s) \ge 0$ in a neighborhood of 0, along with coercivity of the operator $A + f'(0)I$, was necessary to prove the well-posedness of the linearized state equation. Later, in \cite{Casas-Kunisch2026}, this assumption was removed by including stabilizability and detectability conditions which also allowed us to generalize the class of nonlinearities. In \cite{Casas-Kunisch2023B} we started our investigations of optimal control of semilinear equations with discount cost factor. There, the discount factor only affected the state, and the nonlinearities were of a polynomial type with an odd degree. In the present paper, we dispense with this assumption, allowing for a much wider class of nonlinearities. We consider two independent discount factors for the state and control, which may differ. The inclusion of a discount factor for the controls can produce associated states that are not bounded, which leads to technical difficulties in the analysis of the equation and linearized equation. New ideas have been developed to address this issue. We then investigate the interplay between these discount factors. Here, we introduce a new assumption on the nonlinear function $f$; see \eqref{E3.1}. This assumption allows us to prove the well-posedness of the linearized state equation and is satisfied by many nonlinearities, particularly polynomial or exponential functions.

 We should also mention that  infinite horizon  optimal control is related  to the optimal control formulation of stabilization problems. In the latter case, no discounted cost is involved and the focus of that research is distinct from the optimization theoretic one.

The structure of this paper is as follows. Section 2 provides an analysis of the state equation, paying particular attention to the weights introduced by the discount factor. Section 3 focuses on the differentiability properties of the control-to-state mapping, while Section 4 focuses on the differentiability properties of the cost functional and first-order optimality conditions. The second-order analysis of $\Pb$ is given in Section 5 and in Section 6 we present the rate of convergence analysis of a family of finite horizon problems to the infinite horizon problem. The conclusions section summarizes the conditions imposed on the discount factors. This summary is helpful for understanding the paper and for modeling the cost functional.

\subsection{General assumptions on the problem data}

For the nonlinear term $f:\mathbb{R} \longrightarrow \mathbb{R}$ we assume that it is of class $C^2$ and satisfies:
\begin{equation}
\exists \Lambda_f \le 0  \text{ such that } f'(s) \ge \Lambda_f\ \forall s \in \mathbb{R} \ \text{ and }\ f(0) = 0.
\label{E1.4}
\end{equation}

The assumption $f(0) = 0$ is made for simplicity of presentation.  If $f$ does not vanish at 0, we can replace $f$ and $g$ with $f-f(0)$ and $g-f(0)$ respectively, and the above assumption will still hold. We observe that any polynomial $f$ of odd degree with a positive main coefficient satisfies the condition $f'(s) \ge \Lambda_f$. Any non-decreasing monotone function $f$ of class $C^2$ also satisfies this condition. Functions of the form $f(s) = \e^{\zeta(s)}$, where $\zeta$ is an odd degree polynomial with a positive main coefficient, are included in the above formulation.

Regarding the linear operator $A$ we assume that
\[
Ay = -\sum_{i, j = 1}^n\partial_{x_j}[a_{ij}\partial_{x_i}y] + a_0y,
\]
with $a_{i,j}, a_0 \in L^\infty(\Omega)$ and $a_{ij} = a_{ji}$ satisfying for almost all $x \in \Omega$
\begin{equation}
a_0(x) \ge 0\ \text{ and }\ \exists \Lambda_A > 0 \text{ such that }\Lambda_A|\xi|^2 \le \sum_{i, j = 1}^na_{ij}(x)\xi_i\xi_j \ \forall \xi \in \mathbb{R}^n.
\label{E1.7}
\end{equation}
The assumption $a_0(x) \ge 0$ is made by convenience. Indeed, we notice that if $a_0 \not\ge 0$, we can replace $a_0$ and $f$ with $a_0(x) + \|a_0\|_{L^\infty(\Omega)}$ and $f(s) - \|a_0\|_{L^\infty(\Omega)}s$, respectively.

We define the continuous bilinear mapping $a:H^1(\Omega) \times H^1(\Omega) \longrightarrow \mathbb{R}$ by
\begin{equation}
a(y,z) = \int_\Omega\Big[\sum_{i, j = 1}^na_{ij}(x)\partial_{x_i}y(x)\partial_{x_j}z(x) + a_0(x)y(x)z(x)\Big]\dx.
\label{E1.8}
\end{equation}
From \eqref{E1.7} and $a_{i,j}, a_0 \in L^\infty(\Omega)$ we infer the existence of a constant $M_A$ such that
\begin{equation}
\Lambda_A\|\nabla y\|^2_{L^2(\Omega)^n} \le a(y,y) \le M_A\|y\|^2_{H^1(\Omega)}\quad \forall y \in H^1(\Omega).
\label{E1.9}
\end{equation}

Associated with $A$ we introduce the boundary operator:
\[
\partial_{n_A}y(x,t) = \sum_{i, j = 1}^na_{ij}(x)\partial_{x_i}y(x,t)n_j(x),
\]
where $\mathbf{n}(x) = (n_j(x))_{j = 1}^n$ denotes the unit outward normal vector to $\Gamma$ at the point $x$. The reader is referred to \cite[page 511]{Dautray-Lions2000} for a rigorous definition of the normal derivative in a trace sense with $\partial_{n_A} y(\cdot,t) \in H^{-\frac{1}{2}}(\Gamma)$ for almost all $t \in (0,T)$.

\section{Analysis of the state equation}
\label{S2}

In this section we analyze the following equation
\begin{equation}
\left\{
\begin{array}{l}
\displaystyle
\frac{\partial y}{\partial t} + Ay + f(y) = h\ \mbox{ in } Q,\\[0.5ex] \partial_{n_A}y = 0 \ \mbox{ on } \Sigma,\   y(0) = y_0 \ \mbox{ in } \Omega,
\end{array}
\right.
\label{E2.1}
\end{equation}
where $A$, $f$, and $y_0$ satisfy the assumptions established in the section \ref{S1}. Let us introduce some notation. Given a Banach space $X$, for every $\lambda \in \mathbb{R}$ and $r \in [1,\infty]$, we denote by $L^r_\lambda(0,\infty;X)$ the Banach space of measurable functions $y:(0,\infty) \longrightarrow X$ such that $\|y\|_{L^r_\lambda(0,\infty;X)} < \infty$, where
\begin{align*}
&\|y\|_{L^r_\lambda(0,\infty;X)} = \Big(\int_0^\infty\e^{-\lambda t}\|y(t)\|^r_X\dt\Big)^{\frac{1}{r}}\ \text{ for } 1 \le r < \infty,\\
&\|y\|_{L^\infty_\lambda(0,\infty;X)} = \text{ess sup}_{t \in (0,\infty)}\e^{-\frac{\lambda}{2}t}\|y(t)\|_X.
\end{align*}
Further $C_\lambda[0,\infty;X)$ denotes the Banach space of continuous functions $y:[0,\infty) \longrightarrow X$ such that $\|y\|_{C_\lambda([0,\infty;X)} < \infty$, where
\[
\|y\|_{C_\lambda[0,\infty;X)} = \sup_{t \in [0,\infty)}\e^{-\frac{\lambda}{2}t}\|y(t)\|_X.
\]
If $X = L^2(B)$, where $B \subset \mathbb{R}^n$ is a Lebesgue measurable set, we sometimes write $L^2_\lambda(B \times (0,\infty))$ instead of $L^2_\lambda(0,\infty;L^2(B))$ and for $X = C(B)$ we set $C_\lambda(B \times (0,\infty)) = C_\lambda[0,\infty;C(B))$, where $C(B)$ denotes the Banach space of continuous and bounded real functions on $B$.

We also define the Hilbert space
\[
W_\lambda(0,\infty) = \{y \in L^2_\lambda(0,\infty;H^1(\Omega)) : \frac{\partial y}{\partial t} \in L^2_\lambda(0,\infty;H^1(\Omega)^*)\}\]
 endowed with the norm
\[
\|y\|_{W_\lambda(0,\infty)} = \Big(\|y\|^2_{L^2_\lambda(0,\infty;H^1(\Omega))} + \Big\|\frac{\partial y}{\partial t}\Big\|^2_{L^2_\lambda(0,\infty;H^1(\Omega)^*)}\Big)^{\frac{1}{2}}.
\]

If $\lambda = 0$, we simply remove the subindex $\lambda$ as usual. It is immediate that $y \in W_\lambda(0,\infty)$ if and only if $\e^{-\frac{\lambda}{2}t}y \in W(0,\infty)$. Using this fact and the embedding $W(0,\infty) \subset C[0,\infty;L^2(\Omega))$ we infer the continuity of the inclusion $W_\lambda(0,\infty) \subset C_\lambda[0,\infty;L^2(\Omega))$.
For every $\lambda \in \mathbb{R}$ we denote the space
\[
Y_\lambda = C_\lambda[0,\infty;L^2(\Omega)) \cap L^2_\lambda(0,\infty;H^1(\Omega)),
\]
endowed with the norm $\|y\|_{Y_\lambda} = \|y\|_{C_\lambda[0,\infty;L^2(\Omega))} + \|y\|_{L^2_\lambda(0,\infty;H^1(\Omega))}$. Thus, $Y_\lambda$ is a Banach space.
In the sequel $\langle\cdot,\cdot\rangle$ denotes the duality pairing between $H^1(\Omega)^*$ and $H^1(\Omega)$. In the sequel $\langle\cdot,\cdot\rangle$ denotes the duality paring between $H^1(\Omega)^*$ and $H^1(\Omega)$. Along this paper we fix $p$ as follows
\[
\frac{4}{4 - n} < p \le 6\text{ if } n = 2 \text{ or } 3 \text{ and } 2 \le p \le 6 \text{ if } n = 1.
\]

The following lemma is crucial in many proofs of this paper.

\begin{lemma}
Let $X$ be a Banach space and $\lambda, \hat\lambda \in \mathbb{R}$ with $\lambda < \hat\lambda$. Assume that $\{u_k\}_{k = 1}^\infty$ is a bounded sequence in $L^r_\lambda(0,\infty;X)$ for some $r \in [1,\infty]$. If $u \in L^r(0,\infty;X)$ and $\lim_{k \to \infty}\|u_k - u\|_{L^r(0,T;X)} = 0$ for each $T \in (0,\infty)$, then $\lim_{k \to \infty}\|u_k - u\|_{L^r_{\hat\lambda}(0,\infty;X)} = 0$ holds.
\label{L2.1}
\end{lemma}

\begin{proof}
Let $C = \sup_{1 \le k}\|u_k-u\|_{L^r_\lambda(0,\infty;X)}$. Given $\varepsilon > 0$ we select $T_{\varepsilon} < \infty$ such that $C\e^{\frac{\lambda - \hat\lambda}{r}T_\varepsilon} < \varepsilon$ if $r < \infty$. For $r = \infty$ we assume that $C\e^{\frac{\lambda - \hat\lambda}{2}T_\varepsilon} < \varepsilon$. Then, we have
\begin{align*}
&\text{if } r < \infty, \text{then } \|u_k - u\|_{L^r_{\hat\lambda}(T_\varepsilon,\infty;X)} \le \e^{\frac{\lambda - \hat\lambda}{r}T_\varepsilon}\|u_k - u\|_{L^r_\lambda(T_\varepsilon,\infty;X)} \le C\e^{\frac{\lambda - \hat\lambda}{r}T_\varepsilon} < \varepsilon,\\
&\text{if } r = \infty, \text{then } \|u_k - u\|_{L^r_{\hat\lambda}(T_\varepsilon,\infty;X)} \le \e^{\frac{\lambda - \hat\lambda}{2}T_\varepsilon}\|u_k - u\|_{L^r_\lambda(T_\varepsilon,\infty;X)} \le C\e^{\frac{\lambda - \hat\lambda}{2}T_\varepsilon} < \varepsilon.
\end{align*}
The assumptions of the lemma and these inequalities yield
\[
\limsup_{k \to \infty}\|u_k - u\|_{L^r_{\hat\lambda}(0,\infty;X)} \le \limsup_{k \to \infty}\|u_k - u\|_{L^r_{\hat\lambda}(0,T_\varepsilon;X)} + \limsup_{k \to \infty}\|u_k - u\|_{L^r_{\hat\lambda}(T_\varepsilon,\infty;X)} \le \varepsilon.
\]
Since $\varepsilon$ can be arbitrarily small, the claim of the lemma follows.
\end{proof}

\begin{definition}
We call $y$ a solution to \eqref{E2.1} if $y \in W(0,T)$ and $f(y) \in L^2(Q_T)$ for every $T \in  (0,\infty)$, where $Q_T = \Omega \times (0,T)$, and $y$ satisfies the following equation in the usual variational sense
\begin{equation}
\left\{
\begin{array}{l}
\displaystyle
\frac{\partial y}{\partial t} + Ay + f(y) = h \
\mbox{ in } Q_T,\\[0.5ex]
\partial_{n_A}y = 0 \ \mbox{ on } \Sigma_T = \Gamma \times (0,T),\ \ y(0) = y_0 \ \mbox{ in } \Omega.
\end{array}
\right.
\label{E2.2}
\end{equation}
\label{D2.1}
\end{definition}
See, for instance, \cite[pages 136--137]{Lad-Sol-Ura68} or \cite[page 108]{Showalter1997} for the definition of a variational solution (or generalized solution) of \eqref{E2.2}.

The next theorem establishes the existence, uniqueness, and some regularity of the solution.

\begin{theorem}
Assume that $-2\Lambda_f < \lambda < \infty$ and $h \in L^2_\lambda(Q)$. Then, Equation \eqref{E2.1} has a unique solution $y \in W_\lambda(0,\infty)$ with $f(y) \in L^2_\lambda(Q)$, and the following estimates hold
\begin{align}
&\|y\|_{C_\lambda[0,\infty;L^2(\Omega))} + \sqrt{\min\{\frac{\lambda}{2} + \Lambda_f,2\Lambda_A\}}\|y\|_{L^2_\lambda(0,\infty;H^1(\Omega))}\notag\\
& \le 2\|y_0\|_{L^2(\Omega)} + \frac{2\sqrt{2}}{\sqrt{\lambda + 2\Lambda_f}}\|h\|_{L^2_\lambda(Q)},\label{E2.3}\\
&\|f(y)\|_{L^2_\lambda(Q)} \le C_{f,\lambda,\Omega}\Big(\|y_0\|_{L^\infty(\Omega)} + \|h\|_{L^2_\lambda(Q)}\Big),\label{E2.4}\\
&\|y\|_{W_\lambda(0,\infty)} \le M_{f,\lambda,\Omega}\Big(\|y_0\|_{L^\infty(\Omega)} + \|h\|_{L^2_\lambda(Q)}\Big),\label{E2.5}\\
&\lim_{t \to \infty}\e^{-\frac{\lambda}{2} t}\|y(t)\|_{L^2(\Omega)} = 0.\label{E2.6}
\end{align}
Moreover, if $h \in L^p_\lambda(0,\infty;L^2(\Omega))$, then we have
\begin{equation}
\|y\|_{L^\infty_\lambda(Q)} \le K_{f,\lambda,\Omega}\Big(\|y_0\|_{L^\infty(\Omega)} + \|h\|_{L^p_\lambda(0,\infty;L^2(\Omega))}\Big).\label{E2.7}
\end{equation}
For each $\hat\lambda > -2\Lambda_f$,  the constants $C_{f,\lambda,\Omega}$, $M_{f,\lambda,\Omega}$, and $K_{f,\lambda,\Omega}$ are uniformly bounded for $\lambda \ge \hat\lambda$.

\label{T2.1}
\end{theorem}

\begin{proof}
The existence and uniqueness of a solution follows from \cite{CW2023}. The proof of the regularity of $y$ and the corresponding estimates are split in five steps.\vspace{2mm}

{\em Step I: $y \in Y_\lambda$.} Given $T \in (0,\infty)$ arbitrary, testing the equation \eqref{E2.2} with $\e^{-\lambda t}y$ and integrating in $Q_T$ we obtain:
\begin{align}
&\frac{1}{2}\int_0^T\e^{-\lambda t}\frac{d}{dt}\|y(t)\|^2_{L^2(\Omega)}\dt + \int_0^T\e^{-\lambda t}a(y(t),y(t))\dt + \int_{Q_T}\e^{-\lambda t}f(y)y\dx\dt\notag\\
&= \int_{Q_T}\e^{-\lambda t}hy\dx\dt \le \|h\|_{L^2_\lambda(Q)}\|y\|_{L^2_\lambda(Q)}.\label{E2.8}
\end{align}
Integrating by parts we get
\[
\int_0^T\e^{-\lambda t}\frac{d}{dt}\|y(t)\|^2_{L^2(\Omega)}\dt = \e^{-\lambda T}\|y(T)\|^2_{L^2(\Omega)} - \|y_0\|^2_{L^2(\Omega)} + \lambda\int_0^T\e^{-\lambda t}\|y(t)\|^2_{L^2(\Omega)}\dt.
\]
Inserting this identity in \eqref{E2.8}, using \eqref{E1.4} and \eqref{E1.9}, and applying the mean value theorem to the function $f$ we obtain
\begin{align}
&\frac{1}{2}\e^{-\lambda T}\|y(T)\|^2_{L^2(\Omega)} + \Big(\frac{\lambda}{2} + \Lambda_f\Big)\int_0^T\e^{-\lambda t}\|y(t)\|^2_{L^2(\Omega)}\dt + \Lambda_A\int_0^T\e^{-\lambda t}\|\nabla y(t)\|^2_{L^2(\Omega)^n}\dt\notag\\
&\le \frac{1}{2}\|y_0\|^2_{L^2(\Omega)} + \|h\|_{L^2_\lambda(Q)}\|y\|_{L^2_\lambda(Q_T)}\notag\\
&  \le \frac{1}{2}\|y_0\|^2_{L^2(\Omega)} + \frac{1}{\lambda + 2\Lambda_f}\|h\|^2_{L^2_\lambda(Q)} + \frac{1}{2}\Big(\frac{\lambda}{2} + \Lambda_f\Big)\|y\|^2_{L^2_\lambda(Q_T)}.
\label{E2.9}
\end{align}
From this inequality we infer
\begin{align*}
&\e^{-\frac{\lambda}{2}T}\|y(T)\|_{L^2(\Omega)} + \sqrt{\min\{\frac{\lambda}{2} + \Lambda_f,2\Lambda_A\}}\|y\|_{L^2_\lambda(0,T;H^1(\Omega))}\\
& \le \|y_0\|_{L^2(\Omega)} + \frac{\sqrt{2}}{\sqrt{\lambda + 2\Lambda_f}}\|h\|_{L^2_\lambda(Q)} .
\end{align*}
Since $T$ was taking arbitrarily large, this inequality implies \eqref{E2.3}.\vspace{2mm}

{\em Step II: $f(y) \in L^2_\lambda(Q)$.} For every integer $k \ge 1$ we define the functions $f_k(s) = \proj_{[-k,+k]}(f(s))$. Since $f_k$ is a bounded and Lipschitz function we deduce from {\em Step I} that $f_k(y) \in L^2(0,T;H^1(\Omega))$ for every $T < \infty$. Testing the equation \eqref{E2.2} with $\e^{-\lambda t}f_k(y)$ and integrating in $Q_T$ we obtain:
\begin{align}
&\int_0^T\e^{-\lambda t}\langle\frac{\partial y}{\partial t},f_k(y)\rangle \dt + \int_0^T\e^{-\lambda t}a(y(t),f_k(y(t)))\dt + \int_{Q_T}\e^{-\lambda t}f(y(t))f_k(y(t))\dx\dt\notag\\
&= \int_{Q_T}\e^{-\lambda t}hf_k(y)\dx\dt \le \|h\|_{L^2_\lambda(Q)}\Big(\int_0^T\e^{-\lambda t}\|f_k(y(t))\|^2_{L^2(\Omega)}\dt\Big)^{\frac{1}{2}}\notag\\
&\le \frac{1}{2}\|h\|^2_{L^2_\lambda(Q)} + \frac{1}{2}\int_0^T\e^{-\lambda t}\|f_k(y(t))\|^2_{L^2(\Omega)}\dt.\label{E2.10}
\end{align}

We consider the function $F_k:\mathbb{R} \longrightarrow \mathbb{R}$, defined by $F_k(\rho) = \int_0^\rho f_k(s)\ds$. Using the mean value theorem and \eqref{E1.4} we get
\begin{equation}
F_k(\rho) = \int_0^\rho f_k'(\theta_s s)s\ds \ge \frac{\Lambda_f}{2}\rho^2.
\label{E2.11}
\end{equation}
Now we have
\begin{align}
&\int_0^T\e^{-\lambda t}\langle\frac{\partial y}{\partial t},f_k(y)\rangle \dt = \int_0^T\e^{-\lambda t}\frac{d}{dt}\int_\Omega F_k(y(t))\dx\dt = \e^{-\lambda T}\int_\Omega F_k(y(x,T))\dx\notag\\
& - \int_\Omega F_k(y_0(x))\dx + \lambda\int_0^T\e^{-\lambda t}\int_\Omega F_k(y(x,t))\dx\dt.\label{E2.12}
\end{align}
From \eqref{E2.11} we deduce that $\int_\Omega F_k(y(x,t))\dx \ge \frac{\Lambda_f}{2}\|y(t)\|^2_{L^2(\Omega)}$ for all $t \in (0,T]$. We also have
\[
\int_\Omega |F_k(y_0(x))|\dx \le |\Omega|\sup_{|s| \le \|y_0\|_{L^\infty(\Omega)}}|f(s)|\|y_0\|_{L^\infty(\Omega)} = |\Omega|C_{f,y_0}\|y_0\|_{L^\infty(\Omega)}.
\]
Inserting the last two inequalities in \eqref{E2.12} we infer
\begin{align}
&\int_0^T\e^{-\lambda t}\langle\frac{\partial y}{\partial t},f_k(y)\rangle \dt \notag\\
& \ge \e^{-\lambda T}\frac{\Lambda_f}{2}\|y(T)\|^2_{L^2(\Omega)} + \frac{\lambda\Lambda_f}{2}\int_0^T\e^{-\lambda t}\|y(t)\|^2_{L^2(\Omega)}\dt - |\Omega|C_{f,y_0}\|y_0\|_{L^\infty(\Omega)}\notag\\
&\ge \frac{\Lambda_f}{2}\Big(\|y\|^2_{C_\lambda[0,\infty;L^2(\Omega)} + \lambda\|y\|^2_{L^2_\lambda(0,\infty;L^2(\Omega))}\Big)  - |\Omega|C_{f,y_0}\|y_0\|_{L^\infty(\Omega)}.\label{E2.13}
\end{align}
Using \eqref{E1.4} and \eqref{E1.9} we get with the mean value theorem
\begin{align*}
&\int_0^T\e^{-\lambda t}a(y(t),f_k(y(t)))\dt = \int_0^T\e^{-\lambda t}\int_\Omega\Big[f_k'(y)\sum_{i,j = 1}^na_{ij}\partial_{x_i}y\partial_{x_j}y + a_0f_k(y)y\Big]\dx\dt\\
& \ge \Lambda_f\int_0^\infty\e^{-\lambda t}a(y(t),y(t))\dt.
\end{align*}
From the last inequality, \eqref{E2.10}, \eqref{E2.13}, \eqref{E2.3}, and  $f^2_k(y) \le f(y)f_k(y)$ we infer that
\[
\int_0^T\e^{-\lambda t}\|f_k(y)\|^2_{L^2(\Omega)}\dt \le C_{f,\lambda,\Omega}\Big(\|y_0\|_{L^\infty(\Omega)} + \|h\|_{L^2_\lambda(Q)}\Big)\quad \forall T < \infty.
\]
This implies that the equality is also valid for $T = \infty$. Finally, taking the limit as $k \to \infty$ and applying the monotone convergence theorem we get \eqref{E2.4}.\vspace{2mm}

{\em Step III: $y \in W_\lambda(0,\infty)$.} We infer from the partial differential equation \eqref{E2.1} and the Steps I and II that
\[
\e^{-\frac{\lambda}{2}t}\frac{\partial y}{\partial t} = \e^{-\frac{\lambda}{2}t}(h - f(y)) - A(\e^{-\frac{\lambda}{2}t}y) \in L^2(0,\infty;H^1(\Omega)^*).
\]
Then, we get that
\[
\frac{\partial(\e^{-\frac{\lambda}{2}t}y)}{\partial t} = \e^{-\frac{\lambda}{2}t}\frac{\partial y}{\partial t} - \lambda\e^{-\frac{\lambda}{2}t}y \in L^2(0,\infty;H^1(\Omega)^*).
\]
Estimate \eqref{E2.5} follows from \eqref{E2.3}-\eqref{E2.4} and the above identities.\vspace{2mm}

{\em Step IV: $\lim_{t \to \infty}\e^{-\frac{\lambda}{2} t}\|y(t)\|_{L^2(\Omega)} = 0$.} Since $y \in W_\lambda(0,\infty)$, then we have that $\e^{-\frac{\lambda}{2}t}y \in W(0,\infty)$. As any element in this space it satisfies $\lim_{t \to \infty}\|\e^{-\frac{\lambda}{2}t}y(t)\|_{L^2(\Omega)} = 0$; see \cite[Theorem 2.4]{Casas-Kunisch2022A}.

{\em Step V: $y \in L^\infty_\lambda(Q)$.} The proof follows the steps of \cite[Theorem A.2]{Casas-Kunisch2023C}. We simply indicate how to integrate the term $\e^{-\lambda t}$ in the proof for the cases $n = 2$ or 3, proceeding analogously for $n = 1$. First, we take $\rho \ge \|y_0\|_{L^\infty(\Omega)}$ and set $y_\rho(x,t) = \e^{-\frac{\lambda}{2}t}y(x,t) - \proj_{[-\rho,+\rho]}(\e^{-\frac{\lambda}{2}t}y(x,t))$ and $A_\rho(t) = \{x \in \Omega : \e^{-\frac{\lambda}{2}t}|y(x,t)| > \rho\}$. Following the notation in \cite[Theorem A.2]{Casas-Kunisch2023C}, we take $\alpha \in \big(\frac{pn}{2p-4},\frac{n}{n - 2}\big)$. We have that $\e^{-\frac{\lambda}{2}t}a(y,y_\rho) \ge a(y_\rho,y_\rho)$ and
\[
\e^{-\frac{\lambda}{2}t}\Big[\frac{\partial y}{\partial t}y_\rho + f(y)y_\rho\Big]= \frac{\partial(\e^{-\frac{\lambda}{2}t}y)}{\partial t}y_\rho + [\frac{\lambda}{2} + f'(\theta y)]\e^{-\frac{\lambda}{2}t}yy_\rho \ge  \frac{\partial y_\rho}{\partial t}y_\rho + [\frac{\lambda}{2} + \Lambda_f]y^2_\rho.
\]
Since $y$ satisfies \eqref{E2.2} for every $T < \infty$, we have that $y \in H^1(\Omega \times (T_0,T))$ for any $0 < T_0 < T < \infty$; see \cite[Corollary 2.4]{Showalter1997}. Hence, the above relations make sense pointwise almost everywhere.

Testing equation \eqref{E2.2} with $\e^{-\lambda t}y_\rho$ and using the above identities, \eqref{E1.9}, and integrating in $[0,t]$ for any $t \in (0,T]$ and using that $\frac{p}{2} > 1$ we get wtih $B = \|h\|_{L_\lambda^p(0,\infty;L^2(\Omega))}$
\begin{align*}
&\frac{1}{2}\|y_\rho(t)\|^2_{L^2(\Omega)} + [\frac{\lambda}{2} + \Lambda_f]\int_{Q_t}y^2_\rho\dx\ds + \Lambda_A\int_{Q_t}|\nabla y_\rho|^2\dx\ds \le \int_{Q_t}\e^{-\frac{\lambda}{2}s}hy_\rho\dx\ds\\
&\le B\|y_\rho\|_{L^{p'}(0,t;L^2(\Omega))}\ds \le \Big(\int_0^t\|y_\rho(s)\|^{p'}_{L^2(A_\rho(s))}\ds\Big)^{\frac{1}{p'}}\\
& \le B\Big(\int_0^t\|y_\rho(s)\|^{p'}_{L^{2\alpha}(A_\rho(s))}|A_\rho(s)|^{\frac{p'}{2\alpha'}}\ds\Big)^{\frac{1}{p'}}\\
&\le C_1B\Big(\int_0^t\|y_\rho(s)\|^2_{H^1(\Omega)}\ds\Big)^{\frac{1}{2}}\Big(\int_0^t|A_\rho(s)|^{\frac{p'}{\alpha'(2 - p')}}\ds\Big)^{\frac{2 - p'}{2p'}}\\
&\le \frac{\min\{\frac{\lambda}{2} + \Lambda_f,\Lambda_A\}}{2}\int_0^t\|y_\rho(s)\|^2_{H^1(\Omega)}\ds\\
& + \frac{C^2_1B^2}{\min\{\lambda + 2\Lambda_f,2\Lambda_A\}}\Big(\int_0^t|A_\rho(s)|^{\frac{p'}{\alpha'(2 - p')}}\ds\Big)^{\frac{2 - p'}{p'}},\end{align*}
where $C_1$ is the norm of the continuous emebedding $H^1(\Omega) \subset L^{2\alpha}(\Omega)$. This leads to
\[
\Big(\|y_\rho(t)\|^2_{L^2(\Omega)} + \|y_\rho\|^2_{L^2(0,t;H^1(\Omega))}\Big)^{\frac{1}{2}} \le CB\Big(\int_0^t|A_\rho(s)|^{\frac{p}{\alpha'(p - 2)}}\ds\Big)^{\frac{p - 2}{2p}}
\]
for some constant $C$ independent of $t$ and $\rho$. This establishes the inequality (A.8) of \cite[Theorem A.2]{Casas-Kunisch2023C} and the rest of the proof is the same to infer \eqref{E2.7}.
\end{proof}

We finish this section by studying the continuity of the solution of \eqref{E2.1} with respect to the right hand side $h$.

\begin{remark}
It is immediate that $L^2_{\lambda_1}(Q) \subset L^2_{\lambda_2}(Q)$ if $\lambda_1 < \lambda_2$ and $\|y\|_{L^2_{\lambda_2}(Q)} \le \|y\|_{L^1_{\lambda_2}(Q)}$. Consequently, if $h \in L^2_{\lambda_0}(Q)$ for $\lambda_0 \le -2\Lambda_f$, then $h \in L^2_\lambda(Q)$ for every $\lambda > -2\Lambda_f$ and by Theorem \ref{T2.1} we infer that the solution $y$ of \eqref{E2.1} belongs to $W_\lambda(0,\infty)$ and $f(y) \in L^2_\lambda(Q)$ for every $\lambda > -2\Lambda_f$. Moreover, the estimates \eqref{E2.3}--\eqref{E2.6} hold. The same comment applies to spaces $L^p_\lambda(0,\infty;L^2(\Omega))$ and estimate \eqref{E2.7}.
\label{R2.1}
\end{remark}

\begin{lemma}
Let us assume that $\hat\lambda > \lambda > -2\Lambda_f$ and $h_k \rightharpoonup h$ in $L^2_\lambda(Q)$. Let $y_k$ and $y$ be the solutions of \eqref{E2.1} associated with $h_k$ and $h$, respectively. Then, we have $y_k \rightharpoonup y$ in $Y_\lambda$, $f(y_k) \rightharpoonup f(y)$ in $L^2_\lambda(Q)$, and $y_k \to y$ in $Y_{\hat\lambda}$. In addition, if $h_k \rightharpoonup h$ in $L^p_\lambda(0,\infty;L^2(\Omega))$, then $\lim_{k \to \infty}\|y_k - y\|_{L^\infty_{\hat\lambda}(Q)} = 0$ and $f(y_k) \to f(y)$ in $L^2_{\hat\lambda}(Q)$.
\label{L2.2}
\end{lemma}

\begin{proof}
The estimates \eqref{E2.4} and \eqref{E2.5} yield the existence of a constant $C_\lambda$ such that $\|y_k\|_{W_\lambda(0,\infty)} + \|f(y_k)\|_{L_\lambda^2(Q)} \le C_\lambda$ for all $k \ge 1$. Then, there exists a subsequence, denoted in the same way, such that $y_k \rightharpoonup \tilde y$ in $W_\lambda(0,\infty)$ and $f(y_k) \rightharpoonup \phi$ in $L^2_\lambda(Q)$. This implies that $y_k \rightharpoonup \tilde y$ in $W(0,T)$ for every $T < \infty$. Using the compactness of the embedding $W(0,T) \subset L^2(Q_T)$ we deduce the existence of a second subsequence, denoted in the same way, such that $y_k \to \tilde y$ strongly in $L^2(Q_T)$ and $y_k(x,t) \to \tilde y(x,t)$ a.e.~in $Q_T$. Therefore, we also have that $f(y_k(x,t)) \to f(\tilde y(x,t))$ almost everywhere in $Q_T$. Since $T$ can be selected arbitrarily large, we infer that $\phi = f(\tilde y)$ and hence $f(y_k) \rightharpoonup f(\tilde y)$ in $L^2_\lambda(Q)$. Using these convergence properties, it is  immediate to pass  to the weak limit in the equation satisfied by $y_k$ in $Q_T$ and to deduce that $\tilde y = y$ in $Q_T$ for every $T < \infty$. Consequently the whole sequences $\{y_k\}_{k = 1}^\infty$ and $\{f(y_k)\}_{k = 1}^\infty$ converge strongly to $y$ and weakly to $f(y)$ in $L^2(Q_T)$ for every $T < \infty$. Now, applying Lemma \ref{L2.1} with $X = L^2(\Omega)$ and $r = 2$ the convergence $y_k \to y$ in $L^2_{\hat\lambda}(Q)$ follows.

Subtracting the equations \eqref{E2.2} satisfied by $(y_k,h_k)$ and $(y,h)$ and testing the resulting equation with $\e^{-\hat\lambda t}(y_k - y)$, we obtain
\begin{align}
&\int_0^T\e^{-\hat\lambda t}\langle\frac{\partial (y_k - y)}{\partial t},y_k - y\rangle \dt + \int_0^T\e^{-\hat\lambda t}a(y_k-y,y_k-y)\dt\notag\\
&  = \int_{Q_T}\e^{-\hat\lambda t}[(h_k - h) -(f(y_k) - f(y))](y_k - y)\dx\dt\notag\\
& \le \Big(\|h_k - h\|_{L^2_{\hat\lambda}(Q)} + \|f(y_k) - f(y)\|_{L^2_{\hat\lambda}(Q)}\Big)\|y_k - y\|_{L^2_{\hat\lambda}(Q)}.
\label{E2.17}
\end{align}
Let us take a look at the left hand side of the above inequality:
\begin{align*}
&\int_0^T\e^{-\hat\lambda t}\langle\frac{\partial (y_k - y)}{\partial t},y_k - y\rangle \dt = \frac{1}{2}\int_0^T\e^{-\hat\lambda t}\frac{d}{dt}\|y_k(t) - y(t)\|^2_{L^2(\Omega)}\dt\\
&= \frac{1}{2}\e^{-\hat\lambda T}\|y_k(T) - y(T)\|^2_{L^2(\Omega)} + \frac{\hat\lambda}{2}\int_0^T\e^{-\hat\lambda t}\|y_k(t) - y(t)\|^2_{L^2(\Omega)}\dt,\\
&\Lambda_A\int_0^T\e^{-\hat\lambda t}\|\nabla(y_k(t) - y(t))\|^2_{L^2(\Omega)^n}\dt \le \int_0^T\e^{-\hat\lambda t}a(y_k-y,y_k-y)\dt.
\end{align*}
Using these relations in \eqref{E2.17} and taking into account that $T < \infty$ is arbitrarily large we infer
\begin{align*}
&\|y_k - y\|^2_{C_{\hat\lambda}[0,\infty;L^2(\Omega))} + \|y_k - y\|^2_{L^2_{\hat\lambda}(0,\infty;H^1(\Omega))}\\
& \le C\Big(\|h_k - h\|_{L^2_{\hat\lambda}(Q)} + \|f(y_k) - f(y)\|_{L^2_{\hat\lambda}(Q)}\Big)\|y_k - y\|_{L^2_{\hat\lambda}(Q)} \stackrel{k \to \infty}{\longrightarrow} 0.
\end{align*}

Now, we assume that $h_k \rightharpoonup h$ in $L^p_\lambda(0,\infty;L^2(\Omega))$. From \eqref{E2.7} we infer the existence of a constant $M_\lambda$ such that $\|y_k\|_{L^\infty_\lambda(Q)} \le M_\lambda$ for every $k \ge 1$ and $\|y\|_{L^\infty_\lambda(Q)} \le M_\lambda$. We prove that $\|y_k - y\|_{L^\infty_{\hat\lambda}(Q_T)} \to 0$ for every $T < \infty$. We introduce $w_k = \e^{-\frac{\hat\lambda}{2}t}(y_k - y)$, that satisfies the equation
\[
\left\{\begin{array}{l}\displaystyle\frac{\partial w_k}{\partial t} + \frac{\hat\lambda}{2}w_k + Aw_k + f'(y_{\theta_k})w_k = \e^{-\frac{\hat\lambda}{2}t}(h_k - h) \text{ in } Q_T,\\\partial_{n_A}w_k = 0 \text{ on } \Sigma_T,\ \ w_k(0) = 0 \text{ in }\Omega,\end{array}\right.
\]
where $y_{\theta_k} = y + \theta_k(y_k - y)$ for a measurable function $\theta_k:Q_T \longrightarrow [0,1]$. Since the right hand side of the equation is bounded in $L^p(0,T;L^2(\omega))$, then there exists a number $\mu \in (0,1)$ such that $\{w_k\}_{k = 1}^\infty$ is bounded in $C^{0,\mu}(\bar Q_T)$. Using the compactness of the embedding $C^{0,\mu}(\bar Q_T) \subset C(\bar Q_T)$ and the convergence $y_k \to y$ in $L^2(Q_T)$ imply $\lim_{k \to \infty}\|y_k - y\|_{C(\bar Q_T)} = 0$. Again, applying Lemma \ref{L2.1} with $X = L^\infty(\Omega)$ and $r = \infty$ we get $\lim_{k \to \infty}\|y_k - y\|_{L^\infty_{\hat\lambda}(Q)} = 0$.

Since $\lim_{k \to \infty}\|y_k - y\|_{L^\infty_{\hat\lambda}(Q)} = 0$, then $\lim_{k \to \infty}\|y_k - y\|_{L^\infty(Q_T)} = 0$ and $\lim_{k \to \infty}\|f(y_k) - f(y)\|_{L^\infty(Q_T)} = 0$ for all $T < \infty$. Thus, we have $\lim_{k \to \infty}\|f(y_k) - f(y)\|_{L^2_{\hat\lambda}(Q_T)} = 0$ for every $T < \infty$. Once again using Lemma \ref{L2.1} with $X = L^2(\Omega)$ and $r = 2$ we conclude the convergence $f(y_k) \to f(y)$ in $L^2_{\hat\lambda}(Q)$.
\end{proof}

\begin{remark}
Turning to equation \eqref{E1.1}, suppose that $\lambda > -2\Lambda_f$ and $g \in L^2_\lambda(Q)$. As a consequence of Lemma \ref{L2.2} we have the following property: if $u_k \stackrel{*}{\rightharpoonup} u$ in $L^\infty(0,\infty;L^2(\omega))$, then
    \begin{equation}
    \lim_{k \to \infty}\Big(\|y_{u_k} - y_u\|_{Y_\lambda} + \|y_{u_k} - y_u\|_{L^\infty_\lambda(Q)} + \|f(y_{u_k}) - f(y_u)\|_{L^2_\lambda(Q)}\Big) = 0.
    \label{E2.18}
    \end{equation}
    Indeed, it is enough to observe that $L^\infty(0,\infty;L^2(\omega))$ is continuously embedded in $L^p_\lambda(0,\infty;L^2(\omega))$ for every $\lambda > 0$.
    \label{R2.2}
\end{remark}

\section{Differentiability of the input-to-state mapping}
\label{S3}

The goal of this section is to analyze the differentiability of the relation control-to-state. To this end we make the following additional assumptions on the control problem:

{\em Assumption:} We assume the existence of $q \in [0,\infty)$ and constants $C_{f,1}, C_{f,2} \ge 0$ such that
\begin{align}
&|f^{(j)}(s)| \le C_{f,1}(|s|^q + 1)(C_{f,2}|f(s)| + 1)\quad \forall s \in \mathbb{R} \text{ and } j =1, 2,
\label{E3.1}\\
&\sigma_s > -2(q+3)\Lambda_f, \; -2\Lambda_f < \lambda_c < \frac{\sigma_s}{q + 3},  \;\bar\lambda := (q + 1)\lambda_c, \sigma_c > 0, \label{E3.2}\\
&g \in L^p_{\lambda_c}(0,\infty;L^2(\Omega)) \text{ and } y_d \in L^p_{\sigma_s}(0,\infty;L^2(\Omega)).
\label{E3.3}
\end{align}

Observe that any polynomial function of degree $m$ satisfies \eqref{E3.1} with $q = m - 1$ and $C_{f,2} = 0$ and also with $q = 0$ and $C_{f,2} = 1$, taking $C_{f,1} > 0$ large enough in both cases. For any polynomial function $\zeta$ of degree $m \ge 1$, the function $f(s) = \e^{\zeta(s)}$ also satisfies \eqref{E3.1} with $q = 2m-2$ and $C_{f,2} = 1$.

We also notice that $\lambda_c < \frac{\sigma_s}{q + 3}$ is equivalent to $\lambda_c < \frac{\sigma_s - \bar\lambda}{2}$.

\begin{lemma}
The following properties hold:

\noindent i) For every $u \in L^p_{\lambda_c}(0,\infty;L^2(\omega))$ the solution $y_u$ of \eqref{E1.1} satisfies $f'(y_u), f''(y_u) \in L^2_{\bar\lambda}(Q)$ and there exists a constant $C$ independent of $(y_0,g,u)$ such that
\begin{equation}
\|f^{(j)}(y_u)\|_{L^2_{\bar\lambda}(Q)} \le C\Big(\|y_0\|_{L^\infty(\Omega)} + \|g + u\chi_\omega\|_{L^p_{\lambda_c}(0,\infty;L^2(\Omega))} + 1\Big)^{q + 1},\ j=1, 2.
\label{E3.4}
\end{equation}

\noindent ii) For $u, v \in L^p_{\lambda_c}(0,\infty;L^2(\omega))$ and $M = \max\{\|u\|_{L^p_{\lambda_c}(0,\infty;L^2(\omega))}, \|v\|_{L^p_{\lambda_c}(0,\infty;L^2(\omega))}\}$, let $y = y_u + \theta(y_v - y_u)$, where $\theta:Q \longrightarrow [0,1]$ is a measurable function. Then, we have $f(y), f'(y), f''(y) \in L^2_{\bar\lambda}(Q)$ and there exists a constant $C_M$ such that
\begin{equation}
\max\Big\{\|f(y)\|_{L^2_{\bar\lambda}(Q)}, \|f'(y)\|_{L^2_{\bar\lambda}(Q)}, \|f''(y)\|_{L^2_{\bar\lambda}(Q)}\Big\} \le C_M.
\label{E3.5}
\end{equation}
\label{L3.1}
\end{lemma}

\begin{proof}
i) Since we have $\bar\lambda = (q+1){\lambda_c}$ we get
\[
\e^{-\bar\lambda t}|y_u(x,t)|^{2q}|f(y_u(x,t))|^2 \le \Big(\e^{-\frac{{\lambda_c}}{2}t}|y_u(x,t)|\Big)^{2q}\e^{-{\lambda_c} t}|f(y_u(x,t))|^2.
\]
Then, using \eqref{E2.7}, \eqref{E2.4}, and Remark \ref{R2.1} we get
\begin{align*}
&\||y_u|^qf(y_u)\|_{L^2_{\bar\lambda}(Q)} \le \|y_u\|^q_{L^\infty_{\lambda_c}(Q)}\|f(y_u)\|_{L^2_{\lambda_c}(Q)}\\
& \le K^q_{f,\lambda_c,\Omega}C_{f,\lambda_c,\Omega}\Big(\|y_0\|_{L^\infty(\Omega)} + \|g + u\chi_\omega\|_{L^{p}_{\lambda_c}(Q)}\Big)^{q + 1}.
\end{align*}
Using this estimate and \eqref{E3.1} we deduce \eqref{E3.4}.

ii) Let us define the function $\tilde f:\mathbb{R} \longrightarrow \mathbb{R}$ by $\tilde f(s) = f(s) - \Lambda_fs$. Then, $\tilde f'(s) \ge 0$ for all $s \in \mathbb{R}$ and consequently $\tilde f$ is monotone increasing. Therefore, we have $|\tilde f(y(x,t))| \le \max\{|\tilde f(y_u(x,t))|,|\tilde f(y_v(x,t))|\}$, which implies $|f(y)| \le |f(y_u)| + |f(y_v)|  + 3\Lambda_f(|y_u| + |y_v|)|$. From \eqref{E2.3} and \eqref{E2.4} we infer the existence of a constant $K_M$ such that $\|f(y)\|_{L^2_{\lambda_c}(Q)} \le K_M$. Combining this inequality with \eqref{E3.1}  and using \eqref{E2.7} we get \eqref{E3.5}.
\end{proof}

The goal of this section is to analyze the differentiability of the mapping control-to-state. As a first step  we analyze the linearized state equation.
\begin{lemma}
Assume that $\lambda > -2\Lambda_f$ and $b \in L^2_{\bar\lambda}(Q) \cap L^\infty(Q_T)$ for every $T < \infty$ satisfies $b(x,t) \ge \Lambda_f$ almost everywhere. Then, for every $h \in L^2_\lambda(Q)$ the equation
\begin{equation}
\left\{
\begin{array}{l}
\displaystyle\frac{\partial z}{\partial t} + Az + bz = h\ \mbox{ in } Q,\\[0.5ex] \partial_{n_A}z = 0 \ \mbox{ on } \Sigma,\   z(0) = 0 \ \mbox{ in } \Omega,
\end{array}
\right.
\label{E3.6}
\end{equation}
has a unique solution $z \in Y_\lambda$ in the sense of Definition \ref{D2.1}, and
\begin{equation}
\|z\|_{Y_\lambda} \le K_1\|h\|_{L^2_\lambda(Q)}
\label{E3.7}
\end{equation}
with $K_1 = \frac{2}{\min\{1,\frac{\lambda}{2},\Lambda_A\}}$.
In addition, if $h \in L^p_\lambda(0,\infty;L^2(\Omega))$, then $z \in C_\lambda(\bar Q)$ and
\begin{equation}
\|z\|_{C_\lambda(\bar Q)} \le K_2\|h\|_{L^p_\lambda(0,\infty;L^2(\Omega))},
\label{E3.8}
\end{equation}
where $K_2$ is independent of $b$ and $h$.
\label{L3.2}
\end{lemma}

\begin{proof}
We split the proof into two steps.

{\em Step I. Proof of \eqref{E3.7}.} For every $T < \infty$ we consider the equation
\begin{equation}
\left\{
\begin{array}{l}
\displaystyle
\frac{\partial z}{\partial t} + Az + bz = h\ \mbox{ in } Q_T,\\[0.5ex] \partial_{n_A}z = 0 \ \mbox{ on } \Sigma_T,\   z(0) = 0 \ \mbox{ in } \Omega.
\end{array}
\right.
\label{E3.9}
\end{equation}
This a standard linear parabolic equation having a unique solution $z \in W(0,T) \subset C([0,T];L^2(\Omega))$. Testing \eqref{E3.9} with $\e^{-\lambda t}z$ we get similarly as in \eqref{E2.9} the inequality
\begin{align*}
&\frac{1}{2}\e^{-\lambda T}\|z(T)\|^2_{L^2(\Omega)} + (\frac{\lambda}{2} + \Lambda_f)\!\int_0^T\e^{-\lambda t}\|z(t)\|^2_{L^2(\Omega)}\dt + \Lambda_A\!\int_0^T\e^{-\lambda t}\|\nabla z(t)\|^2_{L^2(\Omega)^n}\dt\\
&\le \|h\|_{L^2_\lambda(Q)}\|z\|_{L^2_\lambda(Q)} \le \frac{1}{2\min\{\frac{\lambda}{2} + \Lambda_f,\Lambda_A\}}\|h\|^2_{L^2_\lambda(Q)} + \frac{\min\{\frac{\lambda}{2} + \Lambda_f,\Lambda_A\}}{2}\|z\|^2_{L^2_\lambda(Q)}.
\end{align*}
Here we used that $b(x,t) \ge \Lambda_f$.   The above inequality leads to
\[
\|z\|_{C_\lambda([0,T];L^2(\Omega))} + \sqrt{\min\{\frac{\lambda}{2} + \Lambda_f,\Lambda_A\}}\|z\|_{L^2_\lambda(0,T;H^1(\Omega))}  \le \frac{2}{\sqrt{\min\{\frac{\lambda}{2} + \Lambda_f,\Lambda_A\}}}\|h\|_{L^2_\lambda(Q)}.
\]
Since $T$ was taken in $(0,\infty)$ arbitrarily we infer \eqref{E3.7}. \vspace{2mm}

{\em Step II. Proof of \eqref{E3.8}.} Now, we assume that $h \in L^p_\lambda(0,\infty;L^2(\Omega))$. Again the proof follows the steps of \cite[Theorem A.2]{Casas-Kunisch2023C}. We indicate how to integrate the term $\e^{-\lambda t}$ in the proof for the case $n = 2$ or 3, proceeding analogously for $n = 1$. We take $\rho > 0$ arbitrary and set $z_\rho(x,t) = \e^{-\frac{\lambda}{2}t}z(x,t) - \proj_{[-\rho,+\rho]}(\e^{-\frac{\lambda}{2}t}z(x,t))$ and $A_\rho(t) = \{x \in \Omega : \e^{\frac{\lambda}{2}t}|z(x,t)| > \rho\}$. Following the notation in \cite[Theorem A.2]{Casas-Kunisch2023C}, we take select $\alpha \in \big(\frac{pn}{2p-4},\frac{n}{n - 2}\big)$. We have that $\e^{-\frac{\lambda}{2}t}a(z,z_\rho) \ge a(z_\rho,z_\rho)$ and
\begin{align*}
&\langle\e^{-\frac{\lambda}{2}t}\frac{\partial z}{\partial t},z_\rho\rangle + \Lambda_f\int_\Omega zz_\rho\dx = \langle\frac{\partial(\e^{-\frac{\lambda}{2}t}z)}{\partial t},z_\rho\rangle + (\frac{\lambda}{2} + \Lambda_f)\e^{-\frac{\lambda}{2}t}\int_\Omega zz_\rho\dx\\
& \ge \langle\frac{\partial z_\rho}{\partial t},z_\rho\rangle + (\frac{\lambda}{2} + \Lambda_f)\int_\Omega z^2_\rho\dx = \frac{1}{2}\frac{d}{dt}\|z_\rho\|^2_{L^2(\Omega)} + (\frac{\lambda}{2} + \Lambda_f)\|z_\rho\|^2_{L^2(\Omega)},
\end{align*}
where $\langle\cdot,\cdot\rangle$ stands for the duality between $H^1(\Omega)^*$ and $H^1(\Omega)$. Testing the equation \eqref{E3.9} with $\e^{-\frac{\lambda}{2}t}z_\rho$, using the above identities, \eqref{E1.9}, the fact that $\frac{p}{2} \ge 1$, and integrating in $[0,t]$ for any $t \in (0,T]$ we get that
\begin{align*}
&\frac{1}{2}\|z_\rho(t)\|^2_{L^2(\Omega)} + (\frac{\lambda}{2} + \Lambda_f)\int_{Q_t}z^2_\rho\dx\ds + \Lambda_A\int_{Q_t}|\nabla z_\rho|^2\dx\ds\\
& \le \int_{Q_t}\e^{-\frac{\lambda}{2}s}hz_\rho\dx\ds \le \|h\|_{L^p_\lambda(0,\infty;L^2(\Omega))}\Big(\int_0^t\|z_\rho(s)\|^{p'}_{L^2(A_\rho(s))}\ds\Big)^{\frac{1}{p'}}\\
& \le \|h\|_{L^p_\lambda(0,\infty;L^2(\Omega))}\Big(\int_0^t\|z_\rho(s)\|^{p'}_{L^{2\alpha}(A_\rho(s))}|A_\rho(s)|^{\frac{p'}{2\alpha'}}\ds\Big)^{\frac{1}{p'}}
\end{align*}

\begin{align*}
&\le C_1\|h\|_{L^p_\lambda(0,\infty;L^2(\Omega))}\Big(\int_0^t\|z_\rho(s)\|^2_{H^1(\Omega)}\ds\Big)^{\frac{1}{2}}\Big(\int_0^t|A_\rho(s)|^{\frac{p'}{\alpha'(2 - p')}}\ds\Big)^{\frac{2 - p'}{2p'}}\\
&\le \frac{\min\{\frac{\lambda}{2}+\Lambda_f,\Lambda_A\}}{2}\!\int_0^t\|z_\rho(s)\|^2_{H^1(\Omega)}\ds\\
&+ \frac{C_1^2\|h\|^2_{L^p_\lambda(0,\infty;L^2(\Omega))}}{\min\{\lambda+2\Lambda_f,2\Lambda_A\}}\Big(\!\int_0^t|A_\rho(s)|^{\frac{p'}{\alpha'(2 - p')}}\ds\!\Big)^{\frac{2 - p'}{p'}}.
\end{align*}
This leads to
\[
\Big(\|z_\rho(t)\|^2_{L^2(\Omega)} + \|z_\rho\|^2_{L^2(0,t;H^1(\Omega))}\Big)^{\frac{1}{2}} \le C \|h\|_{L^p_\lambda(0,\infty;L^2(\Omega))}\Big(\int_0^t|A_\rho(s)|^{\frac{p}{\alpha'(p - 2)}}\ds\Big)^{\frac{p - 2}{2p}}
\]
for some constant $C$ independent of $t$ and $\rho$. This establishes the inequality (A.8) of \cite[Theorem A.2]{Casas-Kunisch2023C} and the rest of the proof is the same to infer the $L^\infty_\lambda(Q)$ estimate for $z$. The continuity in $\bar Q$ now follows from the continuity in $\bar Q_T$ for every $T < \infty$.
\end{proof}

\begin{lemma}
Assume that $\lambda > -2\Lambda_f$. Then, for every $u,v \in L^p_\lambda(0,\infty;L^2(\omega))$ the equation
\begin{equation}
\left\{
\begin{array}{l}
\displaystyle\frac{\partial z}{\partial t} + Az + f'(y_u)z = v\chi_\omega\ \mbox{ in } Q,\\[0.5ex] \partial_{n_A}z = 0 \ \mbox{ on } \Sigma,\   z(0) = 0 \ \mbox{ in } \Omega,
\end{array}
\right.
\label{E3.10}
\end{equation}
has a unique solution $z_v \in Y_\lambda \cap C_\lambda(\bar Q)$ and
\begin{equation}
\|z_v\|_{L^2_\lambda(0,\infty;H^1(\Omega))} + \|z_v\|_{C_\lambda(\bar Q)} \le C_{1,M}\|v\|_{L_\lambda^p(0,\infty;L^2(\omega))},
\label{E3.11}
\end{equation}
where $M = \|u\|_{L^p_\lambda(0,\infty;L^2(\omega))}$. Moreover, we have that
\begin{equation}
\left\{\begin{array}{l}\displaystyle\|y_u - y_v\|_{L^2_\lambda(0,\infty;H^1(\Omega))} + \|y_u - y_v\|_{C_\lambda(\bar Q)} \le C_{2,M}\|u - v\|_{L_\lambda^p(0,\infty;L^2(\omega))},\\
\displaystyle\|y_u - y_v\|_{Y_\lambda} \le C_{3,M}\|u - v\|_{L_\lambda^2(Q)},
\end{array}\right.
\label{E3.12}
\end{equation}
where $M = \max\{\|u\|_{L^p_\lambda(0,\infty;L^2(\omega))}, \|v\|_{L^p_\lambda(0,\infty;L^2(\omega))}\}$.
\label{L3.3}
\end{lemma}

\begin{proof}
The first part of the lemma is a straightforward consequence of Lemma \ref{L3.2}, \eqref{E2.7}, \eqref{E1.4}, and the fact $b = f'(y_u) \in L^2_{\bar\lambda}(Q)$; see Lemma \ref{L3.1}. We only need to prove \eqref{E3.12}. To this end we subtract the equations satisfied by $y_u$ and $y_v$ and obtain for some measurable function $\theta:Q \longrightarrow [0,1]$
\[
\left\{
\begin{array}{l}
\displaystyle\frac{\partial(y_u - y_v)}{\partial t} + A(y_u - y_v) + f'(y_u+\theta(y_v - y_u))(y_u - y_v) = (u - v)\chi_\omega\ \mbox{ in } Q,\\[0.5ex] \partial_{n_A}(y_u - y_v) = 0 \ \mbox{ on } \Sigma,\   (y_u - y_v)(0) = 0 \ \mbox{ in } \Omega.
\end{array}
\right.
\]
From Lemma \ref{L3.1} we get that $f'(y_u+\theta(y_v - y_u)) \in L^2_{\bar\lambda}(Q)$. Then, \eqref{E3.12} follows from Lemma \ref{L3.2} with $b = f'(y_u+\theta(y_v - y_u))$.
\end{proof}

\begin{theorem}
The mapping $G:L^p_{\lambda_c}(0,\infty;L^2(\omega)) \longrightarrow Y_{\sigma_s}$, $G(u) = y_u$, is of class $C^2$. Given $u, v, v_1, v_2 \in L^p_{\lambda_c}(0,\infty;L^2(\omega))$, then $z_v = G'(u)v$ and $z_{v_1v_2} = G''(u)(v_1,v_2)$ are the solutions to the equations
\begin{align}
&\left\{
\begin{array}{l}
\displaystyle\frac{\partial z}{\partial t} + Az + f'(y_u)z = v\chi_\omega\ \mbox{ in } Q,\\[0.5ex] \partial_{n_A}z = 0 \ \mbox{ on } \Sigma,\   z(0) = 0 \ \mbox{ in } \Omega,
\end{array}
\right.
\label{E3.13}\\
&\left\{
\begin{array}{l}
\displaystyle\frac{\partial z}{\partial t} + Az + f'(y_u)z = -f''(y_u)z_{v_1}z_{v_2}\ \mbox{ in } Q,\\[0.5ex] \partial_{n_A}z = 0 \ \mbox{ on } \Sigma,\   z(0) = 0 \ \mbox{ in } \Omega,
\end{array}
\right.
\label{E3.14}
\end{align}
where $z_{v_i} = G'(u)v_i$, $i = 1, 2$.
\label{T3.1}
\end{theorem}

\begin{proof}
We split the proof into three steps.

{\em Step I: $G$ is Fr\'echet differentiable.} Given $u \in L^p_{\lambda_c}(0,\infty;L^2(\omega))$, for every $v$ in the same space we denote by $z_v$ the solution of \eqref{E3.13}. From Lemma \ref{L3.3} we know that such a solution exists, it is unique, and $z_v \in Y_{\lambda_c} \cap C_{\lambda_c}(\bar Q) \subset Y_{\sigma_s} \cap C_{\sigma_s}(\bar Q)$. We are going to prove that
\begin{equation}
\lim_{\|v\| \to 0}\frac{\|G(u+v) - G(u) - z_v\|_{Y_{\sigma_s}}}{\|v\|} = 0,
\label{E3.15}
\end{equation}
where $\|\cdot\|$ stands for the norm in $L^p_{\lambda_c}(0,\infty;L^2(\omega))$. Without loss of generality we can assume that $\|v\| \le 1$ and we set $M = \|u\| + 1$. Let us denote $w = G(u+v) - G(u) - z_v = y_{u+v} - y_u - z_v$. Subtracting the associated equations and making the Taylor expansion
\[
f(y_{u+v}) = f(y_u) + f'(y_u)(y_{u+v} - y_u) + \frac{1}{2}f''(y_\theta)(y_{u+v} - y_u)^2
\]
with $y_\theta = y_u + \theta(y_{u+v} - y_u)$ and $\theta:Q \longrightarrow [0,1]$ measurable, we obtain
\[
\left\{
\begin{array}{l}
\displaystyle\frac{\partial w}{\partial t} + Aw + f'(y_u)w = -\frac{1}{2}f''(y_\theta)(y_{u+v} - y_u)^2\ \mbox{ in } Q,\\[0.5ex] \partial_{n_A}w = 0 \ \mbox{ on } \Sigma,\   w(0) = 0 \ \mbox{ in } \Omega.
\end{array}
\right.
\]
Let us prove that $f''(y_\theta)(y_{u+v} - y_u)^2 \in L^2_{\sigma_s}(Q)$. Using \eqref{E3.5} and \eqref{E3.12}, and the fact that $\sigma_s > \bar\lambda + \lambda_c$ we get
\[
\|f''(y_\theta)(y_{u+v} - y_u)^2\|_{L^2_{\sigma_s}(Q)} \le \|f''(y_\theta)\|_{L^2_{\bar\lambda}(Q)}\|y_{u+v} - y_u\|^2_{L^\infty_{\lambda_c}(Q)} \le C_MC_{2,M}\|v\|^2.
\]
Combining this inequality with \eqref{E3.7} for $\sigma_s$ we get $\|w\|_{Y_{\sigma_s}} \le K_1C_MC_{2,M}\|v\|^2$, which implies \eqref{E3.15}. Finally we observe that the continuity of the linear mapping $G'(u)$ is a straightforward consequence of \eqref{E3.11}. \vspace{2mm}

{\em Step II: $G$ is twice Fr\'echet differentiable.} Given $u, v_1, v_2 \in L^p_{\lambda_c}(0,\infty;L^2(\omega))$ we denote $y_{u + v_1} = G(u + v_1)$, $y_u = G(u)$, $z_{v_i} = G'(u)v_i$, $i = 1, 2$, and $\eta_{v_1v_2} = G'(u + v_1)v_2$. Let us prove the existence of a unique solution $z_{v_1v_2}$ of \eqref{E3.14} and the continuity of the bilinear mapping $G''(u)$. For this we estimate the right hand side of the equation:

\begin{align*}
&\Big(\int_Q\e^{-\sigma_s t}|f''(y_u)|^2|z_{v_1}|^2|z_{v_2}|^2\dx\dt\Big)^{\frac{1}{2}} \le \|f''(y_u)\|_{L^2_{\bar\lambda}(Q)}\|z_{v_1}\|_{L^\infty_{\lambda_c}(Q)}\|z_{v_2}\|_{L^\infty_{\lambda_c}(Q)}\\
&\le C_{1,M}^2\|f''(y_u)\|_{L^2_{\bar\lambda}(Q)}\|v_1\|\|v_2\|,
\end{align*}
where we used that $\sigma_s > \bar\lambda + 2{\lambda_c}$ and \eqref{E3.11}. Then, the result follows from Lemma \ref{L3.2}.

Next we prove that
\begin{equation}
\lim_{\|v_1\| \to 0}\sup_{\|v_2\| = 1}\frac{\|\eta_{v_1v_2} - z_{v_2} - z_{v_1v_2}\|_{Y_{\sigma_s}}}{\|v_1\|} = 0.
\label{E3.16}
\end{equation}
Setting $w = \eta_{v_1v_2} - z_{v_2} - z_{v_1v_2}$, using the mean value theorem
\[
f'(y_{u + v_1}) = f'(y_u) + f''(y_\vartheta)(y_{u + v_1}- y_u)\text{ with } y_\vartheta = y_u + \vartheta(y_{u + v_1} - y_u),
\]
and subtracting the associated equations we get
\[
\left\{
\begin{array}{l}
\displaystyle\frac{\partial w}{\partial t} + Aw + f'(y_u)w = -f''(y_\vartheta)(y_{u+v_1} - y_u)\eta_{v_1v_2} + f''(y_u)z_{v_1}z_{v_2}\ \mbox{ in } Q,\\[0.5ex] \partial_{n_A}w = 0 \ \mbox{ on } \Sigma,\   w(0) = 0 \ \mbox{ in } \Omega.
\end{array}
\right.
\]
Applying Lemma \ref{L3.2} we obtain
\begin{align*}
\|w\|_{Y_{\sigma_s}} & \le K_1\|-f''(y_\vartheta)(y_{u+v_1} - y_u)\eta_{v_1v_2} + f''(y_u)z_{v_1}z_{v_2}\|_{L^2_{\sigma_s}(Q)}\\
&\le K_1\Big(\|f''(y_\vartheta)(y_{u+v_1} - y_u)(\eta_{v_1v_2} - z_{v_2})\|_{L^2_{\sigma_s}(Q)}\\
& + \|f''(y_\vartheta)(y_{u+v_1} - y_u -z_{v_1})z_{v_2}\|_{L^2_{\sigma_s}(Q)}\\
&+ \|[f''(y_u) - f''(y_\vartheta)]z_{v_1}z_{v_2}\|_{L^2_{\sigma_s}(Q)}\Big) = K_1(I_1 + I_2 + I_3).
\end{align*}
First we estimate $I_1$. From \eqref{E3.12} we get
\begin{align}
&\|f''(y_\vartheta)(y_{u+v_1} - y_u)(\eta_{v_1v_2} - z_{v_2})\|_{L^2_{\sigma_s}(Q)}\notag\\
&\le C_{2,M}\|v_1\|\|f''(y_\vartheta)(\eta_{v_1v_2} - z_{v_2})\|_{L^2_{\sigma_s - {\lambda_c}}(Q)}.\label{E3.17}
\end{align}
It is enough to prove that $\lim_{\|v_1\| \to 0}\|f''(y_\vartheta)(\eta_{v_1v_2} - z_{v_2})\|_{L^2_{\sigma_s - \lambda_c}(Q)} = 0$. For this purpose we apply Lemma \ref{L2.1} with $r = 2$, $X = L^2(\Omega)$, $\hat\lambda = \sigma_s - \lambda_c$, and $\lambda = \sigma_s - 2\lambda_c$. Note that $\sigma_s - 2{\lambda_c} > \bar\lambda$. First we have with Lemma \ref{L3.1}-ii)
\begin{align*}
&\|f''(y_\vartheta)(\eta_{v_1v_2} - z_{v_2})\|^2_{L^2_{\sigma_s - 2\lambda_c}(Q)}\\
& \le 2\!\int_0^\infty\e^{-\bar\lambda t}\|f''(y_\vartheta)\|^2_{L^2(\Omega)}\!\dt \Big(\|\eta_{v_1v_2}\|^2_{L^\infty_{\lambda_c}(Q)} {+} \|z_{v_2}\|^2_{L^\infty_{\lambda_c}(Q)}\Big) \le C^2_{1,M}\|f''(y_\vartheta)\|^2_{L_{\bar\lambda}^2(Q)} < \infty.
\end{align*}
Now, we prove the convergence $\lim_{\|v_1\| \to 0}\|f''(y_\vartheta)(\eta_{v_1v_2} - z_{v_2})\|_{L^2(Q_T)} = 0$ for every $T < \infty$. Subtracting the equations for $\eta_{v_1v_2}$ and $z_{v_2}$ we get for $\psi = \eta_{v_1v_2} - z_{v_2}$
\[
\left\{\!\!\!
\begin{array}{l}
\displaystyle\frac{\partial\psi}{\partial t} + A\psi + f'(y_u)\psi = -f''(y_\vartheta)(y_{u+v_1}- y_u)\eta_{v_1v_2}\ \mbox{in}\ Q_{T_\varepsilon},\\[0.5ex] \partial_{n_A}\psi = 0 \ \mbox{ on } \Sigma_{T_\varepsilon},\   \psi(0) = 0 \mbox{ in } \Omega.
\end{array}
\right.
\]
Since $\|y_{u+v_1}- y_u\|_ {L^\infty_{\lambda_c}(Q)} \le C_{2,M}\|v_1\|$ we infer that
$\|y_{u+v_1}- y_u\|_ {L^\infty(Q_T)} \to 0$ as $\|v_1\| \to 0$. Combining this with the estimate  $\eta_{v_1v_2}\|_{L^\infty_{\lambda_c}(Q)} \le C_{1,M}\|v_2\| = C_{1,M}$ and the boundedness of $f''(y_\vartheta)$ in $Q_T$ we infer that the right hand side of the above equation converges to 0 in $L^\infty(Q_T)$ and hence $\|\psi\|_{L^\infty(Q_T)} \to 0$ as $\|v_1\| \to 0$. Then, Lemma \ref{L2.1} applies and we deduce from \eqref{E3.17} that $\frac{1}{\|v_1\|}I_1 \to 0$ as $\|v_1\| \to 0$.

Let us estimate $I_2$. Since $\|v_2\| = 1$, we infer from \eqref{E3.11} that
\begin{align*}
I_2 &\le \|f''(y_\vartheta)(y_{u+v_1} - y_u -z_{v_1})\|_{L^2_{\sigma_s-\lambda_c}(Q)}\|z_{v_2}\|_{L^\infty_{\lambda_c}(Q)}\\
& \le C_{1,M}\|f''(y_\vartheta)(y_{u+v_1} - y_u -z_{v_1})\|_{L^2_{\sigma_s - \lambda_c}(Q)}.
\end{align*}
To estimate $\|f''(y_\vartheta)(y_{u+v_1} - y_u -z_{v_1})\|_{L^2_{\sigma_s-\lambda_c}(Q)}$ we use Lemma \ref{L2.1} again with $r = 2$, $X = L^2(\Omega)$, $\hat\lambda = \sigma_s - \lambda_c$, and $\lambda = \lambda_c + \bar\lambda$. First, we get
\begin{align*}
 &\frac{1}{\|v_1\|}\|f''(y_\vartheta)(y_{u+v_1} - y_u -z_{v_1})\|_{L^2_{\lambda_c + \bar\lambda}(Q)}\\
 &\le \frac{\sqrt{2}}{\|v_1\|}\|f''(y_\vartheta)\|_{L^2_{\bar\lambda}(Q)}\Big(\|y_{u+v_1} - y_u\|_{L^\infty_{\lambda_c}(Q)} + \|z_{v_1}\|_{L^\infty_{\lambda_c}(Q)}\Big)\\
 & \le \sqrt{2}(C_{1,M} + C_{2,M})\|f''(y_\vartheta)\|_{L^2_{\bar\lambda}(Q)} < \infty.
\end{align*}
For every $T < \infty$, in the interval $(0,T)$, the equation satisfied by $y_{u+v_1} - y_u -z_{v_1}$ is the same as the one for $w$ in {\em Step I}. Therefore, we have
\[
\|y_{u+v_1} - y_u -z_{v_1}\|_{L^\infty(Q_T)} \le C\|f''(y_\theta)(y_{u+v_1} - y_u)^2\|_{L^\infty(Q_T)} \le C'\|v_1\|^2.
\]
This implies that
\[
\frac{1}{\|v_1\|}\|f''(y_\vartheta)(y_{u+v_1} - y_u -z_{v_1})\|_{L^2(Q_T)} \le C'\|f''(y_\vartheta)\|_{L\infty(Q_T)}\|v_1\| \stackrel{\|v_1\| \to 0}{\longrightarrow} 0.
\]
From Lemma \ref{L2.1} we deduce that $\frac{1}{\|v_1\|}I_2 \to 0$ as $\|v_1\| \to 0$.

Now we estimate $I_3$. First we observe
\begin{align*}
I_3 &\le \|f''(y_u) - f''(y_\vartheta)\|_{L^2_{\sigma_s-2\lambda_c}(Q)}\|z_{v_1}\|_{L^\infty_{\lambda_c}(Q)}\|z_{v_2}\|_{L^\infty_{\lambda_c}(Q)}\\
&\le C^2_{1,M}\|f''(y_u) - f''(y_\vartheta)\|_{L^2_{\sigma_s-2{\lambda_c}}(Q)}\|v_1\|,
\end{align*}
where we have used that $\|v_2\| = 1$. This leads to
\[
\frac{1}{\|v_1\|}I_3 \le C^2_{1,M}\|f''(y_u) - f''(y_\vartheta)\|_{L^2_{\sigma_s - 2{\lambda_c}}}(Q).
\]
Then, it is enough to prove that $\|f''(y_u) - f''(y_\vartheta)\|_{L^2_{\sigma_s - 2{\lambda_c}}}(Q) \to 0$ as $\|v_1\| \to 0$. This follows from Lemma \ref{L2.1} with $r = 2$, $X = L^2(\Omega)$, $\hat\lambda = \sigma_s - 2\lambda_c$ and $\lambda = \bar\lambda$.\vspace{2mm}

{\em Step III. Continuity of $G''$.} Here we prove the continuity of the mapping
\begin{align*}
G'':L^p_{\lambda_c}(0,\infty;L^2(\omega)) &\longrightarrow \mathcal{B}\Big(L^p_{\lambda_c}(0,\infty;L^2(\omega)) \times L^p_{\lambda_c}(0,\infty;L^2(\omega)),Y_{\sigma_s}\Big)\\
&G''(u)(v_1,v_2) = z_{v_1v_2}.
\end{align*}
Let $\{u_k\}_{k = |}^\infty$ be a sequence such that $u_k \to u$ in $L^p_{\lambda_c}(0,\infty;L^2(\omega))$. We have to prove that
\begin{equation}
\lim_{k \to \infty}\sup_{\|v_1\| = \|v_2\| = 1}\|G''(u_k)(v_1,v_2) - G''(u)(v_1,v_2)\|_{Y_{\sigma_s}} = 0.
\label{E3.19}
\end{equation}
We set $z^k_{v_1v_2} = G''(u_k)(v_1,v_2)$, $z_{v_1v_2} = G''(u)(v_1,v_2)$, and $w_k = z^k_{v_1v_2} - z_{v_1v_2}$. Then, subtracting the equations satisfied by $z^k_{v_1v_2}$ and $z_{v_1v_2}$ we get for every $T < \infty$
\[
\left\{
\begin{array}{l}
\displaystyle\frac{\partial w_k}{\partial t} + Aw_k + f'(y_u)w_k\\
\hspace{1cm} = f''(y_u)z_{v_1}z_{v_2} - f''(y_{u_k})z^k_{v_1}z^k_{v_2} - f''(y_{\theta_k})(y_{u_k} - y_u)z^k_{v_1v_2}\ \mbox{ in } Q_T,\\[0.5ex] \partial_{n_A}w_k = 0 \ \mbox{ on } \Sigma_T,\   w_k(0) = 0 \ \mbox{ in } \Omega,
\end{array}
\right.
\]
where $z^k_{v_i} = G'(u_k)v_i$, $z_{v_i} = G'(u)v_i$, $i = 1, 2$, and $y_{\theta_k} = y_u + \theta_k(y_{u_k} - y_u)$ with $\theta_k:Q \longrightarrow [0,1]$ measurable. From \eqref{E2.7} we know that $\{y_{u_k}\}_{k = 1}^\infty$ and $y_u$ are uniformly bounded in $Q_T$. Hence, we apply Lemmas \ref{L3.2} and \ref{L3.3} to deduce that the right hand side of this equation converges to 0 in $L^\infty(Q_T)$. Therefore, using again Lemma \ref{L3.3}  we infer that
\[
\lim_{k \to \infty} \Big(\|z^k_{v_1v_2} - z_{v_1v_2}\|_{L^2(0,T;H^1(\Omega))} +  \|z^k_{v_1v_2} - z_{v_1v_2}\|_{C([0,T];L^2(\Omega))}\Big) = 0.
\]
Further from Lemma \ref{L3.2} we infer
\begin{align*}
&\|z_{v_1v_2}^k\|_{L_{\bar\lambda + 2{\lambda_c}}^2(0,\infty;H^1(\Omega))} + \|z_{v_1v_2}^k\|_{C_{\bar\lambda + 2{\lambda_c}}([0,\infty);L^2(\Omega))}\\
&\le K_1\|f''(y_{u_k})z_{v_1}^kz_{v_2}^k\|_{L^2_{\bar\lambda + 2{\lambda_c}}(Q)} \le  K_1\|f''(y_{u_k})\|_{L_{\bar\lambda}(Q)}\|z_{v_1}^k\|_{L^\infty_{\lambda_c}(Q)}\|z_{v_2}^k\|_{L^\infty_{\lambda_c}(Q)}\\
&\le K_1K_2^2\|f''(y_{u_k})\|_{L_{\bar\lambda}^2(Q)}\|v_1\|\|v_2\| = K_1K_2^2\|f''(y_{u_k})\|_{L^2_{\bar\lambda}(Q)} \le C,
\end{align*}
where we used that $\|v_1\|= \|v_2\| = 1$, \eqref{E3.5}, and the fact that $\{u_k\}_{k = 1}^\infty$ is bounded in $L^p_{\lambda_c}(0,\infty;L^2(\omega))$. Now, we apply Lemma \ref{L2.1} first with $r = 2$ and $X = H^1(\Omega)$ and later with $r = \infty$ and $X = L^2(\Omega)$, $\hat\lambda = \sigma_s$ and $\lambda = \bar\lambda + 2\lambda_c$ in both cases, to get \eqref{E3.19}.
\end{proof}

\begin{remark}
The continuous embedding $L^\infty(0,\infty;L^2(\omega)) \subset L^p_\lambda(0,\infty;L^2(\omega))$ for every $\lambda > 0$ and Theorem \ref{T3.1} imply that $G:L^\infty(0,\infty;L^2(\omega)) \longrightarrow Y_{\sigma_s}$ is also of class $C^2$ and the expressions for $G'(u)v$ and $G''(u)(v_1,v_2)$ given in Theorem \ref{T3.1} remain correct.
\label{R3.1}
\end{remark}

\section{Existence of a solution for \Pb and first-order optimality conditions}
\label{S4}

In this section, we analyze the control problem, proving the existence of a solution and establishing the first-order optimality conditions satisfied by any local solution of \Pb. We say that a control $\bar u \in \Uad$ is a local solution of \Pb if
\begin{equation}
\exists \varepsilon > 0 \text{ such that }J(\bar u) \le J(u) \ \forall u \in \Uad : \|u - \bar u\|_{L^2_{\sigma_c}(Q_\omega)} \le \varepsilon,
\label{E4.1}
\end{equation}
where $Q_\omega = \omega \times (0,\infty)$. Observe that if \eqref{E4.1} is fulfilled, then it also holds for every $\lambda > 0$ and some $\varepsilon_\lambda > 0$. Indeed, if $\lambda < \sigma_c$, then using that $\|u - \bar u\|_{L^2_\lambda(Q_\omega)} \ge \|u - \bar u\|_{L^2_{\sigma_c}(Q_\omega)}$, the above inequality holds with $\lambda$ and $\varepsilon_\lambda = \varepsilon$. For $\lambda > \sigma_c$, we use the boundedness of $\K$ and get that $\|u\|_{L^\infty(0,\infty;L^2(\omega))} \le \gamma_{ad}$ for some $\gamma_{ad} < \infty$ and all $u \in \Uad$. Thus, we have
\begin{align}
&\|u - \bar u\|_{L^2_{\sigma_c}(Q_\omega)} \le \|u - \bar u\|^{1 - \frac{\sigma_c}{2\lambda}}_{L^\infty(0,\infty;L^2(\omega))}\Big(\int_0^\infty \e^{-\sigma_ct}\|u(t) - \bar u(t)\|^{\frac{\sigma_c}{\lambda}}_{L^2(\omega)}\dt\Big)^{\frac{1}{2}}\notag\\
&\le (2\gamma_{ad})^{1 - \frac{\sigma_c}{2\lambda}} \Big(\int_0^\infty \e^{-\frac{\sigma_c}{2}t}\|u(t) - \bar u(t)\|^{\frac{\sigma_c}{\lambda}}_{L^2(\omega)}\e^{-\frac{\sigma_c}{2}t}\dt\Big)^{\frac{1}{2}}\notag\\
&\le (2\gamma_{ad})^{1 - \frac{\sigma_c}{2\lambda}}\Big(\int_0^\infty \e^{-\frac{\sigma_c\lambda}{2\lambda-\sigma_c}t}\dt\Big)^{\frac{2\lambda-\sigma_c}{4\lambda}}\Big(\int_0^\infty \e^{-\lambda t}\|u(t) - \bar u(t)\|^2_{L^2(\omega)}\dt\Big)^{\frac{\sigma_c}{4\lambda}}.\notag\\
&= C_{\sigma_c,\lambda}\|u - \bar u\|_{L^2_{\lambda}(Q_\omega)}^{\frac{\sigma_c}{2\lambda}}.\label{E4.1A}
\end{align}
Then, it is enough to take $\varepsilon_\lambda = \Big(\frac{1}{C_{\sigma_c,\lambda}}\varepsilon\Big)^{\frac{2\lambda}{\sigma_c}}$. It is also immediate to check that $\bar u$ is a local solution in the $L^2_{\sigma_c}(Q_\omega)$ sense, as above defined, if and only if it is a local solution in the $L_\lambda^r(0,\infty;L^2(\omega))$ sense for every $r \in [1,\infty)$.

This follows from the following inequalities:
\begin{align*}
&\text{If } r > 2 \text{ then}\\
&\|u - \bar u\|_{L^r_\lambda(0,\infty;L^2(\omega))} \le \|u - \bar u\|^{\frac{r - 2}{r}}_{L^\infty(0,\infty;L^2(\omega))}\|u - \bar u\|^{\frac{2}{r}}_{L^2_\lambda(Q_\omega)} \le (2\gamma_{ad})^{\frac{2}{r}}\|u - \bar u\|^{\frac{2}{r}}_{L^2_\lambda(Q_\omega)},\\
&\|u - \bar u\|_{L^2(Q_\omega)} \le \Big(\int_0^\infty\e^{-\lambda t}\dt\Big)^{\frac{r - 2}{2r}}\|u - \bar u\|_{L^r_\lambda(0,\infty;L^2(\omega))} = \frac{1}{\lambda^{\frac{r - 2}{2r}}}\|u - \bar u\|_{L^r_\lambda(0,\infty;L^2(\omega))},\\
&\text{if } r < 2 \text{ then}\\
&\|u - \bar u\|_{L^r_\lambda(0,\infty;L^2(\omega))} \le \Big(\int_0^\infty\e^{-\frac{r\lambda}{r - 2}t}\dt\Big)^{\frac{2 - r}{2r}}\|u - \bar u\|_{L^2_\lambda(Q_\omega)} = \Big(\frac{r-2}{r\lambda}\Big)^{\frac{2 - r}{2r}}\|u - \bar u\|_{L^2_\lambda(Q_\omega)},\\
&\|u - \bar u\|_{L^2(Q_\omega)} \le (2\gamma_{ad})^{\frac{2 - r}{2}}\|u - \bar u\|^{\frac{r}{2}}_{L^r_\lambda(0,\infty;L^2(\omega))}.
\end{align*}
We also have that any local solution in the $L^2_{\sigma_c}(Q_\omega)$ sense is also a local solution in the $L_\lambda^\infty(0,\infty;L^2(\omega))$ sense, but the contrary implication fails.

\begin{theorem}
Problem \Pb has at least one solution.
\label{T4.1}
\end{theorem}

To prove this theorem it is enough to take a minimizing sequence $\{u_k\}_{k = 1}^\infty \subset \Uad$ weakly$^*$ converging to some $\bar u$ in $L^\infty(0,\infty;L^2(\omega))$ and apply \eqref{E2.18}.

Next we derive the first-order optimality conditions satisfied by any local solution of \Pb. For this purpose we first analyze the adjoint state equation.

\begin{theorem}
Assume that $\lambda > -2\Lambda_f$, $h \in L^2_\lambda(Q)$, and $u \in L^\infty(0,\infty;L^2(\omega))$. Then, the problem
\begin{equation}
\left\{
\begin{array}{l}
\displaystyle-\frac{\partial\varphi}{\partial t} + A\varphi + f'(y_u)\varphi = \e^{-\lambda t}h\ \mbox{ in } Q,\\[0.5ex] \partial_{n_A}\varphi = 0 \ \mbox{ on } \Sigma,\   \lim_{t \to \infty}\e^{-\Lambda_f t}\|\varphi(t)\|_{L^2(\Omega)} = 0,
\end{array}
\right.
\label{E4.2}
\end{equation}
has a unique solution $\varphi \in Y_{-\lambda}$ and there exists a constant $C$ independent of $u$ and $h$ such that
\begin{equation}
\|\varphi\|_{Y_{-\lambda}} \le C\|h\|_{L^2_\lambda(Q)}.
\label{E4.3}
\end{equation}
In addition, if $h \in L^p_{\frac{p}{2}\lambda}(0,\infty;L^2(\Omega))$, then we have $\varphi \in C_{-\lambda}(\bar Q)$ and there exists a constant $C_\infty$ independent of $u$ and $h$ such that
\begin{equation}
\|\varphi\|_{C_{-\lambda}(\bar Q)} \le C_\infty\|h\|_{L^p_{\frac{p}{2}\lambda}(0,\infty;L^2(\Omega))}.
\label{E4.4}
\end{equation}
Further, if $u_k \stackrel{*}{\rightharpoonup} u$ in $L^\infty(0,\infty;L^2(\omega))$ and $h_k \to h$ in $L^p_{\frac{p}{2}\lambda}(0,\infty;L^2(\Omega))$ hold, then
\begin{equation}
\varphi_k \stackrel{*}{\rightharpoonup} \varphi \ \text{ in } Y_{-\lambda}\ \text{ and }\ \lim_{k \to \infty}\Big(\|\varphi_k - \varphi\|_{Y_{-\hat\lambda}} + \|\varphi_k - \varphi\|_{C_{-\hat\lambda}(\bar Q)}\Big) = 0\quad \forall\hat\lambda < \lambda.
\label{E4.5}
\end{equation}
where $\varphi_k$ is the solution of \eqref{E4.2} with $(u,h)$ replaced by $(u_k,h_k)$.
\label{T4.2}
\end{theorem}

\begin{proof}
For the proof of uniqueness, the reader is referred to \cite[Theorem 3.2]{Casas-Kunisch2023B}. Next we prove the existence of a solution and the associated estimates.\vspace{0.2mm}

{\em Step I - Existence of a solution.} Let $\{T_k\}_{k = 1}^\infty$ be a monotone increasing sequence of positive numbers converging to infinite. For every $k$ we consider the following equation
\begin{equation}
\left\{
\begin{array}{l}
\displaystyle-\frac{\partial\varphi_k}{\partial t} + A\varphi_k + f'(y_u)\varphi_k = \e^{-\lambda t}h\ \mbox{ in } Q_{T_k},\\[0.5ex] \partial_{n_A}\varphi_k = 0 \ \mbox{ on } \Sigma_{T_k},\   \varphi_k(T_k)= 0 \ \mbox{ in } \Omega.
\end{array}
\right.
\label{E4.6}
\end{equation}
Existence and uniqueness of a solution $\varphi_k \in  H^1(Q_{T_k}) \! \cap  \!C([0,T_k],H^1(\Omega))$ are well known; see \cite[Corollary 2.4]{Showalter1997}. Testing equation \eqref{E4.6} with $\e^{\lambda t}\varphi_k$ for $t \in [0,T_k)$ we get
\begin{align*}
&\frac{\e^{\lambda t}}{2}\|\varphi_k(t)\|^2_{L^2(\Omega)} + \min\{\frac{\lambda}{2} + \Lambda_f,\Lambda_A\}\int_t^{T_k}\e^{\lambda s}\|\varphi_k(s)\|^2_{H^1(\Omega))}\ds\\
&\le \frac{\e^{\lambda t}}{2}\|\varphi_k(t)\|^2_{L^2(\Omega)} + \frac{\lambda}{2}\int_t^{T_k}\e^{\lambda s}\|\varphi_k(s)\|^2_{L^2(\Omega)}\ds + \int_t^{T_k}\e^{\lambda s}a(\varphi_k(s),\varphi_k(s))\ds\\
& + \int_t^{T_k}\int_\Omega\e^{\lambda s}f'(y_u)\varphi_k^2\dx\ds = \int_t^{T_k}\int_\Omega h\varphi_k\dx\ds\\
& \le \frac{1}{2\min\{\frac{\lambda}{2} + \Lambda_f,\Lambda_A\}}\|h\|^2_{L^2_\lambda(Q)} + \frac{\min\{\frac{\lambda}{2} + \Lambda_f,\Lambda_A\}}{2}\int_t^{T_k}\e^{\lambda s}\|\varphi_k(s)\|^2_{L^2(\Omega)}\ds.
\end{align*}
Let us denote by $\hat\varphi_k$ the extension by zero of $\varphi_k$ to $Q$. Then, from the above inequality we infer the existence of a constant $C > 0$ such that
\begin{equation}
\|\hat\varphi_k\|_{Y_{-\lambda}} \le C\|h\|_{L^2_\lambda(Q)}.
\label{E4.7}
\end{equation}
Therefore, there exists a subsequence, denoted in the same way, such that
\[
\hat\varphi_k \stackrel{*}{\rightharpoonup} \varphi \text{ in } L^\infty_{-\lambda}((0,\infty);L^2(\Omega))\ \text{ and }\ \hat\varphi_k \rightharpoonup \varphi \text{ in } L^2_{-\lambda}(0,\infty;H^1(\Omega)).
\]
Passing to the limit in \eqref{E4.6} we deduce that $\varphi$ satisfies  equation \eqref{E4.2}. The continuity of $\varphi:[0,\infty) \longrightarrow L^2(\Omega)$ follows from the fact that $\varphi \in W(0,T) \subset C([0,T];L^2(\Omega))$ for every $T < \infty$. Inequality \eqref{E4.3} follows from \eqref{E4.7}. It only remains to verify the last equality in \eqref{E4.2}:
\begin{align*}
& \lim_{t \to \infty}\e^{-\Lambda_f t}\|\varphi(t)\|_{L^2(\Omega)} = \lim_{t \to \infty}\e^{-(\Lambda_f + \frac{\lambda}{2})t}\e^{\frac{\lambda}{2}t}\|\varphi(t)\|_{L^2(\Omega)}\\
& \le C\|h\|_{L^2_\lambda(Q)}\lim_{t \to \infty}\e^{-(\Lambda_f + \frac{\lambda}{2})t} = 0,
\end{align*}
where we have used that $\lambda > -2\Lambda_f$.\vspace{2mm}

{\em Step II - Proof of \eqref{E4.4}.} Let us set $\psi_k = \e^{\frac{\lambda}{2}t}\hat\varphi_k$. From \eqref{E4.6} and the definition of $\hat\varphi_k$ we infer
\begin{equation}
\left\{
\begin{array}{l}
\displaystyle-\frac{\partial\psi_k}{\partial t} + \frac{\lambda}{2} \psi_k + A\psi_k + f'(y_u)\psi_k = \e^{-\frac{\lambda}{2}t}\chi_{(0,T_k)}h\ \mbox{ in } Q,\\[0.5ex]\displaystyle \partial_{n_A}\psi_k = 0 \ \mbox{ on } \Sigma,\   \lim_{t \to \infty}\|\psi_k(t)\|_{L^2(\Omega)} = 0.
\end{array}\right.
\label{E4.8}
\end{equation}
We apply \cite[Theorem A.4]{Casas-Kunisch2023C} to the above equation and recall that $L^p_\lambda(0,\infty;L^2(\Omega)) \subset L^2_\lambda(Q)$ to deduce that $\psi_k \in L^\infty(Q)$ and the existence of a constant $C_\infty$ such that
\begin{equation}
\|\psi_k\|_{L^\infty(Q)} \le C_\infty\|\e^{-\frac{\lambda}{2}t}h\|_{L^p(0,T_k;L^2(\Omega))} = C_\infty\|h\|_{L_{\frac{p}{2}\lambda}^p(0,\infty;L^2(\Omega))}.
\label{E4.9}
\end{equation}
Moreover, from \eqref{E4.6} we deduce that $\varphi_k$ is a continuous function in $\bar Q_{T_k}$. This implies the continuity of $\hat\varphi_k$ in $\bar Q$. Along with the above estimate this yields $\{\psi_k\}_{k = 1}^\infty \subset C(\bar Q)$. From the convergence $\hat\varphi_k \stackrel{*}{\rightharpoonup} \varphi$ in $Y_{-\lambda}$, we infer that $\psi_k \stackrel{*}{\rightharpoonup} \psi = \e^{\frac{\lambda}{2}t}\varphi$ in $C([0,\infty;L^2(\Omega)) \cap L^2(0,\infty;H^1(\Omega))$. Now we prove that $\{\psi_k\}_{k = 1}^\infty$ is a Cauchy sequence in $C(\bar Q)$. Given $\varepsilon > 0$ we select $T_{k_\varepsilon} < \infty$ such  that $\|h\|_{L^p_{\frac{p}{2}\lambda}(T_{k_\varepsilon},\infty;L^2(\Omega))} < \varepsilon$. Suppose that $k_2 > k_1 \ge k_\varepsilon$. Subtracting the equations satisfied by $\psi_{k_2}$ and $\psi_{k_1}$, we again infer from \cite[Theorem A.4]{Casas-Kunisch2023C}
\[
\|\psi_{k_2} - \psi_{k_1}\|_{C(\bar Q)} \le C_\infty\|\e^{-\frac{\lambda}{2}t}h\|_{L^p(T_{k_1},T_{k_2};L^2(\Omega))} \le C_\infty\|h\|_{L_{\frac{p}{2}\lambda}^p(T_{k_\varepsilon},\infty;L^2(\Omega))} < C_\infty\varepsilon.
\]
Then, we have $\lim_{k \to \infty}\|\psi - \psi_k\|_{C(\bar Q)} = 0$. This together with \eqref{E4.9} proves \eqref{E4.4}.

{\em Step III - Proof of \eqref{E4.5}.} Let us set $\psi_k = \varphi_k - \varphi$. Subtracting the equations satisfied by $\varphi_k$ and $\varphi$ we obtain
\begin{equation}
\left\{\hspace{-0.2cm}
\begin{array}{l}
\displaystyle-\frac{\partial\psi_k}{\partial t} + A\psi_k + f'(y_u)\psi_k = \e^{-\lambda t}\Big[(h_k\! -\! h) + [f'(y_u)\! -\! f'(y_{u_k})]\e^{\lambda t}\varphi_k\Big]\mbox{ in } Q,\vspace{0.3mm}\\ \partial_{n_A}\psi_k = 0 \ \mbox{ on } \Sigma,\   \lim_{t \to \infty}\e^{\Lambda_f t}\|\psi_k(t)\|_{L^2(\Omega)} = 0.
\end{array}
\right.
\label{E4.10}
\end{equation}
From the convergence $h_k \to h$ in {$L^p_{\frac{p}{2}\lambda}(0,\infty;L^2(\Omega))$}, \eqref{E4.3}
and \eqref{E4.4} we infer the boundedness of $\{\varphi_k\}_{k = 1}^\infty$ in $Y_{-\lambda} \cap C_{-\lambda}(\bar Q)$. From the definition of $\psi_k$  we deduce the boundedness of $\{\psi_k\}_{k = 1}^\infty$ in $Y_{-\lambda} \cap C_{-\lambda}(\bar Q)$. It is enough to prove $\lim_{k \to \infty}(\|\psi_k\|_{L^2(0,T;H^1(\Omega))} + \|\psi_k\|_{C(\bar Q_T)}) = 0$ and apply Lemma \ref{L2.1} to deduce \eqref{E4.5}. From the boundedness of $\{\psi_k\}_{k = 1}^\infty$ in $C_{-\lambda}(\bar Q)$ we infer the existence of a constant $C_1$ such that $\e^{\frac{\lambda}{2}t}|\psi_k(x,t)| \le C_1$ for all $k \ge 1$ and all $(x,t) \in \bar Q$. Then, given $\varepsilon > 0$ we select $T_\varepsilon < \infty$ such that $\e^{-\frac{\lambda}{2}T_\varepsilon} < \frac{\varepsilon}{C_1}$. Hence, we have
\[
\|\psi_k(t)\|_{L^\infty(\Omega)} \le C_1\e^{-\frac{\lambda}{2}t} < \varepsilon\quad \forall t \ge T_\varepsilon\text{ and } \forall k \ge 1.
\]
For every $T < \infty$, we get from \eqref{E4.10}
\begin{align*}
&\|\psi_k\|_{L^\infty(Q_T)} + \|\psi_k\|_{L^2(0,T;H^1(\Omega))} \le C\Big(\|\psi_k(T)\|_{L^\infty(\Omega)}\\
&+ \Big\|\e^{-\lambda t}\Big[(h_k\! -\! h) + [f'(y_u)\! -\! f'(y_{u_k})]\e^{\lambda t}\varphi_k\Big]\Big\|_{L^p(0,T;L^2(\Omega))}\Big).
\end{align*}
By assumption, we know that $\e^{-\lambda t}(h_k - h) \to 0$ in $L^p(0,T;L^2(\Omega))$. Moreover, \eqref{E2.18} and \eqref{E3.2} imply $\|y_{u_k} - y_u\|_{L^\infty_{\lambda_c}(Q)} \to 0$ and, hence, we have $\|y_{u_k} - y_u\|_{L^\infty(Q_T)} \to 0$ and $\|f'(y_{u_k}) - f'(y_u)\|_{L^\infty(Q_T)} \to 0$. Therefore, using the boundedness of $\{\varphi_k\}_{k = 1}^\infty$ in $C(\bar Q)$ we deduce $\|[f'(y_{u_k}) - f'(y_u)]\varphi_k\|_{L^\infty(Q_T)} \to 0$ for every $T < \infty$. Thus, we get
\[
\limsup_{k \to \infty}\Big(\|\psi_k\|_{L^\infty(Q_T)} + \|\psi_k\|_{L^2(0,T;H^1(\Omega))}\Big) \le C\varepsilon\quad \forall T > 0.
\]
Indeed, if $T \ge T_\varepsilon$ the inequality follows from the definition of $T_\varepsilon$. If $T < T_\varepsilon$, we notice that $\|\psi_k\|_{L^\infty(Q_T)} + \|\psi_k\|_{L^2(0,T;H^1(\Omega))} \le \|\psi_k\|_{L^\infty(Q_{T_\varepsilon})} + \|\psi_k\|_{L^2(0,T_\varepsilon;H^1(\Omega))}$. Since $\varepsilon > 0$ is arbitrarily small, the convergence $\lim_{k \to \infty}(\|\psi_k\|_{L^2(0,T;H^1(\Omega))} + \|\psi_k\|_{C(\bar Q_T)}) = 0$ follows for every $T < \infty$. Now, we apply Lemma \ref{L2.1} first with $X = H^1(\Omega)$ and $r = 2$, and later with $X = L^\infty(\Omega)$ and $r = \infty$ to establish the convergence $\psi_k \to 0$ in $Y_{-\hat\lambda} \cap C_{-\hat\lambda}(\bar Q)$ for every $\hat\lambda < \lambda$.
\end{proof}

The following corollary is essential in the first- and second-order analysis. We should also mention that the assumption that $p \le 6$ plays a key role in this corollary.  For the rest of the paper we make the following assumption on $\sigma_c$\vspace{2mm}

\begin{corollary}
For every $u \in L^p_{\lambda_c}(0,\infty;L^2(\omega))$, the adjoint state equation
\begin{equation}
\left\{
\begin{array}{l}
\displaystyle-\frac{\partial\varphi}{\partial t} + A\varphi + f'(y_u)\varphi = \e^{-\sigma_s t}(y_u - y_d)\ \mbox{ in } Q,\\[0.5ex] \partial_{n_A}\varphi = 0 \ \mbox{ on } \Sigma,\   \lim_{t \to \infty}\e^{-\Lambda_f t}\|\varphi(t)\|_{L^2(\Omega)} = 0.
\end{array}
\right.
\label{E4.11}
\end{equation}
has a unique solution $\varphi_u \in C_{-\sigma_s}(\bar Q) \cap L^2_{-\sigma_s}(0,\infty;H^1(\Omega))$. Moreover, if $u_k \stackrel{*}{\rightharpoonup} u$ in $L^\infty(0,\infty;L^2(\omega))$, then $\lim_{k \to \infty}\Big(\|\varphi_{u_k} - \varphi_u\|_{Y_{-\lambda}} + \|\varphi_{u_k} - \varphi_u\|_{C_{-\lambda}(\bar Q)}\Big) = 0$ for all $\lambda < \sigma_s$.
\label{C4.1}
\end{corollary}

\begin{proof}
We recall that $y_d \in L^p_{\sigma_s}(0,\infty;L^2(\Omega)) \subset L^2_{\sigma_s}(Q)$. Moreover, from Theorem \ref{T2.1} we also know that $y_u \in Y_{\lambda_c} \subset L^2_{\sigma_s}(Q)$. Hence, from Theorem \ref{T4.2} we infer that \eqref{E4.11} has a unique solution $\varphi_u \in Y_{-\sigma_s}$. Let us prove that $\varphi_u \in C_{-\sigma_s}(\bar Q)$. For this purpose we notice that $y_u \in L^p_{\sigma_s}(0,\infty;L^2(\Omega))$. Indeed, using that $\sigma_s > \bar\lambda + 2\lambda_c \ge 3\lambda_c$ and $p \le 6$, we get with Theorem \ref{T2.1} that
\begin{align*}
&\int_0^\infty\e^{-\sigma_st}\|y_u(t)\|^p_{L^2(\Omega)}\dt \le \int_0^\infty\e^{-3\lambda_ct}\|y_u(t)\|^p_{L^2(\Omega)}\dt\\
&= \int_0^\infty\e^{-(3 - \frac{p - 2}{2})\lambda_ct}\|y_u\|^2_{L^2(\Omega)}\Big(\e^{-\frac{\lambda_c}{2}t}\|y_u(t)\|_{L^2(\Omega)}\Big)^{p - 2}\dt\\
& \le \|y_u\|^2_{L^2_{\lambda_c}(0,\infty;L^2(\Omega))}\|y_u\|^{p - 2}_{L^\infty_{\lambda_c}(0,\infty;L^2(\Omega))} < \infty.
\end{align*}
Then, we have that $y_u - y_d \in L^p_{\sigma_s}(0,\infty;L^2(\Omega))$. Hence, Theorem \ref{T4.2} implies that $\varphi_u \in C_{-\sigma_s}(\bar Q)$.

Now, we assume that $u_k \stackrel{*}{\rightharpoonup} u$ in $L^\infty(0,\infty;L^2(\omega))$. From \eqref{E2.18} we infer that $y_{u_k} \to y_u$ in $Y_\lambda \subset L^2_\lambda(Q)$ for all $\lambda > -2\Lambda_f$.  Let us fix $\lambda \in [3\lambda_c,\sigma_s)$. Arguing as above, we have
\begin{align*}
&\int_0^\infty\e^{-\sigma_st}\|y_{u_k}(t) - y_u(t)\|^p_{L^2(\Omega)}\dt \le \int_0^\infty\e^{-3\lambda_ct}\|y_{u_k}(t) - y_u(t)\|^p_{L^2(\Omega)}\dt\\
& \le \|y_{u_k} - y_u\|^2_{L^2_{\lambda_c}(0,\infty;L^2(\Omega))}\|y_{u_k} - y_u\|^{p - 2}_{L^\infty_{\lambda_c}(0,\infty;L^2(\Omega))}.
\end{align*}
Using Young's inequality, the above estimate leads
\begin{align*}
&\|y_{u_k} - y_u\|_{L^p_{\sigma_s}(0,\infty;L^2(\Omega))} \le \|y_{u_k} - y_u\|^{\frac{2}{p}}_{L^2_{\lambda_c}(Q)}\|y_{u_k} - y_u\|^{\frac{p - 2}{p}}_{L^\infty_{\lambda_c}(0,\infty;L^2(\Omega))}\\
&\le \frac{2}{p}\|y_{u_k} - y_u\|_{L^2_{\lambda_c}(Q)} + \frac{p - 2}{p}\|y_{u_k} - y_u\|_{L^\infty_{\lambda_c}(0,\infty;L^2(\Omega))} \le \|y_{u_k} - y_u\|_{Y_{\lambda_c}} \stackrel{k \to \infty}{\longrightarrow} 0.
\end{align*}
Then, the corollary follows from Theorem \ref{T4.2}.
\end{proof}

Assuming that $u \in L^\infty(0,\infty;L^2(\omega))$ and more regularity for $y_d$ we infer a better estimate for $\varphi_u$.

\begin{corollary}
Assume that $\lambda < 2(\sigma_s + \Lambda_f)$, $y_d \in L^p_{2\sigma_s - \lambda}(0,\infty;L^2(\Omega))$, and $u \in L^\infty(0,\infty;L^2(\omega))$. Then, $\varphi_u \in C_{-\lambda}(\bar Q) \cap Y_{-\lambda}$ and there exists a constant $M_\lambda$ depending on $\lambda$, but independent of $u$, such that
\begin{equation}
\|\varphi_u\|_{Y_{-\lambda}} \le M_\lambda\|y_u - y_d\|_{L^2_{2\sigma_s - \lambda}(Q)} \,\text{and}\, \|\varphi_u\|_{C_{-\lambda}(\bar Q)} \le M_\lambda\|y_u - y_d\|_{L^p_{2\sigma_s - \lambda}(0,\infty;L^2(\omega))}.
\label{E53}
\end{equation}
Further, $\lim_{k \to \infty}\Big(\|\varphi_{u_k} - \varphi_u\|_{Y_{-\hat\lambda}} + \|\varphi_{u_k} - \varphi_u\|_{C_{-\hat\lambda}(\bar Q)}\Big) = 0$ for every $\hat\lambda < \lambda$ as $u_k \stackrel{*}{\rightharpoonup} u$ in $L^\infty(0,\infty;L^2(\omega))$.
\label{C4.2}
\end{corollary}

\begin{proof}
First, we observe that $2\sigma_s - \lambda > -2\Lambda_f$. Hence, we deduce from Theorem \ref{T2.1} that $y_u \in Y_{2\sigma_s - \lambda} \cap L^\infty_{2\sigma_s - \lambda}(Q) \subset L^p_{2\sigma_s - \lambda}(0,\infty;L^2(\Omega))$. We apply Theorem \ref{T4.2} with $h = \e^{(\lambda - \sigma_s)t}(y_u - y_d)$. Then, we have
\begin{align*}
\|h\|_{L^2_\lambda(Q)} &= \Big(\int_0^\infty\e^{-\lambda t}\|h(t)\|^2_{L^2(\Omega)}\dt\Big)^{\frac{1}{2}} = \Big(\int_0^\infty\e^{-(2\sigma_s - \lambda)t}\|y_u(t) - y_d(t)\|^2_{L^2(\Omega)}\dt\Big)^{\frac{1}{2}}\\
& = \|y_u - y_d\|_{L^2_{2\sigma_s - \lambda}(Q)}.
\end{align*}
From Theorem \ref{T4.2} we get that $\varphi_u \in Y_{-\lambda}$ and the first estimate of \eqref{E53} follows from \eqref{E4.3}. We also have
\begin{align*}
&\|h\|_{L^p_{\frac{p}{2}\lambda}(0,\infty;L^2(\Omega))} = \Big(\int_0^\infty\e^{-\frac{p}{2}\lambda t}\|h(t)\|^p_{L^2(\Omega)}\dt\Big)^{\frac{1}{p}}\\
& = \Big(\int_0^\infty\e^{[p(\lambda - \sigma_s) - \frac{p}{2}\lambda]t}\|y_u(t) - y_d(t)\|^p_{L^2(\Omega)}\dt\Big)^{\frac{1}{p}}\\
& \le \Big(\int_0^\infty\e^{-(2\sigma_s - \lambda)t}\|y_u(t) - y_d(t)\|^p_{L^2(\Omega)}\dt\Big)^{\frac{1}{p}} = \|y_u - y_d\|_{L^p_{2\sigma_s - \lambda}(0,\infty;L^2(\Omega))}.
\end{align*}
Using again Theorem \ref{T4.2} and estimate \eqref{E4.4} we infer that $\varphi_u \in C_{-\lambda}(\bar Q)$ and the second inequality of \eqref{E53} holds.

Finally, if $u_k \stackrel{*}{\rightharpoonup} u$ in $L^\infty(0,\infty;L^2(\omega))$, we get $y_{u_k} \to y_u$ in $L^p_{\frac{p}{2}\lambda}(0,\infty;L^2(\Omega))$  from \eqref{E2.18}. Hence, the convergence $\lim_{k \to \infty}\Big(\|\varphi_{u_k} - \varphi_u\|_{Y_{-\hat\lambda}} + \|\varphi_{u_k} - \varphi_u\|_{C_{-\hat\lambda}(\bar Q)}\Big) = 0$ for every $\hat\lambda < \lambda$ follows from Theorem \ref{T4.2}.
\end{proof}

Next, we address the differentiability of the cost functional. We will use the following notation hereafter:
\[
U = L^2_{\sigma_c}(Q_\omega) \cap L^p_{\lambda_c}(0,\infty;L^2(\omega)) \quad \text{ and }\quad H= L^2_{\sigma_c}(Q_\omega) \cap L^2_{\lambda_c}(Q_\omega).
\]
We notice that the embedding $L^p_{\lambda_c}(0,\infty;L^2(\omega)) \subset L^2_{\lambda_c}(Q_\omega)$ implies $U \subset H$. Moreover, we have that $H = L^2_{\sigma_c}(Q_\omega)$ if $\sigma_c \le \lambda_c$ and $H = L^2_{\lambda_c}(Q_\omega)$ if $\sigma_c > \lambda_c$.

\begin{theorem}
The functional $J:U \longrightarrow \mathbb{R}$ is of class $C^2$ and
\begin{align}
&J'(u)v = \int_{Q_\omega}(\varphi_u + \nu\e^{-\sigma_c t}u)v\dx\dt,\label{E4.12}\\
&J''(u)(v_1,v_2) = \int_Q[\e^{-\sigma_s t} - \varphi_u f''(y_u)]z_{v_1}z_{v_2}\dx\dt + \nu\int_{Q_\omega}\e^{-\sigma_c t}v_1v_2\dx\dt,\label{E4.13}
\end{align}
where $z_{v_i} = G'(u)v_i$, $i = 1, 2$, and $\varphi_u \in Y_{-\sigma_s} \cap C_{-\sigma_s}(\bar Q)$ is the solution of the adjoint state equation \eqref{E4.11}.
\label{T4.3}
\end{theorem}

\begin{proof}
We define $F_s:L^p_{\lambda_c}(0,\infty;L^2(\omega)) \longrightarrow \mathbb{R}$ and $F_c:L^2_{\sigma_c}(Q_\omega) \longrightarrow \mathbb{R}$ by
\[
F_s(u) = \frac{1}{2}\int_0^\infty\e^{-\sigma_s t}\|y(t) - y_d(t)\|^2_{L^2(\Omega)}\dt\ \text{ and }\ F_c(u) = \frac{\nu}{2}\int_0^\infty\e^{-\sigma_c t}\|u(t)\|^2_{L^2(\omega)}\dt.
\]
Since $G:L^p_{\lambda_c}(0,\infty;L^2(\omega)) \longrightarrow Y_{\sigma_s}$ is of class $C^2$ we infer from the chain rule that $F_s$ is of class $C^2$ and
\begin{align*}
&J'(u)v = F_s'(u)v + F_c'(u)v = \int_0^\infty\Big(\e^{-\sigma_s t}\int_\Omega(y_u - y_d)G'(u)v\dx + \nu\e^{-\sigma_ct}\int_\omega uv\dx\Big)\dt\\
&\phantom{J'(u)v} = \int_Q\Big(\e^{-\sigma_s t}\int_\Omega(y_u - y_d)z_v + \nu\e^{-\sigma_ct}\int_\omega uv\dx\Big)\dt,\\
&J''(u)(v_1,v_2) = F_s''(u)(v_1,v_2) + F_c''(u)(v_1,v_2)\\
& = \int_0^\infty\!\Big(\e^{-\sigma_s t}\!\!\int_\Omega[G'(u)v_1G'(u)v_2 + (y_u - y_d)G''(u)(v_1,v_2)]\dx + \nu\e^{-\sigma_c t}\!\!\int_\omega v_1v_2\dx\Big)\dt\\
& = \int_0^\infty\Big(\e^{-\sigma_s t}\int_\Omega[z_{v_1}z_{v_2} + (y_u - y_d)z_{v_1v_2}]\dx + \nu\e^{-\sigma_ct}\int_\omega v_1v_2\dx\Big)\dt,
\end{align*}
where $z_v = G'(u)v$, $z_{v_i} = G'(u)v_i$, $i = 1, 2$, and $z_{v_1v_2} = G''(u)(v_1,v_2)$. Now, from Corollary \ref{C4.1} we infer the existence of a unique solution of \eqref{E4.11}, $\varphi_u \in Y_{-\sigma_s} \cap C_{-\sigma_s}(\bar Q)$. Inserting $\varphi_u$ into the above derivatives of $J$, integrating by parts and using  equations \eqref{E3.13} and \eqref{E3.14} satisfied by $z_v$, $z_{v_i}$, and $z_{v_1v_2}$, the identities \eqref{E4.12} and \eqref{E4.13} follow.
\end{proof}

\begin{remark}
i) Since $L^\infty(0,\infty;L^2(\omega)) \subset U$ with continuous embedding, we have that $J:L^\infty(0,\infty;L^2(\omega)) \longrightarrow \mathbb{R}$ is well defined and still of class $C^2$. The expressions of the derivatives are also given by the formulas \eqref{E4.12} and \eqref{E4.13}.

ii) For every $u \in L^\infty(0,\infty;L^2(\omega))$ the functional $J'(u)$ can be extended to a linear continuous form $J'(u):H \longrightarrow \mathbb{R}$ through the expression \eqref{E4.12}. Indeed, we have that
\[
|J'(u)v| \le \Big(\|\varphi_u\|_{L^2_{-\lambda_c}(Q)} + \nu\|u\|_{L^2_{\sigma_c}(Q_\omega)}\Big)\Big(\|v\|_{L^2_{\lambda_c}(Q_\omega)} + \|v\|_{L^2_{\sigma_c}(Q_\omega)}\Big)\ \ \forall v \in H.
\]
We also have through the expression \eqref{E4.13} that $J''(u)$ can be extended to a continuous bilinear mapping $J''(u):H \times H \longrightarrow \mathbb{R}$. This is proved as follows: for $\hat\lambda \in (\lambda_c,\frac{\sigma_s - \bar\lambda}{2})$ and for every $v_1, v_2 \in H$ we have
\begin{align*}
&|J''(u)(v_1,v_2)| \le \|z_{v_1}\|_{L^2_{\lambda_c}(Q)}\|z_{v_2}\|_{L^2_{\lambda_c}(Q)} + \nu\|v_1\|_{L^2_{\sigma_c}(Q_\omega)}\|v_2\|_{L^2_{\sigma_c}(Q_\omega)}\\
& + \|\varphi_u\|_{C_{-(\bar\lambda + 2\hat\lambda)}(\bar Q)}\|f''(y_u)\|_{L^2_{\bar\lambda}(Q)}\|z_{v_1}\|_{L^4_{2\hat\lambda}(Q)}\|z_{v_2}\|_{L^4_{2\hat\lambda}(Q)} \le \nu\|v_1\|_{L^2_{\sigma_c}(Q_\omega)}\|v_2\|_{L^2_{\sigma_c}(Q_\omega)}\\
&+ C\Big(1 + \|\varphi_u\|_{C_{-\sigma_s}(\bar Q)}\|f''(y_u)\|_{L^2_{\bar\lambda}(Q)}\Big)\|v_1\|_{L^2_{\lambda_c}(Q_\omega)}\|v_2\|_{L^2_{\lambda_c}(Q_\omega)},
\end{align*}
where we have used \eqref{E3.11}, Corollary \ref{C4.1}, and  the estimates given in Lemma \ref{L4.1} below for $\|z_{v_i}\|_{L^4_{2\hat\lambda}(Q)}$.
\label{R4.1}
\end{remark}

\begin{lemma}
Assume that $u \in L^\infty(0,\infty;L^2(\omega))$ and $v \in L^2_\lambda(Q_\omega)$ with $\lambda > -2\Lambda_f$. Then, the solution $z_v \in Y_\lambda$ of equation \eqref{E3.13} belongs to $L^4_{2\hat\lambda}(Q)$ for every $\hat\lambda > \lambda$ and there exists a constant $M_{\hat\lambda}$ independent of $u$ and $v$ such that
\begin{equation}
\|z_v\|_{L^4_{2\hat\lambda}(Q)} \le M_{\hat\lambda}\|v\|_{L^2_\lambda(Q_\omega)}.
\label{E4.14}
\end{equation}
\label{L4.1}
\end{lemma}

\begin{proof}
The existence and uniqueness of $z_v$ follows from Lemma \ref{L3.2}. Let us define $w = \e^{-\frac{\hat\lambda}{2}t}z_v$ for $\hat\lambda > \lambda$. Then for every $T < \infty$, $w$ satisfies the following equation
\begin{equation}
\left\{
\begin{array}{l}
\displaystyle\frac{\partial w}{\partial t} + Aw + \Big(\frac{\hat\lambda}{2} + f'(y_u)\Big)w = \e^{-\frac{\hat\lambda}{2}t}v\chi_\omega \mbox{ in } Q_T,\\[0.5ex] \partial_{n_A}w = 0 \ \mbox{ on } \Sigma_T,\   w(0) = 0 \ \mbox{ in } \Omega.
\end{array}
\right.
\label{E4.15}
\end{equation}

{\em Step I: dimension $n = 2$.} Using the classical Gagliardo inequality
\[
\|\phi\|_{L^4(\Omega)} \le C\|\phi\|^{\frac{1}{2}}_{L^2(\Omega)}\|\phi\|^{\frac{1}{2}}_{H^1(\Omega)} \quad \forall \phi \in H^1(\Omega),
\]
the following estimate is deduced
\begin{align*}
&\|w\|_{L^4(Q_T)} \le C\Big(\|w\|_{L^\infty(0,T;L^2(\Omega))} + \|w\|_{L^2(0,T;H^1(\Omega))}\Big)\\
& \le M_0\|\e^{-\frac{\hat\lambda}{2}t}v\|_{L^2(\omega \times (0,T))} \le  M_0\|v\|_{L^2_{\hat\lambda}(Q_\omega)} \le M_0\|v\|_{L^2_\lambda(Q_\omega)}.
\end{align*}
Since $T$ is arbitrary and $\|z_v\|_{L^4_{2\hat\lambda}(Q)} = \|w\|_{L^4(Q)}$, inequality \eqref{E4.14} follows for $n = 2$.

{\em Step II: dimension $n = 3$.} Let us assume that $v \in L^p_\lambda(0,\infty;L^2(\omega))$. Later we will remove this assumption. Using  $f'(y_u) \in L^\infty(Q_T)$ and $\frac{\lambda}{2} + f'(y_u) \ge \frac{\lambda}{2} + \Lambda_f > 0$, we infer that $w \in H^1(Q_T) \cap L^\infty(Q_T)$. The reader is referred to \cite[Proposition III-2.5]{Showalter1997} for the $H^1(Q)$ regularity and \cite{Casas-Kunisch2025} for the $L^\infty(Q_T)$ regularity. We test  equation \eqref{E4.15} with $|w|^{\frac{2}{5}}w$. This is possible because $|w|^{\frac{2}{5}}w \in H^1(Q_T) \cap L^\infty(Q_T)$. We observe that
\begin{align*}
&a(w,|w|^{\frac{2}{5}}w) = \frac{7}{5}\int_\Omega\Big(\sum_{i,j = 1}^n[a_{ij}\partial_{x_i}w\partial_{x_j}w]|w|^{\frac{2}{5}} + a_0|w|^{\frac{2}{5}}w^2\Big)\dx\\
&=\frac{35}{36}\int_\Omega\Big(\sum_{i,j = 1}^na_{ij}\partial_{x_i}|w|^{\frac{6}{5}}\partial_{x_i}|w|^{\frac{6}{5}} + a_0[|w|^{\frac{6}{5}}]^2\Big)\dx \ge \frac{35\Lambda_A}{36}\int_\Omega|\nabla|w|^{\frac{6}{5}}|^2\dx,\\
&\int_\Omega\frac{\partial w}{\partial t}|w|^{\frac{2}{5}}w\dx = \frac{5}{12}\frac{d}{dt}\int_\Omega[|w|^{\frac{6}{5}}]^2\dx.
\end{align*}
Using this we get for every $t \in (0,T)$
\begin{align*}
&\frac{5}{12}\int_\Omega[|w(t)|^{\frac{6}{5}}]^2\dx + \min\{\frac{\hat\lambda}{2} + \Lambda_f,\frac{35\Lambda_A}{36}\}\int_0^t\||w|^{\frac{6}{5}}\|^2_{H^1(\Omega)}\ds\\
&\le \int_0^t\int_\omega\e^{-\frac{\hat\lambda}{2}s}v|w|^{\frac{2}{5}}w\dx\ds \le \Big(\int_0^t\int_\omega\e^{-\frac{10\hat\lambda}{13}s}|v|^{\frac{20}{13}}\dx\ds\Big)^{\frac{13}{20}}\Big(\int_{Q_T}|w|^4\dx\dt\Big)^{\frac{7}{20}}\\
&\le \Big(\int_0^T\!\int_\omega\e^{\frac{10(\lambda - \hat\lambda)}{13}t}\e^{-\frac{10\lambda}{13}t}|v|^{\frac{20}{13}}\dx\dt\Big)^{\frac{13}{20}}\Big(\int_{Q_T}|w|^4\dx\dt\Big)^{\frac{7}{20}}\\
& \le \Big(\frac{3|\omega|}{10(\hat\lambda - \lambda)}\Big)^{\frac{3}{13}}\|v\|_{L^2_\lambda(Q)}\|w\|_{L^4(Q_T)}^{\frac{7}{5}} = C_{\hat\lambda - \lambda}\|v\|_{L^2_\lambda(Q)}\|w\|_{L^4(Q_T)}^{\frac{7}{5}}
\end{align*}
From the Gagliardo inequality \cite[pages 173 and 167]{Boyer-Fabrie2013}
\[
\|\phi\|_{L^{\frac{10}{3}}(\Omega)} \le C'\|\phi\|_{L^2(\Omega)}^{\frac{2}{5}}\|\phi\|_{H^1(\Omega)}^{\frac{3}{5}} \quad \forall \phi \in H^1(\Omega),
\]
we infer that
\begin{align*}
&\|w\|^{\frac{12}{5}}_{L^4(Q_T)} = \||w|^{\frac{6}{5}}\|^2_{L^{\frac{10}{3}}(Q_T)}\\
& \le C''\Big(\||w|^{\frac{6}{5}}\|^2_{L^\infty(0,T;L^2(\Omega))} + \||w|^{\frac{6}{5}}\|^2_{L^2(0,T;H^1(\Omega))}\Big) \le C'''C_{\hat\lambda - \lambda}\|v\|_{L^2_\lambda(Q)}\|w\|_{L^4(Q_T)}^{\frac{7}{5}}.
\end{align*}
This yields $\|w\|_{L^4(Q_T)} \le C'''C_{\hat\lambda - \lambda_c}\|v\|_{L^2_\lambda(Q)}$, and therefore we obtain  $\|z_v\|_{L^2_{2\hat\lambda}(Q)} = \|w\|_{L^4(Q)} \le M_{\hat\lambda}\|v\|_{L^2_\lambda(Q)}$ with $M_{\hat\lambda} = C'''C_{\hat\lambda - \lambda}$.

We finish the proof by removing the assumption $v \in L^p_\lambda(0,\infty;L^2(\omega))$. Given $v \in L^2_\lambda(Q_\omega)$, we define $v_k(x,t) = \proj_{[-k,+k]}(v(x,t))$. Then we have $v_k \in L^\infty(Q_\omega) \subset L^p_\lambda(0,\infty;L^2(\omega))$ and $v_k \to v$ in $L^2_\lambda(Q_\omega)$. Let $z_{v_k}$ be the solution of \eqref{E3.13} associated with $v_k$. From \eqref{E3.7} we get $z_{v_k} \to z_v$ in $Y_\lambda$. From this convergence and estimate \eqref{E4.14} satisfied by every $z_{v_k}$, we infer that $z_{v_k} \rightharpoonup z_v$ in $L^4_{2\hat\lambda}(Q)$ and, consequently, $z_v$ also satisfied \eqref{E4.14}.
\end{proof}

Next derive the first-order optimality conditions satisfied by a local solution of \Pb.
\begin{theorem}
If $\bar u$ is a local solution of \Pb, then there exist $\bar y \in L^\infty_{\lambda}(Q) \cap W_{\lambda}(0,\infty)$ for all $\lambda > -2\Lambda_f$ and $\bar\varphi \in Y_{-\sigma_s} \cap C_{-\sigma_s}(\bar Q)$ such that
\begin{align}
&\left\{
\begin{array}{l}
\displaystyle
\frac{\partial\bar y}{\partial t} + A\bar y + f(\bar y) = g + \bar u\chi_\omega\ \mbox{ in } Q,\\[0.5ex] \partial_{n_A}\bar y = 0 \ \mbox{ on } \Sigma,\   \bar y(0) = y_0 \ \mbox{ in } \Omega,
\end{array}
\right.\label{E4.16}\\\emph{\textbf{}}
&\left\{
\begin{array}{l}
\displaystyle-\frac{\partial\bar\varphi}{\partial t} + A\bar\varphi + f'(\bar y)\bar\varphi = \e^{-\sigma_st}(\bar y - y_d)\ \mbox{ in } Q,\\[0.5ex] \partial_{n_A}\bar\varphi = 0 \ \mbox{ on } \Sigma,\   \lim_{t \to \infty}\e^{-\Lambda_f t}\|\bar\varphi(t)\|_{L^2(\Omega)} = 0,
\end{array}
\right. \label{E4.17}\\
&\int_{Q_\omega}(\bar\varphi + \nu\e^{-\sigma_ct}\bar u)(u - \bar u)\dx\dt \ge 0\quad \forall u \in \Uad. \label{E4.18}
\end{align}
In addition, if $y_d \in L^p_{2\sigma_s - \lambda}(0,\infty;L^2(\Omega))$ then we have $\bar\varphi \in Y_{-\lambda} \cap C_{-\lambda}(\bar Q)$ for every $\lambda < 2(\sigma_s + \Lambda_f)$.
\label{T4.4}
\end{theorem}
This theorem is a straightforward consequence of the convexity of the set $\Uad$ and the expression \eqref{E4.12} for $J'$. It suffices to use $J'(\bar u)(u - \bar u) \ge 0$ for all $u \in \Uad$. For the last assertion, we utilize Corollary \ref{C4.2}.

\begin{corollary}
Let $\bar\varphi \in Y_{-\sigma_s} \cap C_{-\sigma_s}(\bar Q)$ and $\bar u \in \Uad$ satisfy \eqref{E4.17} and \eqref{E4.18}. If $\K$ is given by \eqref{E1.2},  the following properties hold for every $t \in (0,\infty)$
\begin{align}
&\int_\omega(\bar\varphi(t) + \nu\e^{-\sigma_c t}\bar u(t))(v - \bar u(t))\dx \ge 0\quad \forall v \in B_\gamma, \label{E4.19}\\
&\text{if } \|\bar u(t)\|_{L^2(\omega)} < \gamma \Rightarrow \bar\varphi(t) + \nu\e^{-\sigma_ct}\bar u(t) = 0\text{ in } \omega,\label{E4.20}\\
&\text{if } \|\bar u(t)\|_{L^2(\omega)} = \gamma \Rightarrow \bar u(x,t) = -\gamma\frac{\bar\varphi(x,t)}{\|\bar\varphi(t)\|_{L^2(\omega)}}.\label{E4.21}
\end{align}
In case  $\K$ is given by \eqref{E1.3}, we have
\begin{equation}
\bar u(x,t) = \proj_{[\alpha,\beta]}\Big(-\frac{1}{\nu}\e^{\sigma_ct}\bar\varphi(x,t)\Big).
\label{E4.22}
\end{equation}
For both cases of $\K$ we have $\bar u \in C_{2\sigma_c - \sigma_s}(\bar Q_\omega)$. In addition, if $y_d \in L^p_{2\sigma_s - \lambda}(0,\infty;L^2(\Omega))$ then we have $\bar u \in C_{2\lambda - \sigma_s}(\bar Q_\omega)$ for all $\lambda < 2(\sigma_s + \Lambda_f)$.
\label{C4.3}
\end{corollary}

\begin{proof}
Let us prove \eqref{E4.19}. Given $v \in B_\gamma$, we denote
\begin{align*}
&I_v = \{t \in (0,\infty) : \int_\omega(\bar\varphi(t) + \nu\e^{-\sigma_ct}\bar u(t))(v - \bar u(t))\dx < 0\},\\
&u_v(x,t) = \left\{\begin{array}{cl} v(x)&\text{if } t \in I_v,\\ \bar u(x,t)&\text{otherwise}.\end{array}\right.
\end{align*}
Then $u_v \in \Uad$ and using \eqref{E4.18} we get
\[
0 \le \int_{Q_\omega}(\bar\varphi + \nu\e^{-\sigma_ct}\bar u)(u_v - \bar u)\dx\dt = \int_{I_v}\int_\omega(\bar\varphi + \nu\e^{-\sigma_ct}\bar u)(v - \bar u)\dx\dt \le 0.
\]
These inequalities are possible if and only if $|I_v| = 0$, which proves \eqref{E4.19}.

From \eqref{E4.19} we infer that
\[
\bar u(t) = \proj_{B_\gamma}\Big(-\frac{1}{\nu}\e^{\sigma_ct}\bar\varphi(t)\Big) =
\left\{\begin{array}{l}\displaystyle -\frac{1}{\nu}\e^{\sigma_ct}\bar\varphi(t)\ \text{ if }\ \|-\frac{1}{\nu}\e^{\sigma_ct}\bar\varphi(t)\|_{L^2(\omega)} < \gamma,\vspace{2mm}\\\displaystyle -\gamma\frac{\frac{1}{\nu}\e^{\sigma_ct}\bar\varphi(t)}{\|-\frac{1}{\nu}\e^{\sigma_ct}\bar\varphi(t)\|_{L^2(\omega)}} = -\gamma\frac{\bar\varphi(t)}{\|\bar\varphi(t)\|_{L^2(\omega)}}\text{ else},\end{array}\right.
\]
where $\proj_{B_\gamma}$ stands for the projection on $B_\gamma$ in the $L^2(\omega)$ norm. The above identities  imply \eqref{E4.20} and \eqref{E4.21}.

Let us prove the continuity of $\bar u$ and choose a sequence  $\{x_k,t_k)\}_{k = 1}^\infty \subset \bar Q_\omega$ converging to a point $(x,t) \in \bar Q_\omega$. If $\|-\frac{1}{\nu}\e^{\sigma_ct}\bar\varphi(t)\|_{L^2(\omega)} < \gamma$, using the continuity of $\bar\varphi$ in $\bar Q$, we deduce the existence of an integer $k_0 \ge 1$ such that $\|-\frac{1}{\nu}\e^{\sigma_ct_k}\bar\varphi(t_k)\|_{L^2(\omega)} < \gamma$,  for every $k \ge k_0$. Therefore, the above representation of $\bar u$ leads to
\[
\lim_{k \to \infty}[\bar u(x_k,t_k) - \bar u(x,t)] = -\frac{1}{\nu}\lim_{k \to \infty}[\e^{\sigma_ct_k}\bar\varphi(x_k,t_k) - \e^{\sigma_ct}\bar\varphi(x,t)] = 0.
\]
Analogously we get that $\bar u(x_k,t_k) \to \bar u(x,t)$ as $k \to \infty$ if $\|-\frac{1}{\nu}\e^{\sigma_ct}\bar\varphi(t)\|_{L^2(\omega)} > \gamma$ . Let us assume that $\|-\frac{1}{\nu}\e^{\sigma_ct}\bar\varphi(t)\|_{L^2(\omega)} = \gamma$. If $\|-\frac{1}{\nu}\e^{\sigma_ct_k}\bar\varphi(t_k)\|_{L^2(\omega)} \ge \gamma$, then
\[
\bar u(x_k,t_k) - \bar u(x,t) = -\gamma\Big[\frac{\bar\varphi(x_k,t_k)}{\|\bar\varphi(t_k)\|_{L^2(\omega)}} - \frac{\bar\varphi(x,t)}{\|\bar\varphi(t)\|_{L^2(\omega)}}\Big] \to 0 \text{ if } k \to \infty.
\]
If $\|-\frac{1}{\nu}\e^{\sigma_ct_k}\bar\varphi(t_k)\|_{L^2(\omega)} < \gamma$, then
\[
\bar u(x_k,t_k) - \bar u(x,t) = -\frac{1}{\nu}[\e^{\sigma_ct_k}\bar\varphi(x_k,t_k) - \e^{\sigma_ct}\bar\varphi(x,t)]\to 0 \text{ if } k \to \infty.
\]
In all cases we have  $\displaystyle\lim_{k \to \infty}\bar u(x_k,t_k) = \bar u(x,t)$, which proves the continuity of $\bar u$ in $\bar Q_\omega$. Moreover, from \eqref{E4.20} we infer that $\e^{(\frac{\sigma_s}{2}-\sigma_c)t}|\bar u(x,t)| = \frac{1}{\nu}\e^{\frac{\sigma_s}{2}t}|\bar\varphi(x,t)|$ if $\|\bar u(t)\|_{L^2(\omega)} < \gamma$. If $\|\bar u(t)\|_{L^2(\omega)} = \gamma$ then \eqref{E4.21} implies that
\[
\e^{(\frac{\sigma_s}{2}-\sigma_c)t}|\bar u(x,t)| = \e^{(\frac{\sigma_s}{2}-\sigma_c)t}\gamma\frac{\frac{1}{\nu}\e^{\sigma_ct}|\bar\varphi(x,t)|}{\|-\frac{1}{\nu}\e^{\sigma_ct}\bar\varphi(t)\|_{L^2(\omega)}} \le \frac{1}{\nu}\e^{\frac{\sigma_s}{2}t}|\bar\varphi(x,t)|.
\]
Therefore, the property $\e^{\frac{\sigma_s}{2}t}\bar\varphi \in C(\bar Q)$ implies that $\bar u \in C_{2\sigma_c - \sigma_s}(\bar Q_\omega)$.

If $\K$ is given by \eqref{E1.3}, the identity \eqref{E4.22} is well known to be a straightforward consequence of \eqref{E4.18}. Moreover, using \eqref{E4.22} and the continuity of $\bar\varphi$ we immediately get the continuity of $\bar u$.

If $y_d \in L^p_{2\sigma_s - \lambda}(0,\infty;L^2(\Omega))$, the we use Corollary \ref{C4.2} to derive the additional regularity of $\bar u$.
\end{proof}

\section{Second-order optimality conditions}
\label{S5}

In this section, we prove necessary and sufficient second-order optimality conditions assuming that $\K$ is given by \eqref{E1.2} or \eqref{E1.3}. We split it into two parts according to the two types of control constraints. The analysis is inspired in the paper \cite{Casas-Kunisch2023C}.

\subsection{Case I: $\K = B_\gamma = \{v \in L^2(\omega) : \|v\|_{L^2(\omega)} \le \gamma\}$}
\label{S5.1}
In this case, we define the Lagrangian function associated with \Pb as follows
\begin{align*}
&\mathcal{L}:U \times L_{-\sigma_s}^\infty(0,\infty) \longrightarrow \mathbb{R}\\
&\mathcal{L}(u,\mu) = J(u) + \frac{1}{2\gamma}\int_0^\infty\mu(t)\|u(t)\|^2_{L^2(\omega)}\dt.
\end{align*}
We know that $J:U \longrightarrow \mathbb{R}$ is of class $C^2$. Moreover, for every $\mu \in L_{-\sigma_s}^\infty(0,\infty)$ the quadratic form
\begin{align*}
&F: L^2_{\lambda_c}(Q_\omega) \times L^2_{\lambda_c}(Q_\omega) \longrightarrow \mathbb{R}\\
&F(v_1,v_2) = \int_0^\infty\mu(t)\int_\omega v_1(t)v_2(t)\dx \dt
\end{align*}
is continuous:
\begin{align*}
&|F(v_1,v_2)| = \Big|\int_0^\infty\e^{\lambda_c t}\mu(t)\int_\omega \e^{-\frac{\lambda_c}{2}t}v_1(t)\e^{-\frac{\lambda_c}{2}t}v_2(t)\dx \dt\Big|\\
& \le \|\mu\|_{L^\infty_{-2\lambda_c}(0,\infty)}\|v_1\|_{L^2_{\lambda_c}(Q_\omega)}\|v_2\|_{L^2_{\lambda_c}(Q_\omega)} \le  \|\mu\|_{L^\infty_{-\sigma_s}(0,\infty)}\|v_1\|_{L^2_{\lambda_c}(Q_\omega)}\|v_2\|_{L^2_{\lambda_c}(Q_\omega)}.
\end{align*}
Therefore, $\mathcal{L}$ is of class $C^2$ and we have the following expressions
\begin{align}
&\frac{\partial\mathcal{L}}{\partial u}(\mu,u)v = \int_{Q_\omega}[\varphi_u + (\nu\e^{-\sigma_ct} + \frac{1}{\gamma}\mu)u]v\dx\dt,\label{E5.1}\\
&\frac{\partial^2\mathcal{L}}{\partial u^2}(\mu,u)(v_1,v_2)\! =\!\! \int_Q\![\e^{-\sigma_s t}\! - \varphi_u f''(y_u)]z_{v_1}z_{v_2}\dx\dt\! +\!\! \int_{Q_\omega}\![\nu\e^{-\sigma_c t}\! +\! \frac{1}{\gamma}\mu]v_1v_2\dx\dt.\label{E5.2}
\end{align}

\begin{remark}
Linear and bilinear mappings $\frac{\partial\mathcal{L}}{\partial u}(\mu,u)$ and $\frac{\partial^2\mathcal{L}}{\partial u^2}(\mu,u)$ can be extended to continuous linear and bilinear forms on $H$ and $H \times H$, respectively. This is an immediate consequence of Remark \ref{R4.1} and the aforementioned property of the bilinear form $F$.
\label{R5.1}
\end{remark}

Given $\bar u \in \Uad$ satisfying the first-order optimality conditions \eqref{E4.17}--\eqref{E4.19}, we define the multiplier $\bar\mu(t) = \|\bar\varphi(t) + \nu\e^{-\sigma_c t}\bar u(t)\|_{L^2(\omega)}$ and the sets
\[
I_{\bar u} = \{t \in (0,\infty) : \|\bar u(t)\|_{L^2(\omega)} = \gamma\}\text{ and } I^+_{\bar u} = \{t \in I_{\bar u} : \bar\mu(t) \neq 0\}.
\]

Observe that the continuity of $\bar\varphi$ and $\bar u$ implies the continuity of $\bar\mu$. Moreover, using  $\bar\varphi, \e^{-\sigma_c t}\bar u \in C_{-\sigma_s}(\bar Q_\omega)$ we obtain  $\bar\mu \in C_{-\sigma_s}[0,\infty)$. We also observe that $\bar\mu(t) \ge 0$ and that \eqref{E4.20} yields $\bar\mu(t)[\|\bar u(t)\|_{L^2(\gamma)} - \gamma] = 0$. From these properties and the following lemma we infer that $\bar\mu$ is a Lagrange multiplier associated with the control constraint.

\begin{lemma}
The identity $\frac{\partial\mathcal{L}}{\partial u}(\bar u,\bar\mu)v = 0$ holds for all $v \in H$.
\label{L5.1}
\end{lemma}

\begin{proof}
First we prove that
\begin{equation}
\|\bar\varphi(t)\|_{L^2(\omega)} = \nu\gamma\e^{-\sigma_ct} + \bar\mu(t)\quad \forall t \in I^+_{\bar u}.
\label{E5.3}
\end{equation}
Using \eqref{E4.19} we get for every $t \in I^+_{\bar u}$:
\begin{align*}
&\|\bar u(t)\|_{L^2(\omega)}\|\bar\varphi(t) + \nu\e^{-\sigma_ct}\bar u(t)\|_{L^2(\omega)} = \gamma\|\bar\varphi(t) + \nu\e^{-\sigma_ct}\bar u(t)\|_{L^2(\omega)}\\
&= \sup_{v \in B_\gamma}\Big(-\int_\omega(\bar\varphi(t) + \nu\e^{-\sigma_ct}\bar u(t))v\dx\Big) \le -\int_\omega(\bar\varphi(t) + \nu\e^{-\sigma_ct}\bar u(t))\bar u(t)\dx\\
&\le \|\bar u(t)\|_{L^2(\omega)}\|\bar\varphi(t) + \nu\e^{-\sigma_ct}\bar u(t)\|_{L^2(\omega)}.
\end{align*}
Thus we have
\[
\|\bar u(t)\|_{L^2(\omega)}\|\bar\varphi(t) + \nu\e^{-\sigma_ct}\bar u(t)\|_{L^2(\omega)} = -\int_\omega(\bar\varphi(t) + \nu\e^{-\sigma_ct}\bar u(t))\bar u(t)\dx.
\]
This implies the existence of a constant $c_t < 0$ such that $\bar u(t) = c_t(\bar\varphi(t) + \nu\e^{-\sigma_ct}\bar u(t))$ in $\omega$. Taking the $L^2(\omega)$ norm on both sides of this identity, we obtain
\[
|c_t| = \frac{\|\bar u(t)\|_{L^2(\omega)}}{\|\bar\varphi(t) + \nu\e^{-\sigma_ct}\bar u(t)\|_{L^2(\omega)}} = \frac{\gamma}{\bar\mu(t)} \Rightarrow \bar u(x,t) = -\gamma\frac{\bar\varphi(t) + \nu\e^{-\sigma_ct}\bar u(t)}{\bar\mu(t)}.
\]
The last identity can be written in the form $(\nu\gamma\e^{-\sigma_ct} + \bar\mu(t))\bar u(t) = -\gamma\bar\varphi(t)$. Taking the $L^2(\omega)$ at both sides of this identity, \eqref{E5.3} follows. By \eqref{E5.1}, \eqref{E4.21}, and \eqref{E5.3} we get
\begin{align*}
\frac{\partial\mathcal{L}}{\partial u}(\bar u,\bar\mu)v &= \int_{Q_\omega}[\bar\varphi + (\nu\e^{-\sigma_ct} + \frac{1}{\gamma}\bar\mu)u]v\dx\dt = \int_{I^+_{\bar u}}\int_\omega[\bar\varphi + (\nu\e^{-\sigma_ct} + \frac{1}{\gamma}\bar\mu)u]v\dx\dt \\
&= \int_{I^+_{\bar u}}\int_\omega[\bar\varphi -\gamma (\nu\e^{-\sigma_ct} + \frac{1}{\gamma}\bar\mu)\frac{\bar\varphi}{\|\bar\varphi(t)\|_{L^2(\omega)}}]v\dx\dt\\
& = \int_{I^+_{\bar u}}\int_\omega[\|\bar\varphi(t)\|_{L^2(\omega)} - (\nu\gamma\e^{-\sigma_ct} + \bar\mu)]\frac{\bar\varphi}{\|\bar\varphi(t)\|_{L^2(\omega)}}v\dx\dt = 0.
\end{align*}
\end{proof}

To carry out the second-order analysis, we introduce the cone of critical directions:
\[
C_{\bar u} = \Big\{v \in H : \int_\omega\bar u(t)v(t)\dx \left\{\begin{array}{cl} \le 0 &\text{if } t \in I_{\bar u} \setminus I^+_{\bar u}\\ = 0 &\text{if } t \in I^+_{\bar u}\end{array}\right.\Big\}.
\]

\begin{remark}
i) We notice that $\bar\mu(t) = \|\bar\varphi(t) + \nu\e^{-\sigma_c t}\bar u(t)\|_{L^2(\omega)}$ is the unique Lagrange multiplier associated with the control constraint. Indeed, if $\tilde\mu$ is another Lagrange multiplier, we infer from the identities $\frac{\partial\mathcal{L}}{\partial u}(\bar u,\tilde\mu)v = \frac{\partial\mathcal{L}}{\partial u}(\bar u,\bar\mu)v = 0$ for all $v \in H$ that $\bar u(x,t)(\bar\mu(t) - \tilde\mu(t)) = 0$ for almost all $(x,t) \in Q_\omega$. If $\|\bar u(t)\|_{L^2(\omega)} > 0$, then this equality implies that $\tilde\mu(t) = \bar\mu(t)$. If $\|\bar u(t)\|_{L^2(\omega)} = 0$, then we have that $t \not\in I_{\bar u}$ and consequently $\tilde\mu(t) = \bar\mu(t) = 0$.

ii) Let us observe that
\[
\int_{Q_\omega} \bar\mu(t)\bar u(x,t)v(x,t)\dx\dt = \int_{I^+_{\bar u}}\bar\mu(t)\int_\omega\bar u(x,t)v(x,t)\dx\dt  = 0\quad \forall v \in C_{\bar u}.
\]
Therefore, we deduce from this equality, \eqref{E5.1}, and Lemma \ref{L5.1} that
\[
J'(\bar u)v = \frac{\partial\mathcal{L}}{\partial u}(\bar u,\bar\mu)v = 0\quad \forall v \in C_{\bar u}.
\]
\label{R5.2}
\end{remark}

\begin{theorem}
Let $\bar u$ be a local solution of \Pb, then $\frac{\partial^2\mathcal{L}}{\partial u^2}(\bar u,\bar\mu)v^2 \ge 0$ for all $v \in C_{\bar u}$.
\label{T5.1}
\end{theorem}

\begin{proof}
Since $\bar u$ is a local minimizer of \Pb, there exists $\varepsilon > 0$ such that $J(\bar u) \le J(u)$ for all $u \in \Uad \cap B_\varepsilon(\bar u)$, where $B_\varepsilon(\bar u) = \{u \in L^2_{\sigma_c}(Q_\omega) : \|u - \bar u\|_{L^2_{\sigma_c}(Q_\omega)} < \varepsilon\}$. Let $v \in C_{\bar u} \cap L^\infty(0,\infty;L^2(\omega))$. The assumption $v \in L^\infty(0,\infty;L^2(\omega))$ will be removed later. Let us fix an integer
\[
k_0 > \max\Big\{\sqrt{\frac{2\max\{\|\bar u\|_{L^2_{\sigma_c}(Q_\omega)},\gamma\}}{\gamma^4\varepsilon}},\frac{1}{\gamma^2}\Big\},
\]
and set
\[
v_k(x,t) = \left\{\begin{array}{cl}0&\text{if } \displaystyle \gamma^2 - \frac{1}{k} < \|\bar u(t)\|^2_{L^2(\omega)} < \gamma^2\\v(x,t)&\text{otherwise}\end{array}\right.\quad \forall k \ge k_0.
\]
It is obvious that $\{v_k\}_{k \ge k_0} \subset L^\infty(0,\infty;L^2(\omega)) \subset H$. Moreover, the convergence $v_k \to v$ in $H$ follows from Lebesgue's dominated convergence theorem. Here, we define $\|v\|_H = \|v\|_{L^2_{\sigma_c}(Q_\omega)} + \|v\|_{L^2_{\lambda_c}(Q_\omega)}$.

For fixed $k \ge k_0$, we define
\[
\alpha_k = \min\Big\{\frac{\min\{1,\gamma\}\varepsilon}{2\max\{\|v\|_H,\|v\|_{L^\infty(0,\infty;L^2(\omega))}\}},\frac{\gamma - \sqrt{\gamma^2 - \frac{1}{k}}}{\|v\|_{L^\infty(0,\infty;L^2(\omega))}}\Big\}
\]
and $\phi_k:(-\alpha_k,+\alpha_k) \longrightarrow L^\infty(0,\infty;L^2(\omega))$ by
\[
\phi_k(\rho) = \sqrt{1 - \frac{\rho^2}{\gamma^2}\|v_k(t)\|^2_{L^2(\omega)}}\,\bar u + \rho v_k.
\]

By definition of $\alpha_k$ we have $\frac{\rho^2}{\gamma^2}\|v_k(t)\|^2_{L^2(\omega)} < 1$ for all $k \ge k_0$, $|\rho| < \alpha_k$, and almost all $t \in (0,\infty)$. Moreover, $|\phi_k(\rho)| \le |\bar u| + \frac{\varepsilon}{2\|v\|_{L^2_{\sigma_c}(Q_\omega)}}|v| \in L^\infty(0,\infty;L^2(\omega))$. Hence, the mapping $\phi_k$ is well defined and is of class $C^\infty$. Let us prove some properties of this mapping.

{\em I - $\phi_k(\rho) \in \Uad$ for all $\rho \in [0,+\alpha_k)$.} Let us set $u_\rho = \phi_k(\rho)$. Then, for almost all $t \in (0,\infty)$
\begin{align}
\|u_\rho(t)\|^2_{L^2(\omega)} &= \Big[1 - \frac{\rho^2}{\gamma^2}\|v_k(t)\|^2_{L^2(\omega)}\Big]\|\bar u(t)\|^2_{L^2(\omega)} + \rho^2\|v_k(t)\|^2_{L^2(\omega)}\notag\\
& + 2\rho \sqrt{1 - \frac{\rho^2}{\gamma^2}\|v_k(t)\|^2_{L^2(\omega)}}\int_\omega\bar u(t)v_k(t)\dx.\label{E5.4}
\end{align}
In the case $t \in I_{\bar u}$, we have $v_k(t) = v(t)$. Then, using $v \in C_{\bar u}$ we deduce that the last integral of the above inequality is less than or equal to zero and, consequently, \eqref{E5.4} leads to $\|u_\rho(t)\|^2_{L^2(\omega)} \le \gamma^2$. If $\gamma^2 - \frac{1}{k} < \|\bar u(t)\|^2_{L^2(\omega)} < \gamma^2$, then we have $v_k(t) = 0$ by definition and, hence, \eqref{E5.4} implies that $\|u_\rho(t)\|^2_{L^2(\omega)} \le \gamma^2$. Finally, we assume that $\|\bar u(t)\|^2_{L^2(\omega)} \le \gamma^2 - \frac{1}{k}$. Then, from the definition of $\alpha_k$ we infer
\[
\|u_\rho(t)\|_{L^2(\omega)} \le \sqrt{\gamma^2 - \frac{1}{k}} + \alpha_k\|v\|_{L^\infty(0,\infty;L^2(\omega))} \le \gamma.
\]

{\em II - $\|\phi_k(\rho) - \bar u\|_{L^2_{\sigma_c}(Q_\omega)} < \varepsilon$.} From the definition of $\phi_k$ we get
\begin{align*}
\|\phi_k(\rho) - \bar u\|_{L^2_{\sigma_c}(Q_\omega)} &\le \left|1 - \sqrt{1 - \frac{\rho^2}{\gamma^2}\|v_k(t)\|^2_{L^2(\omega)}}\right|\|\bar u\|_{L^2(Q_\omega)} + |\rho|\|v_k\|_{L^2_{\sigma_c}(Q_\omega)}\\
&\le \frac{\alpha_k^2}{\gamma^2}\|v\|^2_{L^\infty(0,\infty;L^2(\omega))}\|\bar u\|_{L^2_{\sigma_c}(Q_\omega)} + \alpha_k\|v\|_{L^2_{\sigma_c}(Q_\omega)}.
\end{align*}
From the definition of $\alpha_k$ and $k \ge k_0 > \frac{1}{\gamma^2}$ we obtain
\[
\alpha_k \le \frac{\gamma - \sqrt{\gamma^2 - \frac{1}{k}}}{\|v\|_{L^\infty(0,\infty;L^2(\omega))}} \le \frac{1}{k\gamma\|v\|_{L^\infty(0,\infty;L^2(\omega))}}.
\]
Moreover, $\alpha_k \le \frac{\varepsilon}{2\|v\|_H} \le \frac{\varepsilon}{2\|v\|_{L^2_{\sigma_c}(Q_\omega)}}$ holds. Then, we have
\[
\|\phi_k(\rho) - \bar u\|_{L^2_{\sigma_c}(Q_\omega)} \le \frac{\|\bar u\|_{L^2_{\sigma_c}(Q_\omega)}}{k^2\gamma^4} + \frac{\varepsilon}{2} < \varepsilon.
\]
The last inequality is consequence of $k \ge k_0 > \sqrt{\frac{2\|\bar u\|_{L^2_{\sigma_c}(Q_\omega)}}{\gamma^4\varepsilon}}$.

Now we define the function $\psi_k:(-\alpha_k,+\alpha_k) \longrightarrow \mathbb{R}$ by $\psi_k(\rho) = J(\phi_k(\rho))$. From the local optimality of $\bar u$ and the established properties of $\phi_k$ we infer that $\psi_k(0) = J(\bar u) \le J(\phi_k(\rho)) = \psi_k(\rho)$ for every $\rho \in [0,+\alpha_k)$. Since $\psi_k$ is of class $C^2$ and $\psi'_k(0) = 0$, then $\psi_k''(0) \ge 0$.  Hence, we get
\begin{align*}
& 0 \le \psi_k''(0) = J''(\phi_k(0))\phi_k'(0)^2 + J'(\phi_k(0))\phi''_k(0) = J''(\bar u)v_k^2 + J'(\bar u)\phi''_k(0)\\
& = \int_Q\Big[\e^{-\sigma_s t} - \bar\varphi f''(\bar y)\Big]z_{v_k}^2\dx\dt + \nu\int_{Q_\omega}\e^{-\sigma_c t}v_k^2\dx\dt\\
& - \frac{1}{\gamma^2}\int_0^\infty\|v_k(t)\|^2_{L^2(\omega)}\int_\omega(\bar\varphi + \nu\e^{-\sigma_c t}\bar u)\bar u\dx\dt.
\end{align*}
Using \eqref{E4.20}, \eqref{E4.21}, and \eqref{E5.3} we obtain that
\begin{align*}
&\int_0^\infty\|v_k(t)\|^2_{L^2(\omega)}\int_\omega(\bar\varphi + \nu\e^{-\sigma_c t}\bar u)\bar u\dx\dt\\
& = \int_{I_{\bar u}}\|v_k(t)\|^2_{L^2(\omega)}\Big(-\int_\omega\gamma\frac{\bar\varphi^2(t)}{\|\bar\varphi(t)\|_{L^2(\omega)}}\dx + \nu\e^{-\sigma_c t}\|\bar u(t)\|^2_{L^2(\omega)}\Big)\dt\\
&= \gamma\int_{I_{\bar u}}\|v_k(t)\|^2_{L^2(\omega)}\Big(-\|\bar\varphi(t)\|_{L^2(\omega)} + \nu\gamma\e^{-\sigma_c t}\Big)\dt\\
&= -\gamma\int_{I_{\bar u}}\|v_k(t)\|^2_{L^2(\omega)}\bar\mu(t)\dt = -\gamma\int_0^\infty\bar\mu(t)\|v_k(t)\|^2_{L^2(\omega)}\dt.
\end{align*}
Inserting this into the above inequality, we infer with \eqref{E5.2}
\[
0 \le \psi_k''(0) = \frac{\partial^2\mathcal{L}}{\partial u}(\bar u,\bar\mu)v_k^2,
\]
and thus the convergence $v_k \to v$ in $H$ implies
\[
\frac{\partial^2\mathcal{L}}{\partial u^2}(\bar u,\bar\mu)v^2 = \lim_{k \to \infty}\frac{\partial^2\mathcal{L}}{\partial u^2}(\bar u,\bar\mu)v^2_k \ge 0.
\]
Finally, we remove the assumption $v \in L^\infty(0,\infty;L^2(\omega))$. Given $v \in C_{\bar u}$, we define $v_k(x,t) = \frac{v(x,t)}{1 + \frac{1}{k}\|v(t)\|_{L^2(\omega)}}$ for every integer $k \ge 1$. Then, we have $\{v_k\}_{k = 1}^\infty \subset L^\infty(0,\infty;L^2(\omega))$ and $v_k \to v$ in $H$. Using $v \in C_{\bar u}$ we get
\[
\int_\omega\bar u(t)v_k(t)\dx = \frac{1}{1 + \frac{1}{k}\|v(t)\|_{L^2(\omega)}}\int_\omega\bar u(t)v(t)\dx \left\{\begin{array}{cl} \le 0&\text{if } t \in I_{\bar u}\\= 0&\text{if } t \in I^+_{\bar u}\end{array}\right.
\]
Therefore, we have $\frac{\partial^2\mathcal{L}}{\partial u^2}(\bar u,\bar\mu)v_k^2 \ge 0$ for all $k \ge 1$. Finally, passing to the limit as $k \to \infty$ we conclude that $\frac{\partial^2\mathcal{L}}{\partial u^2}(\bar u,\bar\mu)v^2 \ge 0$.
\end{proof}

\begin{remark}
If $y_d \in L^p_{-2\Lambda_f}(0,\infty;L^2(\Omega))$, $\sigma_c < \sigma_s + \Lambda_f(q + 2)$, and $u \in L^\infty(0,\infty;L^2(\omega))$, then the linear and bilinear forms $\frac{\partial\mathcal{L}}{\partial u}(\mu,u)$ and $\frac{\partial^2\mathcal{L}}{\partial u^2}(\mu,u)$ can be extended to continuous linear and bilinear forms on $L^2_{\sigma_c}(Q_\omega)$ and $L^2_{\sigma_c}(Q_\omega) \times L^2_{\sigma_c}(Q_\omega)$, respectively. Indeed, first we observe that
\begin{align*}
&\Big|\frac{\partial\mathcal{L}}{\partial u}(u,\mu)v\Big| \le \|\varphi_u\|_{L^2_{-\sigma_c}(Q_\omega)}\|v\|_{L^2_{\sigma_c}(Q_\omega)} + \Big(\nu + \frac{1}{\gamma}\|\mu\|_{L^\infty_{-\sigma_s}(0,\infty)}\Big)\|u\|_{L^2_{\sigma_c}(Q_\omega)}\|v\|_{L^2_{\sigma_c}(Q_\omega)}\\
&\le \Big(\|\varphi_u\|_{L^2_{-\sigma_s}(Q)} + \Big[\nu + \frac{1}{\gamma}\|\mu\|_{L^\infty_{-\sigma_s}(0,\infty)}\Big]\|u\|_{L^2_{\sigma_c}(Q_\omega)}\Big) \|v\|_{L^2_{\sigma_c}(Q_\omega)}\quad \forall v \in L^2_{\sigma_c}(Q_\omega).
\end{align*}

To prove the continuity of the bilinear mapping $\frac{\partial\mathcal{L}}{\partial u}(u,\mu) : L^2_{\sigma_c}(Q_\omega) \times L^2_{\sigma_c}(Q_\omega) \longrightarrow \mathbb{R}$ we define $\eta = \sigma_s + \Lambda_f(q+2) - \max\{\sigma_c,-2\Lambda_f\}$. Since $\sigma_s > -2(q+3)\Lambda_f$ we have that $\eta > 0$. Moreover, we can take $\rho > 2$ large enough so that $-2\Lambda_f < \lambda_c = -2\Lambda_f + \frac{\eta}{\rho(q+1)} < \frac{\sigma_s}{q + 3}$. Now we take $\bar\lambda = (q + 1)\lambda_c$, $\lambda = \max\{-2\Lambda_f,\sigma_c\} + \frac{\eta}{2\rho}$,  and $\hat\lambda = \max\{-2\Lambda_f,\sigma_c\} + \frac{\eta}{\rho}$. Then, we have
\[
2(\sigma_s + \Lambda_f) - \bar\lambda - 2\hat\lambda = \eta(2 - \frac{3}{\rho})  = \delta > 0.
\]
Hence, we have with \eqref{E3.7} and Lemma \ref{L4.1}
\begin{align}
&\Big|\frac{\partial^2\mathcal{L}}{\partial u^2}(u,\mu)(v_1,v_2)\Big| \le \|z_{v_1}\|_{L^2_{\sigma_s}(Q)}\|z_{v_2}\|_{L^2_{\sigma_s}(Q)}\notag\\
&+ \|\varphi_u\|_{L^\infty_{\delta - 2(\sigma_s+\Lambda_f)}(Q)}\|{f''}(y_u)\|_{L^2_{\bar\lambda}(Q)}\|z_{v_1}\|_{L^4_{2\hat\lambda}(Q)}\|z_{v_2}\|_{L^4_{2\hat\lambda}(Q)}\notag\\
&+ \Big(\nu + \frac{1}{\gamma}\|\mu\|_{L^\infty_{-\sigma_s}(0,\infty)}\Big)\|v_1\|_{L^2_{\sigma_c}(Q_\omega)} \|v_2\|_{L^2_{\sigma_c}(Q_\omega)}\notag\\
&\le C\|v_1\|_{L^2_{\sigma_s}(Q_\omega)}\|v_2\|_{L^2_{\sigma_s}(Q_\omega)} + \Big(\nu + \frac{1}{\gamma}\|\mu\|_{L^\infty_{-\sigma_s}(0,\infty)}\Big)\|v_1\|_{L^2_{\sigma_c}(Q_\omega)}\|v_2\|_{L^2_{\sigma_c}(Q_\omega)}\notag\\
&+ M_{\hat\lambda}\|\varphi_u\|_{L^\infty_{\delta - 2(\sigma_s+\Lambda_f)}(Q)}\|f''(y_u)\|_{L^2_{\bar\lambda}(Q)}\|v_1\|_{L^2_{\lambda}(Q_\omega)}\|v_2\|_{L^2_{\lambda}(Q_\omega)}\Big)\notag\\
&\le C'\|v_1\|_{L^2_{\sigma_c}(Q_\omega)}\|v_2\|_{L^2_{\sigma_c}(Q_\omega)}\quad \forall v_1, v_2 \in L^2_{\sigma_c}(Q_\omega),\label{E68}
\end{align}
where we have used that $\sigma_c < \lambda$. In addition we have that the proof of Lemma \ref{L5.1} can be repeated to deduce that $\frac{\partial\mathcal{L}}{\partial u}(\bar u,\bar\mu) v = 0$ for every $v \in L^2_{\sigma_c}(Q_\omega)$.
\label{R5.3}
\end{remark}

\begin{remark}
The condition $\sigma_c <  \sigma_s + \Lambda_f(q + 2)$ is a smallness condition on $\sigma_c$ relative to $\sigma_s$ and properties of the nonlinearity, involving the departure from monotonicity, measured by $\Lambda_f$ and the degree of nonlinearity expressed by $q$. For $\Lambda_f=0$ we simply require that $\sigma_c < \sigma_s$.
\label{R5.4}
\end{remark}

Now, we redefine the cone of critical directions as follows
\[
C_{\bar u} = \Big\{v \in L^2_{\sigma_c}(Q_\omega) : \int_\omega\bar u(t)v(t)\dx \left\{\begin{array}{cl} \le 0 &\text{if } t \in I_{\bar u} \setminus I^+_{\bar u}\\ = 0 &\text{if } t \in I^+_{\bar u}\end{array}\right.\Big\}.
\]

We observe that using Remark \ref{R5.3}, the same proof as for Theorem \ref{T5.1} can be used to deduce that if $\bar u$ is a local minimizer of \Pb, then $\frac{\partial^2\mathcal{L}}{\partial u^2}(\bar u,\bar\mu)v^2 \ge 0$ for all $v$ in this redefined cone.

\begin{theorem}
Assume that $\sigma_c < \sigma_s + \Lambda_f(q + 2)$ and $y_d \in L^p_{-2\Lambda_f}(0,\infty;L^2(\Omega))$. Let $\bar u \in \Uad$ satisfy the first-order optimality conditions \eqref{E4.16}--\eqref{E4.18} and the second-order condition $\frac{\partial^2\mathcal{L}}{\partial u^2}(\bar u,\bar\mu)v^2 > 0$ for every $v \in C_{\bar u} \setminus \{0\}$. Then, there exist $\kappa > 0$ and $\varepsilon > 0$ such that
\begin{equation}
J(\bar u) + \frac{\kappa}{2}\|u - \bar u\|^2_{L^2_{\sigma_c}(Q_\omega)} \le J(u)\ \ \forall u \in \Uad \text{ with } \|u - \bar u\|_{L^2_{\sigma_c}(Q_\omega)} \le \varepsilon.
\label{E5.5}
\end{equation}
\label{T5.2}
\end{theorem}

\begin{proof}
We argue by contradiction and assume that \eqref{E5.5} does not hold. Then, for every integer $k \ge 1$ there exists a control $u_k \in \Uad$ such that
\begin{equation}
\|u_k - \bar u\|_{L^2_{\sigma_c}(Q_\omega)} < \frac{1}{k}\ \text{ and }\ J(u_k) < J(\bar u) + \frac{1}{2k}\|u_k - \bar u\|^2_{L^2_{\sigma_c}(Q_\omega)}.
\label{E5.6}
\end{equation}
We define $\rho_k = \|u_k - \bar u\|_{L^2_{\sigma_c}(Q_\omega)} < \frac{1}{k}$ and $v_k = \frac{1}{\rho_k}(u_k - \bar u)$. Since $\|v_k\|_{L^2_{\sigma_c}(Q_\omega)} = 1$ for every $k$, taking a subsequence, we can assume that $v_k \rightharpoonup v$ in $L^2_{\sigma_c}(Q_\omega)$.

Observe that \eqref{E5.6} implies $\lim_{k \to \infty}\|u_k - \bar u\|_{L^p_{\lambda_c}(0,\infty;L^2(\omega))} = 0$. Indeed, we note that for $\sigma_c \le \lambda_c$
\begin{align*}
&\|u_k - \bar u\|_{L^p_{\lambda_c}(0,\infty;L^2(\omega))}\! \le\! \|u_k - \bar u\|^{\frac{p - 2}{p}}_{L^\infty(0,\infty;L^2(\omega))}\|u_k - \bar u\|^{\frac{2}{p}}_{L^2_{\lambda_c}(Q_\omega)}\\
& \le (2\gamma)^{\frac{p - 2}{p}}\|u_k - \bar u\|^{\frac{2}{p}}_{L^2_{\lambda_c}(Q_\omega)} \le (2\gamma)^{\frac{p - 2}{p}}\|u_k - \bar u\|^{\frac{2}{p}}_{L^2_{\sigma_c}(Q_\omega)} \le \frac{(2\gamma)^{\frac{p - 2}{p}}}{k^{\frac{2}{p}}} \stackrel{k \to \infty}{\longrightarrow} 0.
\end{align*}
If $\lambda_c < \sigma_c$, then we argue as in \eqref{E4.1A} and get that
\begin{align*}
&\|u_k - \bar u\|_{L^p_{\lambda_c}(0,\infty;L^2(\omega))} \le (2\gamma)^{\frac{p - 2}{p}}\|u_k - \bar u\|^{\frac{2}{p}}_{L^2_{\lambda_c}(Q_\omega)}\\
& \le (2\gamma)^{\frac{p - 2}{p}}C_{\lambda_c,\sigma_c}\|u_k - \bar u\|_{L^2_{\sigma_c}(Q_\omega)}^{\frac{\lambda_c}{p\sigma_c}} \le \frac{(2\gamma)^{\frac{p - 2}{p}}C_{\lambda_c,\sigma_c}}{k^{\frac{\lambda_c}{p\sigma_c}}} \stackrel{k \to \infty}{\longrightarrow} 0.
\end{align*}

Then, Theorem \ref{T3.1} yields $y_{u_k} = G(u_k) \to G(\bar u) = \bar y$ in $Y_{\sigma_s}$. The rest of the proof is split into three steps.

{\em Step I - $v \in C_{\bar u}$.} From \eqref{E4.18} and \eqref{E5.1} we infer that
\begin{equation}
0 \le J'(\bar u)v_k = \int_{Q_\omega}(\bar\varphi + \nu\e^{-\sigma_c t}\bar u)v_k\dx\dt \to \int_{Q_\omega}(\bar\varphi + \nu\e^{-\sigma_c t}\bar u)v\dx\dt = J'(\bar u)v.
\label{E5.7}
\end{equation}

Using the mean value theorem and \eqref{E5.6} we get
\[
\int_{Q_\omega}(\varphi_{\theta_k} + \nu \e^{-\sigma_c t}u_{\theta_k})v_k\dx\dt = J'(u_{\theta_k})v_k = \frac{J(u_k) - J(\bar u)}{\rho_k} < \frac{\rho_k}{2k} \to 0,
\]
where $\theta_k \in [0,1]$, $u_{\theta_k} = \bar u + \theta_k(u_k - \bar u)$, and $\varphi_{\theta_k}$ is the adjoint state corresponding to $u_{\theta_k}$. Since $y_{\theta_k}= G(u_{\theta_k}) \to G(\bar u) = \bar y$ in $Y_{\sigma_s}$ and $u_{\theta_k} \stackrel{*}{\rightharpoonup} \bar u$ in $L^\infty(0,\infty;L^2(\omega))$,  we deduce  from Corollary \ref{C4.1} that $\varphi_{\theta_k} \to \bar\varphi$ in $Y_{-\lambda}$ as $k \to \infty$ for every $\lambda < \sigma_s$. Using this fact with $\lambda = \sigma_c$, it is straightforward to pass to the limit in the above expression and to get $J'(\bar u)v \le 0$. This inequality and \eqref{E5.7} imply that $J'(\bar u)v = 0$.

Next, taking into account that $\|u_k(t)\|_{L^2(\omega)} \le \gamma$ for almost all $t > 0$, we have for almost every $t \in I_{\bar u}$
\[
\int_\omega\bar u(t)v_k(t)\dt = \frac{1}{\rho_k}\Big[\int_\omega\bar u(t)u_k(t)\dt - \int_\omega\bar u^2(t)\dt\Big] \le \frac{1}{\rho_k}\gamma\Big[\|u_k(t)\|_{L^2(\omega)} - \gamma\Big] \le 0.
\]
We define the function $\phi \in L^\infty(0,\infty)$ by $\phi(t) = 1$ if $\int_\omega\bar u(t)v(t)\dx > 0$ and 0 otherwise. Then, from the convergence $v_k \rightharpoonup v$ in $L^2_{\sigma_s}(Q_\omega)$ and the fact that $\phi\bar u \in L^2_{\sigma_c}(Q_\omega)$ we infer from the above inequality
\[
\int_{I_{\bar u}}\e^{-\sigma_c t}\phi(t)\int_\omega\bar u(t)v(t)\dx\dt = \lim_{k \to \infty}\int_{I_{\bar u}}\e^{-\sigma_c t}\phi(t)\int_\omega\bar u(t)v_k(t)\dx\dt \le 0.
\]
This is possible if and only if $\int_\omega\bar u(t)v(t)\dx \le 0$ for almost all $t \in I_{\bar u}$. Finally, we prove that this integral is 0 if $t \in I_{\bar u}^+$. For this purpose we use the observation from above that the proof of Lemma \ref{L5.1}, provides  that $\frac{\partial\mathcal{L}}{\partial u}(\bar u,\bar\mu) v = 0$ for every $v \in L^2_{\sigma_c}(Q_\omega)$.
Together with \eqref{E5.1} and the fact that $J'(\bar u)v = 0$ this implies
\[
0 = \frac{\partial\mathcal{L}}{\partial u}(\bar u,\bar\mu)v = J'(\bar u)v + \frac{1}{\gamma}\int_0^\infty\bar\mu(t)\int_\omega\bar u(t)v(t)\dx\dt = \frac{1}{\gamma}\int_{I_{\bar u}}\bar\mu(t)\int_\omega\bar u(t)v(t)\dx\dt,
\]
which implies that $\int_\omega\bar u(t)v(t)\dx = 0$ for almost all $t \in I^+_{\bar u}$, and thus $v \in C_{\bar u}$.

{\em Step II - $\frac{\partial^2\mathcal{L}}{\partial u^2}(\bar u,\bar\mu)v^2 \le 0$.} First we observe that
\begin{align*}
&\int_0^\infty\bar\mu(t)\|u_k(t)\|^2_{L^2(\omega)}\dt =
\int_{I_{\bar u}^+}\bar\mu(t)\|u_k(t)\|^2_{L^2(\omega)}\dt\\
& \le \int_{I_{\bar u}^+}\bar\mu(t)\|\bar u(t)\|^2_{L^2(\omega)}\dt = \int_0^\infty\bar\mu(t)\|\bar u(t)\|^2_{L^2(\omega)}\dt.
\end{align*}
This inequality and \eqref{E5.6} imply
\[
\mathcal{L}(u_k,\bar\mu) < \mathcal{L}(\bar u,\bar\mu) + \frac{1}{2k}\|u_k - \bar u\|^2_{L^2_{\sigma_c}(Q_\omega)}.
\]
Performing a Taylor expansion and using again Lemma \ref{L3.1} we infer for some $\vartheta_k \in [0,1]$
\begin{align*}
&\frac{1}{2}\frac{\partial^2\mathcal{L}}{\partial u^2}(\bar u + \vartheta_k(u_k - \bar u),\bar\mu)(u_k - \bar u)^2\\
& = \frac{\partial\mathcal{L}}{\partial u}(\bar u,\bar\mu)(u_k - \bar u) + \frac{1}{2}\frac{\partial^2\mathcal{L}}{\partial u^2}(\bar u + \vartheta_k(u_k - \bar u),\bar\mu)(u_k - \bar u)^2\\
& = \mathcal{L}(u_k,\bar\mu) - \mathcal{L}(\bar u,\bar\mu) < \frac{1}{2k}\|u_k - \bar u\|^2_{L^2_{\sigma_c}(Q_\omega)}.
\end{align*}
Dividing the above inequality by $\frac{\rho_k^2}{2}$ we get
\begin{equation}
\frac{\partial^2\mathcal{L}}{\partial u^2}(\bar u + \vartheta_k(u_k - \bar u),\bar\mu)v_k^2 \le \frac{1}{k}\|v_k\|^2_{L^2_{\sigma_c}(Q_\omega)} = \frac{1}{k}.
\label{E5.8}
\end{equation}
Denoting by $u_{\vartheta_k} = \bar u + \vartheta_k(u_k - \bar u)$, $y_{\vartheta_k}$ its associated state, and $\varphi_{\vartheta_k}$ the corresponding adjoint state, we get from \eqref{E5.2}
\begin{align}
&\frac{\partial^2\mathcal{L}}{\partial u^2}(\bar u + \vartheta_k(u_k - \bar u),\bar\mu)v_k^2 = \int_Q\Big[\e^{-\sigma_s t} - f''(y_{\vartheta_k})\varphi_{\vartheta_k}\Big]z_{\vartheta_k,v_k}^2\dx\dt\notag\\
& + \nu\|v_k\|^2_{L^2_{\sigma_c}(Q_\omega)} + \frac{1}{\gamma}\int_0^\infty\bar\mu(t)\|v_k(t)\|^2_{L^2(\omega)}\dt,\label{E5.9}
\end{align}
where $z_{\vartheta_k,v_k} = G'(u_{\vartheta_k})v_k$ satisfies the equation
\begin{equation}
\left\{\begin{array}{l}
\displaystyle
\frac{\partial z_{\vartheta_k,v_k}}{\partial t} + Az_{\vartheta_k,v_k} +  f'(y_{\vartheta_k})z_{\vartheta_k,v_k} = v_k\chi_\omega \mbox{ in } Q,\\ \partial_{nu_A}z_{\vartheta_k,v_k} = 0  \mbox{ on } \Sigma,\ z_{\vartheta_k,v_k}(0) = 0  \mbox{ in } \Omega.
\end{array}
\right.
\label{E5.10}
\end{equation}
Next, we prove that
\[
\frac{\partial^2\mathcal{L}}{\partial u^2}(\bar u,\bar\mu)v^2 \le \liminf_{k \to \infty} \frac{\partial^2\mathcal{L}}{\partial u^2}(\bar u + \vartheta_k(u_k - \bar u),\bar\mu)v_k^2.
\]
Then combining this and \eqref{E5.8} we conclude the proof of the second step. Looking at the expression \eqref{E5.9}, the only delicate issue in this lower limit is the integral $\int_Q f''(y_{\vartheta_k})\varphi_{\vartheta_k}z_{\vartheta_k,v_k}^2\dx\dt$. We prove that
\begin{equation}
\lim_{k \to \infty}\int_Q f''(y_{\vartheta_k})\varphi_{\vartheta_k}z_{\vartheta_k,v_k}^2\dx\dt = \int_Q f''(\bar y)\bar\varphi z_v^2\dx\dt,
\label{E5.11}
\end{equation}
where $z_v =G'(\bar u)v$. First, we  strong convergence $u_{\vartheta_k} \to \bar u$ in $L^p_{\sigma_c}(0,\infty;L^2(\omega))$ and \eqref{E2.18} yields the convergence $y_{u_{\vartheta_k}} \to \bar y$ and hence $f''(y_{u_{\vartheta_k}}) \to f''(\bar y)$ strongly in $L^\infty(Q_T)$ for every $T < \infty$.
We also get from \eqref{E2.18} and Corollary \ref{C4.1} that $\varphi_{\vartheta_k} \to \bar\varphi$ in $C_{-\lambda}(\bar Q)$ for every $\lambda < \sigma_s$, and from  \eqref{E4.4} in Theorem \ref{T4.2} that $\|\varphi_{\vartheta_k}\|_{C_{-\sigma_s}(\bar Q)}$ is bounded. From the boundedness of $\{f'(y_{\vartheta_k})\}_{k = 1}^\infty$ and $\{z_{\vartheta_k,v_k}\}_{k = 1}^\infty$ in $L^\infty(Q_T)$ and $L^2(Q_T)$, respectively, (see Lemma \ref{L3.2}) and  equation \eqref{E5.10} we infer that $\{z_{\vartheta_k,v_k}\}_{k = 1}^\infty$ is bounded in $W(0,T)$ for every $T < \infty$. Then, we can pass to the limit in the equation \eqref{E5.10} and use the compactness of the embedding $W(0,T) \subset L^2(Q_T)$ to deduce the strong convergence $z_{\vartheta_k,v_k} \to z_v$ in $L^2(Q_T)$ for every $T < \infty$. Using these convergence properties we infer that
\begin{equation}
\lim_{k \to \infty}\int_0^T\Big\| f''(y_{\vartheta_k})\varphi_{\vartheta_k}z_{\vartheta_k,v_k}^2 - f''(\bar y)\bar\varphi z_v^2\Big\|_{L^1(\Omega)}\dt = 0 \quad \forall T < \infty.
\label{E5.12}
\end{equation}

Now, we use the assumption $\sigma_s + \Lambda_f(q+2) > \sigma_c$ and define $\eta$, $\rho$, $\sigma_c$, $\lambda$, $\hat\lambda$, and $\delta$ as above \eqref{E68} . Corollary \ref{C4.2} implies the boundedness of $\|\varphi_{\vartheta_k}\|_{C_{{\frac{\delta}{2}} - 2(\sigma_s + \Lambda_f)}(\bar Q)}$. Using this we get
\begin{align*}
&\int_0^\infty e^{\frac{\delta}{4}t}
\int_\omega \left| f''(y_{\vartheta_k}) \varphi_{\vartheta_k} \right|
z_{\vartheta_k,v_k}^2 \, dx \, dt\\
&\le \|\varphi_{\vartheta_k}\|_{C_{{\frac{\delta}{2} - 2(\sigma_s + \Lambda_f)}}(\bar Q)}
\|f''(y_{\vartheta_k})\|_{L^2_{\bar\lambda}(Q)}
\|z_{\vartheta_k,v_k}\|^2_{L^4_{2\hat\lambda }(Q)},
\end{align*}
from which by \eqref{E3.5} and Lemma \ref{L4.1} we get
\begin{align*}
&\int_0^\infty e^{\frac{\delta}{4}t}
\int_\omega \left| f''(y_{\vartheta_k}) \varphi_{\vartheta_k} \right|
z_{\vartheta_k,v_k}^2 \, dx \, dt\\
&\le C_M M_{\tilde\lambda}  \|\varphi_{\vartheta_k}\|_{C_{\frac{\delta}{2} - 2(\sigma_s + \Lambda_f)}(\bar Q)} \|v_k\|^2_{L^2_{\sigma_c}(Q_\omega)} \le  C_M M_{\tilde\lambda}  \|\varphi_{\vartheta_k}\|_{C_{\frac{\delta}{2} - 2(\sigma_s + \Lambda_f)}(\bar Q)} <\infty,
\end{align*}
where we used that  $\|v_k\|_{L^2_{\sigma_c}(Q_\omega)} = 1$.  Then, applying Lemma \ref{L2.1} with $X = L^1(\Omega)$, $r = 1$, and $\lambda = -\frac{\delta}{4}$ we infer the desired convergence \eqref{E5.11}.\vspace{2mm}

 \vspace{2mm}

{\em Step III - Final contradiction.}  The facts proved in Steps I and II along with the assumption $\frac{\partial^2\mathcal{L}}{\partial u^2}(\bar u,\bar\mu)v^2 > 0$ for every $v \in C_{\bar u} \setminus \{0\}$ lead to $v = 0$ and $z_v = 0$. Therefore, looking at the relations \eqref{E5.8}, \eqref{E5.9}, and \eqref{E5.11} we obtain
\[
0 \ge \liminf_{k \to \infty}\frac{\partial^2\mathcal{L}}{\partial u^2}(\bar u + \vartheta_k(u_k - \bar u),\bar\mu)v_k^2 \ge \liminf_{k \to \infty}\nu\|v_k\|^2_{L^2_{\sigma_c}(Q_\omega)} = \nu,
\]
which contradicts the fact that $\nu > 0$.
\end{proof}

\subsection{Case II: $\K = \{v \in L^2(\omega) : \alpha \le v(x) \le \beta \text{ \rm for a.a. } x \in \omega\}$}
\label{S5.2}

In this case, the cone of critical directions is defined by
\[
C_{\bar u} = \Big\{v \in H : J'(\bar u)v = 0 \text{ and } v(x,t) \left\{\begin{array}{cl} \ge 0&\text{if } \bar u(x,t) = \alpha\\\le 0&\text{if } \bar u(x,t) = \beta\end{array}\right.\Big\}.
\]
Analogously to Theorem \ref{T5.1} we have the following result.
\begin{theorem}
If $\bar u$ is a local minimizer of \Pb, then $J''(\bar u)v^2 \ge 0$ for all $v \in C_{\bar u}$.
\label{T5.3}
\end{theorem}

\begin{proof}
Since $\bar u$ is a local minimizer of \Pb, there exists $\varepsilon > 0$ such that $J(\bar u) \le J(u)$ for all $u \in \Uad \cap B_\varepsilon(\bar u)$. Given $v \in C_{\bar u}$ we define for every integer $k \ge 1$ the function $v_k$ by
\[
v_k(x,t)=\left\{\begin{array}{cl}0&\text{if } \alpha < \bar u(x,t) < \alpha + \frac{1}{k}\text{ or } \beta - \frac{1}{k} < \bar u(x,t) < \beta,\\\proj_{[-k,+k]}(v(x,t))&\text{otherwise.}\end{array}\right.
\]
It is obvious that $\{v_k\}_{k = 1}^\infty \subset L^\infty(Q_\omega) \subset H$ and $v_k \to v$ in $H$ as $k \to \infty$. Furthermore, if we set $\rho_k = \min\{\frac{1}{k^2},\frac{\beta - \alpha}{k},\frac{\varepsilon}{\|v\|_H}\}$, then $\bar u + \rho v_k \in \Uad \cap B_\varepsilon(\bar u)$ for every $\rho \in (0,\rho_k)$. In view of \eqref{E4.22}, it is straightforward to check that the condition $J'(\bar u)v = 0$ in the definition of $C_{\bar u}$ is equivalent to $(\bar\varphi + \nu\bar u)(x,t)v(x,t) = 0$ for almost all $(x,t) \in Q_\omega$. Using this fact, it is immediate that $J'(\bar u)v_k = 0$ for every $k$. Then, performing a Taylor expansion, we get for every $\rho \in (0,\rho_k)$
\[
0 \le J(\bar u + \rho v_k) - J(\bar u) = \rho J'(\bar u)v_k + \frac{\rho^2}{2}J''(\bar u + \theta_{\rho,k}\rho v_k)v_k^2 = \frac{\rho^2}{2}J''(\bar u + \theta_{\rho,k}\rho v_k)v_k^2.
\]
Dividing by $\frac{\rho^2}{2}$ we obtain  $J''(\bar u + \theta_{\rho,k}\rho v_k)v_k^2 \ge 0$. Since $J:L^p_{\lambda_c}(0,\infty;L^2(\omega)) \longrightarrow \mathbb{R}$ is of class $C^2$ and  $\bar u + \theta_{\rho,k}\rho v_k \to \bar u$ in $L^p_{\lambda_c}(0,\infty;L^2(\omega))$ as $\rho \to 0$, the inequality $J''(\bar u)v_k^2 \ge 0$ follows. Moreover, since $v_k \to v$ in $H$ we infer from Remark \ref{R4.1}-ii) that $J''(\bar u)v^2 = \lim_{k \to \infty}J''(\bar u)v_k^2 \ge 0$.
\end{proof}

To prove the second-order sufficient optimality conditions, we proceed as in \S\ref{S5.1}, redefining $C_{\bar u}$ with $H$ replaced by $L^2_{\sigma_c}(Q_\omega)$ and using the assumptions on $y_d$ and $\sigma_c$ as in Theorem \ref{T5.2}.

\begin{theorem}
Assume that $\sigma_c < \sigma_s + \Lambda_f(q + 2)$ and $y_d \in L^p_{-2\Lambda_f}(0,\infty;L^2(\Omega))$, that $\bar u \in \Uad$ satisfies the first-order optimality conditions, and that $J''(\bar u)v^2 > 0$ for every $v \in C_{\bar u} \setminus \{0\}$. Then \eqref{E5.5} holds.
\label{T5.4}
\end{theorem}

For the proof of this result for finite horizon case  the reader is referred to \cite{Bonnans98},\cite{Casas-Troltzsch2012}. The difficulties due to the infinite horizon can be overcome by following the arguments used in the proof of Theorem \ref{T5.2}.

\section{Approximation by finite horizon problems}
\label{S6}

In this section, we consider the approximation of \Pb by finite horizon optimal control problems and provide error estimates for these approximations. For every $0 < T < \infty$ we consider the  control problem
\[
\PbT \quad \min_{u \in \UTad} J_T(u),
\]
where $\UTad = \{u \in L^\infty(0,T;L^2(\omega)) : u(t) \in \K \text{ for a.a. } t \in (0,T)\}$,
\[
J_T(u) = \frac{1}{2}\int_0^T[\e^{-\sigma_s t}\|y_{T,u}(t) - y_d(t)\|_{L^2(\Omega)}^2 + \nu\e^{-\sigma_ct}\|u(t)\|_{L^2(\omega)}^2]\dt,
\]
where $y_{T,u}$ denotes the solution of the equation
\begin{equation}
\left\{
\begin{array}{l}
\displaystyle
\frac{\partial y}{\partial t} + Ay + f(y) = g + u\chi_\omega \mbox{ in } Q_T,\\ \partial_{n_A}y = 0  \mbox{ on } \Sigma_T,\ y(0) = y_0  \mbox{ in } \Omega.
\end{array}
\right.
\label{E6.1}
\end{equation}

\textbf{Notation:} In the sequel, $u_0$ denotes a control of $\Uad$. Then, for every control $u \in \UTad$ we consider an extension $\hat u \in \Uad$ by setting $\hat u(x,t) = u_0(x,t)$ if $t > T$.\vspace{2mm}

The next two theorems show that the control problems \PbT approximate \Pb.

\begin{theorem}
For all $T > 0$  problem \PbT has at least one solution $u_T$. The extensions $\{\hat u_T\}_{T > 0}$ of any family of solutions are bounded in $L^2_{\sigma_c}(Q_\omega)$. Every weak limit $\bar u$ in $L^2_{\sigma_c}(Q_\omega)$ of a sequence $\{\hat u_{T_k}\}_{k = 1}^\infty$ with $T_k \to \infty$ as $k \to \infty$ is a solution of \Pb. Moreover,  strong convergence $\hat u_{T_k} \to \bar u$ in $L^p_{\lambda_c}(0,\infty;L^2(\omega))$ holds.
\label{T6.1}
\end{theorem}

\begin{proof}
The existence of a solution $u_T$ for \PbT is standard. Given a family of solutions $\{u_T\}_{T > 0}$, we have that their extensions satisfy $\{\hat u_T\}_{T > 0} \subset \Uad$. This family is bounded in $L^\infty(0,\infty;L^2(\omega))$. This proves the existence of weak limits $\hat u_{T_k} \stackrel{*}{\rightharpoonup} \bar u$ in $L^\infty(0,\infty;L^2(\omega))$ as $T_k \to \infty$. For any of these limit points we also have that $\hat u_{T_k} \rightharpoonup \bar u$ in $U$. Since $\Uad$ is convex and closed in $U$ we have that $\bar u \in \Uad$. From \eqref{E2.18} we infer that $y_{\hat u_{T_k}} \to \bar y = y_{\bar u}$ in $Y_\lambda$ for every $\lambda > -2\Lambda_f$. Let  $\tilde u$ be a solution of \Pb and fix $\lambda \in (-2\Lambda_f,\sigma_s)$. Using the mentioned convergences and the optimality of every $u_{T_k}$ we obtain
\begin{align*}
&J(\bar u) \le \liminf_{k \to \infty}J(\hat u_{T_k}) \le \limsup_{k \to \infty}J(\hat u_{T_k}) \le \limsup_{k \to \infty}J_{T_k}(u_{T_k})\\
& + \limsup_{k \to \infty}\int_{T_k}^\infty[\e^{-\sigma_st}\|y_{\hat u_{T_k}} - y_d\|^2_{L^2(\Omega)} + \nu \e^{-\sigma_ct}\|u_0\|^2_{L^2(\omega)}]\dt \le \limsup_{k \to \infty}J_{T_k}(\tilde u)\\
&+  \limsup_{k \to \infty}\Big(2\e^{-(\sigma_s - \lambda)T_k}\|y_{\hat u_{T_k}}\|^2_{L^2_{\lambda}(Q)} + \int_{T_k}^\infty[[2\e^{-\sigma_st}\|y_d\|^2_{L^2(\Omega)} + \nu \e^{\sigma_ct}\|u_0\|^2_{L^2(\omega)}]\dt\Big)\\
& = J(\tilde u) = \inf\Pb.
\end{align*}
This implies that $\bar u$ is a solution of \Pb and $\lim_{k \to \infty}J(\hat u_{T_k})  = J(\bar u)$. From this convergence and the fact that $y_{\hat u_{T_k}} \to \bar y$ in $Y_{\sigma_s}$ we deduce $\|\hat u_{T_k}\|_{L^2_{\sigma_c}(Q_\omega)} \to \|\bar u\|_{L^2_{\sigma_c}(Q_\omega)}$. Hence, we get that $\hat u_{T_k} \to \bar u$ strongly in $L^2_{\sigma_c}(Q_\omega)$. As we proved at the beginning of the proof of Theorem \ref{T5.2}, this implies $\lim_{k \to \infty} \|\hat u_{T_k} - \bar u\|_{L^p_{\lambda_c}(0,\infty;L^2(\omega))} = 0$.
\end{proof}

\begin{theorem}
Let $\bar u$ be a strict local minimizer of \Pb. Then, there exist $T_0 \in (0,\infty)$ and a family $\{u_T\}_{T > T_0}$ of local minimizers to \PbT such that the convergence $\hat u_T \to \bar u$ in $L^p_{\sigma_c}(0,\infty;L^2(\omega))$ holds as $T \to \infty$.
\label{T6.2}
\end{theorem}

\begin{proof}
Since $\bar u$ is a strict local minimizer of \Pb, there exists $\rho > 0$ such that $J(\bar u) < J(u)$ for every $u \in \Uad \cap B_\rho(\bar u)$ with $u \neq \bar u$, where $B_\rho(\bar u)$ is the closed ball in $L^2_{\sigma_c}(Q_\omega)$ centered at $\bar u$ with radius $\rho > 0$. We consider the control problems
\[
\Pbr \quad \min_{u \in B_\rho(\bar u) \cap \Uad} J(u) \quad \text{ and }\quad \PbTr \quad \min_{u \in B_{T,\rho}(\bar u) \cap \UTad} J_T(u),
\]
where $B_{T,\rho}(\bar u) = \{u \in L^2_{\sigma_c}(Q_{T,\omega}) : \|u - \bar u\|_{L^2_{\sigma_c}(Q_{T,\omega})} \le \rho\}$. Obviously $\bar u$ is the unique solution of \Pbr. Existence of a solution $u_T$ of \PbTr is straightforward. Then, arguing as in the proof of Theorem \ref{T6.1} and using the uniqueness of the solution of \Pbr, we deduce the convergence $\hat u_T \to \bar u$ in $L^2_{\sigma_c}(Q_\omega)$ as $T \to \infty$. This implies the existence of $T_0 > 0$ such that $\|u_T - \bar u\|_{L^2_{\sigma_c}(Q_{T,\omega})} \le \|\hat u_T - \bar u\|_{L^2_{\sigma_c}(Q_\omega)} < \rho$ for all $T > T_0$. Hence, $u_T$ is also a local minimizer of \PbT for $T > T_0$. The strong convergence $\hat u_T \to \bar u$ in $L^p_{\sigma_c}(0,\infty;L^2(\omega))$ follows from the convergence in $L^2_{\sigma_c}(Q_\omega)$ and the fact that $\|\hat u_T\|_{L^\infty(0,\infty;L^2(\omega))} \le \gamma$ for every $T > 0$.
\end{proof}

\begin{theorem}
Assume that $\Uad$ is defined by \eqref{E1.2} or \eqref{E1.3}, $\sigma_c < \sigma_s + \Lambda_f(q + 2)$, and $y_d \in L^p_{-2\Lambda_f}(0,\infty;L^2(\Omega))$. Let $\bar u$ be a local minimizer of \Pb satisfying the second-order sufficient optimality condition given in Theorems \ref{T5.2} or \ref{T5.4}. Let $\{u_T\}_{T > T_0}$ be a family of local minimizers of problems \PbT such that $\hat u_T \to \bar u$ strongly in $L^p_{\lambda_c}(0,\infty;L^2(\omega))$. Then, there exists $T_0^* \in (T_0,\infty)$ such that $J_T(u_T) \le J_T(\bar u)$ for every for every $T \ge T^*_0$.
\label{T6.3}
\end{theorem}

The proof of this theorem is analogous to that of \cite[Theorem 4.3]{Casas-Kunisch2023C} with obvious modifications.

\begin{theorem}
Under the assumptions and notation of Theorem \ref{T6.3} and supposing that $\sigma_c \le \lambda_c$ and $g \in L^p_{\lambda_c}(Q) \cap L^2_{\sigma_s}(Q)$, there exist a constant $C$ depending only on $\kappa$, $\sigma_s$, and $\Lambda_f$, and $T^* \in [T_0,\infty)$ such that for every $T > T^*$
\begin{equation}
\|u_T - \bar u\|_{L^2_{\sigma_c}(Q_{T,\omega})}
\le C\Big(\e^{-\frac{\sigma_s}{2}T}\big(\|y_T(T)\|_{L^2(\Omega)} + 1\big) {+} \|y_d\|_{L^2_{\sigma_s}(Q^T)} {+} \|g\|_{L^2_{\sigma_s}(Q^T)}\Big),\label{E6.2}
\end{equation}
where $Q^T = \Omega \times (T,\infty)$ and $y_T = y_{T,u_T}$. Moreover, we have
\begin{align}
&\|y_T - \bar y\|_{L^2_{\sigma_s}(0,T;H^1(\Omega))} + \|y_T - \bar y\|_{L^\infty_{\sigma_s}(0,T;L^2(\Omega))} \le C_{3,M}\|u_T - \bar u\|_{L^2_{\sigma_c}(Q_{T,\omega})},\label{E6.3}\\
&\|y_T - \bar y\|_{C_{\sigma_s}(\bar Q_T)} \le C_{2,M}(2\gamma_{ad})^{\frac{p - 2}{p}}\|u_T - \bar u\|^{\frac{2}{p}}_{L^2_{\sigma_c}(Q_{T,\omega})},\label{E6.4}
\end{align}
where $\gamma_{ad} = \sup\{\|v\|_{L^2(\Omega)} : v \in \K\}$, and $C_{2,M}$ and $C_{3,M}$ were introduced in \eqref{E3.12}.
\label{T6.4}
\end{theorem}

\begin{proof}
Without loss of generality, we take $u_0 = \bar u$ to extend every $u_T$. Let us take $\kappa$ and $\varepsilon$ as given in \eqref{E5.5}. The convergence $\hat u_T \to \bar u$ in $L^p_{\lambda_c}(0,\infty;L^2(\omega))$ and the boundedness of $\{\hat u_T\}_{T > T_0}$ imply that $\hat u_T \to \bar u$ in $L^2_{\sigma_c}(Q_\omega)$. Hence, we deduce the existence of $T^* \in [T^*_0,\infty)$ such that $\|\hat u_T - \bar u\|_{L^2_{\sigma_c}(Q_\omega)} \le \varepsilon$, where $T^*_0$ was given in Theorem \ref{T6.3}. Therefore, for $T \ge T^*$ we get with \eqref{E5.5}, by the choice of $u_0$, and the optimality of $u_T$
\begin{align*}
&\frac{\kappa}{2}\|u_T - \bar u\|^2_{L^2_{\sigma_c}(Q_{T,\omega})} = \frac{\kappa}{2}\|\hat u_T - \bar u\|^2_{L^2_{\sigma_c}(Q_\omega)}\le J(\hat u_T) - J(\bar u) \le J_T(u_T) - J_T(\bar u) \notag\\
& + \frac{1}{2}\int_T^\infty\e^{{-\sigma_st}}\|y_{\hat u_T}(t) - y_d(t)\|^2_{L^2(\Omega)}\dt \le \frac{1}{2}\int_T^\infty\e^{{-\sigma_st}}\|y_{\hat u_T}(t) - y_d(t)\|^2_{L^2(\Omega)}\dt.
\end{align*}
This yields
\begin{equation}
\|u_T - \bar u\|_{L^2_{\sigma_c}(Q_{T,\omega})} \le \frac{1}{\sqrt{\kappa}}\Big(\|y_{\hat u_T}\|_{L^2_{\sigma_s}(Q^T)} + \|y_d\|_{L^2_{\sigma_s}(Q^T)}\Big).
\label{E6.5}
\end{equation}
To estimate $\|y_{\hat u_T}\|_{L^2_{\sigma_s}(\Omega \times (T, \infty))}$ we observe that $y_{\hat u_T}$ satisfies
\[
\left\{
\begin{array}{l}
\displaystyle
\frac{\partial y_{\hat u_T}}{\partial t} + Ay_{\hat u_T} + f(y_{\hat u_T}) = g + \bar u\chi_\omega\ \mbox{ in } Q^T,\\[0.5ex] \partial_{n_A}y_{\hat u_T} = 0 \ \mbox{ on } \Gamma \times (T,\infty),\   y_{\hat u_T}(T) = y_T(T) \ \mbox{ in } \Omega.
\end{array}
\right.
\]
Testing this equation with $\e^{-\sigma_st}y_{\hat u_T}$, integrating between $T$ and $t > T$,  and arguing as in \eqref{E2.9} we infer
\begin{align*}
&\frac{1}{2}\e^{-\sigma_st}\|y_{\hat u_T}(t)\|^2_{L^2(\Omega)} + \Big(\frac{\sigma_s}{2} + \Lambda_f\Big)\int_T^t\e^{-\sigma_s t}\|y_{\hat u_T}(s)\|^2_{L^2(\Omega)}\ds\\
&+ \Lambda_A\int_T^t\e^{-\sigma_st}\|\nabla y_{\hat u_T}(s)\|^2_{L^2(\Omega)^n}\ds\\
&\le \frac{\e^{-\sigma_sT}}{2}\|y_T(T)\|^2_{L^2(\Omega)} + \|g + \bar u\chi_\omega\|_{L^2_{\sigma_s}(Q^T)}\|y_{\hat u_T}\|_{L^2_{\sigma_s}(Q^T)}\\
&  \le \frac{\e^{-\sigma_sT}}{2}\|y_T(T)\|^2_{L^2(\Omega)} + \frac{1}{\sigma_s + 2\Lambda_f}\|g + \bar u\chi_\omega\|^2_{L^2_{\sigma_s}(Q^T)} + \frac{1}{2}\Big(\frac{\sigma_s}{2} + \Lambda_f\Big)\|y_{\hat u_T}\|^2_{L^2_{\sigma_s}(Q^T)}.
\end{align*}
From this estimate we get
\begin{align*}
&\|y_{\hat u_T}\|_{L^2_{\sigma_s}(Q^T)} \le \frac{1}{\sqrt{\sigma_s + 2\Lambda_f}}\e^{-\frac{\sigma_s}{2}T}\|y_T(T)\|_{L^2(\Omega)}\\
&+ \frac{\sqrt{2}}{\sigma_s + 2\Lambda_f}\Big(\|g\|_{L^2_{\sigma_s}(Q^T)} + \gamma_{ad}\Big(\int_T^\infty\e^{-\sigma_st}\dt\Big)^{\frac{1}{2}}\Big)\\
&= \frac{1}{\sqrt{\sigma_s + 2\Lambda_f}}\e^{-\frac{\sigma_s}{2}T}\|y_T(T)\|_{L^2(\Omega)} + \frac{\sqrt{2}}{\sigma_s + 2\Lambda_f}\Big(\|g\|_{L^2_{\sigma_s}(Q^T)} + \frac{\gamma_{ad}}{\sqrt{\sigma_s}}\e^{-\frac{\sigma_s}{2}T}\Big).
\end{align*}
Inserting this estimate in \eqref{E6.5} we obtain \eqref{E6.2}.

Let us prove \eqref{E6.3}. Using \eqref{E3.12} we infer
\begin{align*}
&\|y_T - \bar y\|_{L^2_{\sigma_s}(0,T;H^1(\Omega))}\! +\! \|y_T - \bar y\|_{L^\infty_{\sigma_s}(0,T;L^2(\Omega))} \le \|y_{\hat u_T} - \bar y\|_{Y_{\sigma_s}}\\
&\le C_{3,M}\|\hat u_T - \bar u\|_{L^2_{\sigma_s}(Q)} =  C_{3,M}\|\hat u_T - \bar u\|_{L^2_{\sigma_s}(Q_T)} \le C_{3,M}\|\hat u_T - \bar u\|_{L^2_{\sigma_c}(Q_T)}.
\end{align*}
Finally, we prove \eqref{E6.4}. Using again \eqref{E3.12} we get
\begin{align*}
&\|y_T - \bar y\|_{C_{\sigma_s}(\bar Q_T)} \le \|y_{\hat u_T} - \bar y\|_{C_{\sigma_s}(\bar Q)} \le C_{2,M}\|\hat u_T - \bar u\|_{L^p_{\sigma_s}(0,\infty;L^2(\omega))}\\
&= C_{2,M}\|\hat u_T - \bar u\|_{L^p_{\sigma_s}(0,T;L^2(\omega))} \le C_{2,M}(2\gamma_{ad})^{\frac{p - 2}{p}}\|u_T - \bar u\|^{\frac{2}{p}}_{L^2_{\sigma_s}(Q_{T,\omega})}\\
&  \le C_{2,M}(2\gamma_{ad})^{\frac{p - 2}{p}}\|u_T - \bar u\|^{\frac{2}{p}}_{L^2_{\sigma_c}(Q_{T,\omega})}.
\end{align*}
\end{proof}

\begin{remark}
In the case where $0 \in \K$, the estimate \eqref{E6.2}  can be slightly simplified:
\[
\|u_T - \bar u\|_{L^2_{\sigma_c}(Q_{T,\omega})}
\le C\Big(\e^{-\frac{\sigma_s}{2}T}\|y_T(T)\|_{L^2(\Omega)} {+} \|y_d\|_{L^2_{\sigma_s}(Q^T)} {+} \|g\|_{L^2_{\sigma_s}(Q^T)}\Big).
\]
To prove this we extend $u_T$ by zero to $(0,\infty)$ rather than $\bar u$.
\end{remark}

\begin{remark}
The results of this paper can easily be extended to the case where  the tracking term is of the form $\frac{1}{2}\int_0^\infty\e^{-{\sigma_s}t}{\|y_u(t) - y_d(t)\|^2_{L^2(\Omega)}\chi_{\omega_{obs}}(t)}\,dt$, for some measurable set $\omega_{obs} \subset \Omega$. Indeed, this requires  only minor changes including modifying the right hand side of the first equation in \eqref{E4.2} to  be $\e^{-\lambda t}h \chi_{\omega_{obs}}$ and  replacing the first term on the right hand side of \eqref{E4.13} by $\int_Q[\e^{-\sigma_s t}\chi_{\omega_{obs}} - \varphi_u f''(y_u)]z_{v_1}z_{v_2}\dx\dt$, and analogously in \eqref{E5.2}.
\end{remark}

\section{Conclusions}
\label{S7}

Let us comment on the selection of the discount factors $\sigma_s$ and $\sigma_c$ to carry out the analysis of the control problem. The space of admissible controls is $L^\infty(0,\infty;L^2(\omega))$. On this space we have two functionals:
\[
F_s(u) = \frac{1}{2}\int_0^\infty\e^{-\sigma_st}\|y_u(t) - y_d(t)\|^2_{L^2(\Omega)}\dt\ \text{ and } \ F_c(u) = \frac{\nu}{2}\int_0^\infty\e^{-\sigma_ct}\|u(t)|^2_{L^2(\omega)}\dt.
\]
In the analysis of the control problem \Pb there are three important issues in this paper:\vspace{2mm}

{\em I - Well-posedness of the control problem: $\sigma_s > -2\Lambda_f$ and $\sigma_c > 0$.} Using Remark \ref{R2.2}, the existence of an optimal control follows under these selection of $\sigma_s$ and $\sigma_c$.\vspace{2mm}

{\em II - Necessary optimality conditions: $\sigma_s > -2(q+3)\Lambda_f$ and $\sigma_c > 0$.} The assumption on $\sigma_s$ is necessary to prove the differentiability of the control-to-state mapping, from where we infer the differentiability of the  cost functional $J$. These conditions on $\sigma_s$ and $\sigma_c$ are sufficient to obtain the first- and second-order necessary optimality conditions for the class of nonlinearities under consideration.\vspace{2mm}

{\em III - Second-order sufficient optimality conditions: $\sigma_s > -2(q+3)\Lambda_f$ and $0 < \sigma_c < \sigma_s + \Lambda_f(q + 2)$}. Under these selections on $\sigma_s$ and $\sigma_c$ second-order sufficient optimality conditions are proved.

\section*{Funding}

The first author was supported by MICIU/AEI/10.13039/501100011033/ under research project PID2023-147610NB-I00.


\end{document}